\documentclass[preprint,12pt]{imsart}
\RequirePackage{amsthm,amsmath,amsfonts,amssymb}
\RequirePackage[numbers]{natbib}

\RequirePackage[colorlinks,citecolor=blue,urlcolor=blue]{hyperref}
\RequirePackage{graphicx}
\usepackage{booktabs,dcolumn,caption}
\newcolumntype{d}[1]{D{.}{.}{#1}}
\startlocaldefs
\numberwithin{equation}{section}
\theoremstyle{plain}
\newtheorem{thm}{Theorem}[section]
\newtheorem{lem}{Lemma}[section]
\newtheorem{cor}{Corollary}[section]

\newtheorem{assm}{Assumption}
\newtheorem{example}{Example}

\newtheorem{rem}{Remark}[section]
\endlocaldefs

\usepackage{a4wide}

\allowdisplaybreaks

\begin{document}

\begin{frontmatter}

  \title{On singular values of  data matrices with  general independent columns}
  
  \runtitle{Singular values of data matrices with general independent columns }

  \begin{aug}
    \author{\fnms{Tianxing}  \snm{Mei}\ead[label=e1]{meitx@hku.hk}},
    \author{\fnms{Chen} \snm{Wang}\ead[label=e2]{stacw@hku.hk}}
    \and
    \author{\fnms{Jianfeng} \snm{Yao~}\ead[label=e3]{jeffyao@hku.hk}}
    
%    \runauthor{T. Mei, C. Wang and    J. Yao}

    \address{
      Department of Statistics and Actuarial Science \\
      The University of Hong Kong \\
      \printead{e1,e2,e3}}
  \end{aug}

 \begin{abstract}
    In this paper, we analyse singular values of a large $p\times n$
    data matrix  $\mathbf{X}_n= (\mathbf{x}_{n1},\ldots,\mathbf{x}_{nn})$ where the column
    $\mathbf{x}_{nj}$'s are independent $p$-dimensional vectors, possibly with different
    distributions.
    Such data matrices are common in high-dimensional
    statistics.
    Under a key assumption that the covariance matrices
    $\mathbf{\Sigma}_{nj}=\text{Cov}(\mathbf{x}_{nj})$ can be asymptotically simultaneously
    diagonalizable, and appropriate convergence of their spectra, we
    establish a limiting distribution for the singular values of
    $\mathbf{X}_n$ when both dimension $p$ and $n$ grow to infinity in
    a comparable magnitude.
    The matrix model goes beyond and includes many  existing works on different types of sample covariance
    matrices, including the weighted sample covariance matrix, the Gram
    matrix model and the sample covariance matrix of linear times series models.
    Furthermore, we develop two 
    applications of our  general approach.
    First, we obtain the existence and uniqueness of a new
    limiting spectral
    distribution of realized covariance matrices for a 
    multi-dimensional diffusion process with anisotropic time-varying
    co-volatility processes.
    Secondly,  we derive the limiting spectral
    distribution for singular values of the data matrix for a recent 
    matrix-valued auto-regressive model.
    Finally, for a generalized finite mixture model, the limiting spectral distribution for singular values of the data matrix is obtained.
  \end{abstract}

  \begin{keyword}[class=MSC2020]
    \kwd[Primary ]{62H10}
    \kwd[; secondary ]{60B20}
  \end{keyword}

  \begin{keyword}
    \kwd{Large data matrix}
    \kwd{Large sample covariance matrices}
    \kwd{Singular value  distribution}
    \kwd{Eigenvalue distribution}
    \kwd{Separable covariance matrix}
    \kwd{Realized covariance matrix}
    \kwd{Matrix-valued autoregressive model}
  \end{keyword}
\end{frontmatter}

%frontmatter-------------------------------------------------------
\def\lb{\label}
\def\ct{\cite}
\newcommand{\rf}[1]{(\ref{#1})}
%MATH -----------------------------------------------------------
\newcommand{\scr}[1]{\mathscr{#1}}
\newcommand{\norm}[2]{\left\|{#1}\right\|_{#2}}
\newcommand{\abs}[1]{\left\vert#1\right\vert}
\newcommand{\set}[1]{\left\{#1\right\}}
\def\R{\mathbb R}
\def\Q{\mathbb{Q}}
\def\A{\mathcal{A}}
\def\E{\mathbb{E}}
\def\D{\scr{D}}
\def\F{\scr F}
\def\K{\scr K}
\def\B{\scr B}
\def\N{\mathbb N}
\renewcommand{\P}{\mathbb P}
\def\R{\mathbb R}
\def\Z{\mathbb Z}
\def\d{{\rm d}}
\def\bfE{\mbox{\boldmath $E$}}
\def\bfP{\mbox{\boldmath $P$}}
\def\bfQ{\mbox{\boldmath $Q$}}
\newcommand{\one}[1]{\mathbbm{1}_{#1}}
\newcommand{\ones}[2]{\mathbbm{1}_{#1}(#2)} %\ones{[X_t=i]}{\omg}
%Newcommand-------------------------------------------------
\def\bg{\begin}
\def\ed{\end}
\def\be{\bg{equation}}\def\de{\ed{equation}}
\def\bgar{\bg{eqnarray}}\def\edar{\end{eqnarray}}
\newcommand{\nnb}{\nonumber}
\def\l{\left}\def\r{\right}
\def\beqlb{\begin{eqnarray}}\def\eeqlb{\end{eqnarray}}
\def\beqnn{\begin{eqnarray*}}\def\eeqnn{\end{eqnarray*}}
%Greece-----------------------------------------------------------
\def\alp{\alpha}
\def\bt{\beta}
\def\gm{\gamma}
\def\Gm{\Gamma}
\def\dlt{\delta}
\def\Dlt{\Delta}
\def\eps{\varepsilon}
\def\e{\eps}
\def\tht{\theta}
\def\Tht{\Theta}
\def\kp{\kappa}
\def\lmd{\lambda}
\def\Lmd{\Lambda}
\def\vro{\varrho}
\def\sgm{\sigma}
\def\Sgm{\Sigma}
\def\vph{\varphi}
\def\omg{\omega}
\def\Omg{\Omega}
\def\zt{\zeta}
%Usual--------------------------------------------------------
\def\fa{\forall}
\def\emp{\emptyset}
\def\ex{\exists}
\def\bs{\backslash}
\def\fr{\frac}
\def\nbl{\nabla}
\def\pat{\partial}
\def\ift{\infty}
%Contain-----------------------------------------------------
\def\l{\left}
\def\r{\right}
\def\({\l(}
\def\){\r)}
\def\bcap{\bigcap}
\def\bcup{\bigcup}
\def\lg{\l\langle}
\def\rg{\r\rangle}
%Arrow-------------------------------------------------------
\def\lar{\leftarrow}
\def\Lar{\Leftarrow}
\def\rar{\rightarrow}
\def\Rar{\Rightarrow}
\def\lla{\longleftarrow}
\def\Lla{\Longleftarrow}
\def\To{\Longrightarrow}
\def\lra{\leftrightarrow}
\def\Lra{\Leftrightarrow}
\def\llra{\longleftrightarrow}
\def\Llra{\Longleftrightarrow}
%Skip-----------------------------------------------------------
\def\q{\quad}
\def\mpb{\vskip6pt}
%Names--------------------------------------------------------
\def\gap{\mathop{\text{gap}}}
\def\var{\mathop{\text{Var}}}
\def\cov{\mathop{\text{Cov}}}
\def\diag{\mathop{\text{diag}}}
\def\ess{\mathop{\text{ess}}}
\def\hess{\mathop{\text{Hess}}}
\def\ric{\mathop{\text{Ric}}}
\def\lip{\mathop{\text{Lip}}}
\def\dist{\mathop{\text{dist}}}
\def\supp{\mathop{\text{supp}}}
\def\sep{\mathop{\text{sep}}}
%\def\cap{\text {\rm cap}}
%Limit----------------------------------------------------------- 
%\def\e{{\mbox{\rm e}}}
\def\dsum{\displaystyle\sum}
\def\prodd{\displaystyle\prod}
\def\dsup{\displaystyle\sup}
\def\dinf{\displaystyle\inf}
\def\dlim{\displaystyle\lim}
\def\dlimsup{\displaystyle\limsup}
\def\dliminf{\displaystyle\liminf}
%Integral-----------------------------------------------------------
\newcommand{\jff}[2]{\int_{#1}^{#2}}
\newcommand{\jf}[1]{\int_{#1}}
\def\d{{\mbox{\rm d}}}
%Hats---------------------------------------------------------------
\renewcommand{\hat}{\widehat}
\renewcommand{\bar}{\overline}
\def\tld{\widetilde}
\def\chk{\check} %'\v' in math environment
\newcommand{\vv}[1]{\boldsymbol{#1}}
%Colors-------------------------------------------------------------
\newcommand{\red}[1]{{\color{red} #1}}
\newcommand{\yel}[1]{{\color{yellow} #1}}
\newcommand{\blu}[1]{{\color{blue} #1}}
\newcommand{\grn}[1]{{\color{green} #1}}

\def\red#1{}

\section{Introduction}
Large  data matrices are now  common in many areas of research 
such as genomic data analysis,   on-line recommendation systems or
portfolio managements. Consider a $p\times n$ data matrix 
$\vv{X}_n=(\vv{x}_{n1},\ldots,\vv{x}_{nn})$ where the columns
$\vv{x}_{nj}$'s are  in $\mathbb{R}^p$ or $\mathbb{C}^p$.
Singular values of the normalized data matrix  $\vv{X}_n/\sqrt n$, or
equivalently, the eigenvalues of its square 
\begin{equation}\label{eq:1.1}
\vv{S}_n=\frac{1}{n}\vv{X_n}\vv{X}_n^*=\fr{1}{n}\sum_{j=1}^n \vv{x}_{nj}\vv{x}^*_{nj},
\end{equation}
have primary importance for analysis of the data matrix
$\vv{X}_n$. (Here $\vv{x}^*$ represents the complex conjugate.)
For example, the singular value decomposition of  $\vv{X}_n$ can provide an efficient
data reduction if its  ``signal''  singular values are well separated
from the background noise singular values.

Recall that  empirical spectral
distribution (ESD) of  a  Hermitian  matrix $\vv{B}_n$, $\mu_{\vv{B}_n}$,
is the normalized counting measure of
its  real-valued eigenvalues.
% namely \begin{equation}\label{eq:1.2}  \mu_{\vv{B_n}} =\frac{1}{p}\sum_{i=1}^p  \delta_{ \lambda_{n,i}}.\end{equation}
If as $n\to\infty$, $\mu_{\vv{B_n}}$ converges weakly to a probability measure $\mu$, then we call $\mu$
the limiting spectral distribution (LSD) of the matrix sequence $\{\vv{B}_n\}$.
In this 
paper, we study the existence and uniqueness of LSDs for the (squared) data matrix  $\vv{S}_n$
in the following high-dimensional setting:
\begin{equation}\label{eq:1.3}
  n\to\infty, \quad p=p(n)\to\infty  \text{~~ such that ~~} \frac{p}{n}\to c\in(0,\infty).
\end{equation}

%In this paper we consider a situation where the columns $\mathbf{x}_{nj}$s are {\em independent}, but possibly with different distributions.

The study of spectrum of the matrix $\vv{S}_n$ has a long history,
along with  increasingly complex structures for  the joint distribution
of the column vectors $\{\vv{x}_{nj}\}$, see
\cite{BS10,PA14,Yao15} for a detailed review.  Here we recall a few
results relevant to the present paper.

\begin{enumerate}
\item[(a)\ ]
  {\em The Mar\v{c}enko-Pastur law.} 
  One very first result 
  is established in the seminal paper
  \red{Marc\v{c}enko and Pastur (1961)}\cite{MP} for the matrix
  \[ \vv{S}_n = \frac1n \sum_{j=1}^n \tau_j \vv{z}_j\vv{z}_j^*
  \]
  where %$A_n$ is a deterministic nonnegative definite matrix,
  $\{\tau_j\}$ is a sequence of non-negative numbers, and  $\{\vv{z}, \vv{z}_j\}$  a
  sequence of 
  i.i.d.  $p$-dimensional vectors, centred and isotropic in the sense
  that
  $\E(\vv{z})=0$ and $\cov(\vv{z})=\vv{I}_p$.
  Under appropriate conditions on the forth
  moment of $\vv{z}_j$, %spectral convergence of the sequences $\{A_n\}$, 

  and ergodicity of the numerical sequence $\{\tau_j\}$,
  the celebrated Mar\v{c}enko-Pastur law is established in \cite{MP} as
  the LSD of $\vv{S}_n$ under the high-dimensional limit \eqref{eq:1.3}. 
  Note that this model is a particular case of \eqref{eq:1.1} with
  $\vv{x}_{nj}=\sqrt{\tau_j} \vv{z}_j$.

  This model has been  recently extended  in \cite{PajorPastur} where the
  same LSD is established for a wider family of so-called ``good
  vectors'' $ \vv{z}$.  
\item[(b)\ ] {\em Sample covariance matrices}.
  Motivated by statistical multivariate analysis, 
  \cite{BS94,BS95} analysed a class of {sample covariance matrices} of
  the form
  \[  \vv{S}_n = \frac1n \sum_{j=1}^n \vv{\Sigma}_p^{1/2} \vv{z}_j\vv{z}_j^*
  \vv{\Sigma}_p^{1/2},
  \]
  where
  $\{\vv{\Sigma}_p\}$ is a sequence of non-negative definite matrices,
  $\{\vv{z}, \vv{z}_j\}$ an i.i.d. sequence  such that
  the $p$-coordinates of the
  population $\vv{z}$ are also i.i.d. (univariate) with mean zero and
  variance 1.
  
  The name of sample covariance  matrix
  originates  from the fact that the matrix $\vv{S}_n$
  has also the form
  \eqref{eq:1.1} with $\vv{x}_{nj} = \vv{\Sigma}_p^{1/2} \vv{z}_j$,
  which indeed is an i.i.d. sample from the population
  $\vv{x} = \vv{\Sigma}_p^{1/2} \vv{z}$. This population is centred
  with population covariance matrix  $ \cov(\vv{x})=\vv{\Sigma}_p$.
  Under the sole condition of $\mu_{\vv{\Sigma}_p}$ having a weak limit,
  it is shown in \cite{BS94,BS95} that $\vv{S}_n$ has an LSD in the
  form of  a generalized Mar\v{c}enko-Pastur law.

\item[(c)\ ] {\em Large sample covariance matrices without independence structure in columns.}
  \cite{BZ}  directly considered the general model \eqref{eq:1.1} with
  independent columns, not necessarily identically distributed.  They
  however imposed a common covariance structure on the column vectors,
  namely $\cov(\vv{x}_{nj})\equiv \vv{\Sigma}_p$, $1\le j\le n$. Under appropriate
  moment conditions on the coordinates of the vectors  $\vv{x}_{nj}$,
  \cite{BZ}
  shows that the generalized Mar\v{c}enko-Pastur law still
  holds as in (b).

\item[(d)\ ] {\em Separable or weighted covariance matrices.}
  For data matrix $\vv{X_n}$ with both row and column dependence,
  \red{Zhang (2006)}\cite{ZL06} studied the
  separable covariance matrices of the form
  \[
  \vv{S}_n=\frac{1}{n} \vv{X}_n \vv{X}_n^*= \frac1n \vv{A}_p^{1/2}\vv{Z}_n\vv{B}_n\vv{Z}_n^*\vv{A}^{1/2}_p, \quad
  \text{with} \quad \vv{X}_n=\vv{A}_p^{1/2}\vv{Z}_n\vv{B}_n^{1/2},
  \]
  where $\vv{A}_p$ and $\vv{B}_n$ are two non-negative definite
  symmetric matrices and
  $\vv{Z}_n=\{\vv{z}_1,\ldots,\vv{z}_n\}=(z_{ij})$ is a $p\times n$
  pure noise matrix with 
  i.i.d. standardized entries.
  Under the weak convergence of both sequences $\mu_{  \vv{A}_p}$ and $\mu_{ \vv{B}_n}$,
  \cite{ZL06} established an LSD for the matrix $\vv{S}_n$. 

  A particular feature here is that under appropriate moment
  conditions on the i.i.d. noise entries $\{z_{ij}\}$, universality
  applies: namely the LSD of the sequence $\vv{S}_n$ is the same as if
  these entries are Gaussian.  It follows that we can assume $B_n$ is
  diagonal, with non-negative diagonal elements $\{ b_{nj}\}_{1\le j\le n} 
$. In
  this form, we have
  \[
  \vv{S}_n= \frac1n \sum_{j=1}^n
  b_{nj}\vv{A}_p^{1/2}\vv{z}_j\vv{z}^*_j\vv{A}_p^{1/2}.
  \]
  This is the so-called {\em weighted sample covariance matrices} with
  weights $w_{nj} = b_{nj}/n$, $1\le j\le n$.  Such weighted sample
  covariance matrix occurs in  \cite{Zh11} 
  in their study on the realized covariance matrix in stock price
  modelling.  Note that this weighted matrix is a special case of the
  model \eqref{eq:1.1} with column vectors $\vv{x}_{nj}=b_{nj}\vv{A}_p^{1/2}\vv{z}_j$.

\item[(e)\ ]
  {\em Time series and random field data matrices.}
  For data from time
  series models, \red{ Jin et al. (2009), Yao (2012) and Jin et
    al. (2014)}\cite{Jin09,Yao12,Jin14} investigated the LSDs of sample
  covariance matrices with data matrix generated by $p$ independent copies of
  $n$ consecutive observations of a scalar linear time series, which
  can be treated by the method in \red{Bai and Zhou(2008)}\cite{BZ} as
  $p$ independent samples with a common population covariance
  matrix. \red{Liu and Paul (2017)}\cite{LP17} extended the framework to
  high-dimensional linear time series models with coefficient matrices
   simultaneously diagonalizable. When innovations are Gaussian, their
  model can be viewed as an extension of \red{ Jin et al. (2009), Yao
    (2012)}\cite{Jin09,Yao12} to $n$ consecutive observations of $p$
  independent linear processes.

  For random field models, \red{Hachem,
    Loubaton and Najim (2006)}\cite{HLN} considered the Gram random
  matrix model with a given variance profile, where the data matrix
  $\vv{X}=(x_{ij})_{p\times n}$ has i.i.d. entries with
  $x_{ij}=\sigma(i/p,j/n) z_{ij}$, where $z_{ij}$ are i.i.d. entries
  with zero mean and unit variance and $\sigma:[0,1]^2\to\mathbb{R}$ is
  the variance profile function. This is also a particular case of the
  model \eqref{eq:1.1}
  with column vectors $\vv{x}_{nj} =\l\{\sigma(i/p,j/n)
  z_{ij}\r\}_{1\le i\le p}  $.

\end{enumerate}

Despite the rich literature above on singular values of various large
data matrices,  
there still exist
several important types of data in finance and economics, of which limiting behaviours of its singular values
remain unknown.

\medskip
\noindent{\em Case I. Multi-dimensional diffusion process with anisotropic co-volatility.} \quad Data in the analysis of a log price process is always modelled by a
multi-dimensional diffusion process, that is, a $p$-dimensional
process satisfying the stochastic differential equation ${\rm
  d}\vv{X}_t=\vv{\mu}_t{\rm d}t+\vv{\Gamma}_t{\rm d}\vv{B}_t$, where
$\vv{\mu}_t$ is a $p$-dimensional drift process, $\vv{\Gamma}_t$ is a
$p\times p$ matrix-valued co-volatility process and
$\vv{B}_t$ is a standard $p$-dimensional Brownian motion
\citep{JP98,And01,Zh11}.
Financial data analysts are interested in the integrated covariance
matrix $\vv{\Sigma}^{ICV}=\int_0^1\vv{\Gamma}_t\vv{\Gamma'}_t{\rm d}t$
and use the realized covariance matrix
$\vv{\Sigma}^{RCV}=\sum_{l=1}^n\Delta\vv{X}_l\Delta\vv{X}'_l$ as an
estimator of $\vv{\Sigma}^{ICV}$, where
$\Delta\vv{X}_l=\vv{X}_{\tau_{l,n}}-\vv{X}_{\tau_{l-1,n}}$ for
$l=1,\ldots,n$ with $\{\tau_{l,n}\}$ the observation times. In large
sample case when dimension $p$ is fixed, $\vv{\Sigma}^{RCV}$ is proved
to be consistent to $\vv{\Sigma}^{ICV}$. (See \red{Jacod and Protter
  (1998)}\cite{JP98} for details.) However, this is no longer true in the high-dimensional case when dimension $p$ grows
proportionally with the observation frequency $n$. Thus, it is important to find the connections
between spectra  of $\vv{\Sigma}^{RCV}$ and $\vv{\Sigma}^{ICV}$ in the
high-dimensional situation.
For a class of
diffusion processes with co-volatility processes having isotropic
time-varying spectra, i.e, $\vv{\Gamma}_t=\gamma_t\vv{\Sigma}$,
\cite{Zh11} derived an LSD for the  realized covariance matrix
$\vv{\Sigma}^{RCV}$.
For more
diffusion processes that have a co-volatility process
with {\em anisotropic}  time-varying spectrum, the limiting behaviours
of $\vv{\Sigma}^{RCV}$ remain unknown.

\medskip
\noindent{\em Case II. Matrix-valued time series.}\quad
Matrix-valued time series models are always used to investigate data collected in a matrix form and have been widely applied in finance and economics \citep{Chen21,Chen20,Chen19}.
For example,
to study the evolution of macroscopic economic indices among
different countries over certain period, \red{Chen et
  al. (2021)}\cite{Chen21} proposed the matrix-valued auto-regressive
model $\vv{X}_t=\vv{A}\vv{X}_{t-1}\vv{B'}+\vv{Z}_{t-1}$, where
$(\vv{X}_t)$ is an $m\times n$ matrix-value process with column indices
standing for the economical indices and row indices for countries,
$\vv{A}$ and $\vv{B}$ being $m\times m$ and $n\times n$ coefficient
matrices respectively, and $(\vv{Z}_t)$ standing for innovations. Such matrix-valued model has been shown to well capture macroscopic evolutionary characters among small economic
bodies. Though estimation of coefficient
matrices has been extensively studied in \red{Chen et
  al. (2021)}\cite{Chen21} in large sample case, the same problem in
high-dimensional settings remains unclear when row and column
dimensions are large.
It is hence expected that the derivation of a LSD for the
singular values of the data
matrix $\vv{X}_t$ can reflect certain effective information about singular value distributions of the 
coefficient matrices  $\vv{A}$ and $ \vv{B}$
in the  high dimensional situation.

\medskip
\noindent{\em Case III. Finite mixture model.}\quad Models based on finite mixture distributions provides a flexible extension of classical statistical models and have been applied in diverse areas such as genetics, signal processing and machine learning \cite{FS06,Li18,McP00}. The observations in a finite mixture model can always be viewed as samples drawn randomly from several populations with different means or covariance matrices with certain proportion. As a special case, a scale mixture model was studied by \cite{Li18}, in which covariance matrices of different populations differ only by a random factor. The dependence among LSDs of the sample covariance matrix and the common population covariance matrix as well as the distribution of scale variable is derived. However, when population covariance matrices more general structures, the existence of LSDs of sample covariance matrices and their dependence on LSDs of population covariance matrices are still unknown.

\medskip
Back to the existing literature recalled  above, we observe that
models (a)-(e) 
models share a common feature:
the column vectors $\vv{x}_{nj}$'s of their data matrices are 
independent, and their covariance matrices $\vv{\Sigma}_{nj}, 1\leq j\leq n$ are
{\em simultaneously diagonalizable}, either directly or asymptotically.
Indeed for models (a)-(e), the $n$ matrices $\vv{\Sigma}_{nj}, 1\leq j\leq n$ are, 
respectively,
\[
\text{(a)~}  \tau_j \vv{I}_p, \quad
\text{(b)~}    \vv{\Sigma}_p, \quad
\text{(c)~}    \vv{\Sigma}_p, \quad
\text{(d)~}  b_{nj} \vv{A}_p, \quad
\text{(e)~ random field models:}   \diag  \l(  \sigma^2(i/p,j/n)   \r)_{1\leq i\leq p}.
\]
In each case, the $n$ matrices are directly simultaneously diagonalizable. 
For
the remaining case of time series models in (e), 
the data matrix consists of $p$ consecutive observations of $n$
independent stationary linear time series. For each time series $j$, the
 population covariance matrix $\vv{\Sigma}_{nj}$ for the  $p$ observations
is a Toeplitz matrix (filled  with  the first $p$ auto-covariances  of the series).
Although these $n$ Toeplitz matrices $\vv{\Sigma}_{nj}$  are
not directly simultaneously diagonalizable,
it is well known
(\cite{G06}) that under certain
summability assumptions on its entries, any Toeplitz matrix
is  asymptotically equivalent to
a circulant matrix. As  circulant matrices share a same system of
eigenvectors, and thus are simultaneously diagonalizable, we see that 
the Toeplitz matrices $\vv{\Sigma}_{nj}$ 
are asymptotically simultaneously diagonalizable.

These observations
inspire the work in this paper.  We consider a data matrix $\vv{X}_n$
with independent columns  $\vv{x}_{nj}$, and the $n$ population
covariance matrices $\vv{\Sigma}_{nj}=\cov  (\vv{x}_{nj}  )$
can be asymptotically simultaneously diagonalizable, see a
 precise definition of this property  later in
Section~\ref{sec:results}.
Next, a key step is  to model and connect the eigenvalues of $\vv{\Sigma}_{nj}$ 
through two groups of
parameters and a sequence of link functions under some regularization
conditions.  By using the
Stieltejes transform method and under suitable moment conditions,
the LSD for singular values of the data matrix $\vv{X}_n$ is found,
and determined 
through a
system of functional equations involving the limiting distributions of the
two groups of parameters and the limiting link function. The system
captures clearly the connection between the limiting spectrum of the
data matrix and the population spectra.
Note that 
even though our basic assumptions require the independence of 
columns of the data matrix,  our results can  be extended to
certain  data  matrices with both row and column dependence inside,
by using some
universality arguments from  random matrix theory (RMT).
This is done for 
the 
separable covariance covariance model (d)
and for 
sample covariance matrices from  vector-valued linear time series with
coefficient matrices simultaneously diagonalizable, see Section~\ref{ssec:ts}.

In Section~\ref{sec:app}, we apply our general
results to the non-resolved Cases I,II and III mentioned above.
In Section~\ref{ssec:rcv} (case I), 
we establish an LSD for 
the realized covariance matrix  $\vv{\Sigma}^{RCV}$ for 
a multi-dimensional diffusion process with a co-volatility process
that has an anisotropic time-varying spectrum
(under suitable regularity conditions).
In Section~\ref{ssec:mar}, we derive an LSD for the singular values of
the data matrix $\vv{X}_t$ from a  matrix-valued auto-regressive process. In Section~\ref{ssec:mix}, we obtain an LSD of sample covariance matrix for a generalized finite mixture model with population covariance matrices simultaneously diagonalizable.  

The rest of the paper is organized as
follows. Section~\ref{sec:results} introduces the setting and the main
results of the paper.
In Section~\ref{sec:existing}, we show that the existing literature as
recalled in models (a)-(e) is included in our main result, with some
new extensions.
Our results thus 
give a unified approach for these different models of data matrices.
Section~\ref{sec:app} develops applications to three unresolved models,
Cases I, II and III, mentioned above.
Technical proofs are gathered in Appendix. %Section~\ref{sec:proofs}.

\section{Main Results}
\label{sec:results}

In this section, we focus on the sample covariance matrix $\vv{S}_n$ defined in (\ref{eq:1.1}) and the existence and uniqueness of its LSD in the 
high-dimensional setting (\ref{eq:1.3}). The weak convergence of ESDs is established through that of the corresponding Stieltjes transforms. The method of Stieltjes transform has been developed to be a powerful tool 
in study asymptotic spectral properties of random matrices. See \red{Bai and Silverstein (2010)}\cite{BS10} and reference therein. The Stieltjes transform of ESD $\mu_{\vv{S}_n}$ is defined as
\begin{equation}\lb{eq:2.1}
m_n(z)=\int \frac{1}{x-z}{\rm d} \mu_{\vv{S}_n}(x)=\frac{1}{p}{\rm tr}(\vv{S}_n-z\vv{I})^{-1},
\end{equation}
where $z\in\mathbb{C}^+:=\{z\in\mathbb{C}:{\rm Im}(z)>0\}$, the upper half complex plane.
It is well known that a sequence of measure converges vaguely to certain measure if and only the corresponding Stieltjes transform converges to the Stieltjes transform of the limiting measure on the upper complex half plane $\mathbb{C}^+$ pointwisely. Following this routine, it suffices to study the asymptotic behaviour of $m_n(z)$.  

The following additional assumptions on samples $\vv{x}_1,\ldots,\vv{x}_n$ are made.
\begin{assm}\label{assm:1}
\begin{itemize}
\item[$(1)$] For any $i=1,\ldots, p$, $j=1,\ldots,n$, $x_{ij}$ has zero mean, finite fourth moment.
\item[$(2)$] There exists a sequence of $p\times p$ non-negative definite matrices $\{\vv{\Sigma}_i\}$ such that they are uniformly bounded in matrix operator norm $\|\cdot\|_{\rm op}$, i.e.,
\begin{equation}\label{eq:0.4}
\max_{n}\max_{1\leq i\leq n}\|\vv{\Sigma}_i\|_{\rm op}<\infty;
\end{equation}
and diagonalizable simultaneously, i.e., there exists a unitary matrix $\vv{U}_n$ and diagonal matrices $\vv{\Lambda}_1,\ldots,\vv{\Lambda}_n$ with $\vv{\Lambda}_i={\rm diag}\{\lambda_{i,l}:l=1,\ldots,p\}$, $i=1,\ldots,n$ such that
\begin{equation}\label{eq:0.2}
\vv{\Sigma}_i=\vv{U}^*_n\vv{\Lambda}_i\vv{U}_n,~~~~~i=1,2,\ldots,n.
\end{equation}
\item[$(3)$] There exists a bounded continuous functions $f:\mathbb{R}^k\times\mathbb{R}^m\to\mathbb{R}_+=[0,\infty)$ and two sequences $(\vv{a}_l)_{1\leq l\leq p}\in\mathbb{R}^k$ and $(\vv{b}_i)_{1\leq i\leq n}\in\mathbb{R}^m$ such that eigenvalues of $\vv{\Sigma}_1,\ldots,\vv{\Sigma}_n$ satisfy 
\begin{equation}\label{eq:0.3}
\lmd_{i,l}=f(\vv{a}_l;\vv{b}_i),~~~~l=1,\ldots,p; i=1,\ldots,n.
\end{equation}

\item[$(4)$] For any sequence of $p\times p$ matrices $\vv{B}_1,\ldots,\vv{B}_n$ with $\sup_i\|\vv{B}_i\|_{op}<\infty$, it holds that
\begin{equation}\lb{eq:0.1}
\frac{1}{n^3}\sum_{i=1}^n \mathbb{E}|\vv{x^*}_i\vv{B}_i\vv{x}_i-{\rm tr}(\vv{B}_i\vv{\Sigma}_i)|^2=o(1).
\end{equation}

\item[$(5)$] Denote $G_p$ and $H_n$ the ESDs of $\vv{a}_1,\ldots,\vv{a}_p$ and $\vv{b}_1,\ldots,\vv{b}_n$, respectively. As $n,p\to\infty$, $G_p$ and $H_n$ converge weakly to Borel probability measure $G$ and $H$, respectively.

\end{itemize}
\end{assm}
\begin{rem}\label{rem:2.1}
The matrices $\vv{\Sigma}_1,\ldots,\vv{\Sigma}_n$ are always chosen as the corresponding population covariance matrices for $\vv{x}_1,\ldots,\vv{x}_n$, respectively. However, sometimes, although the moment condition (\ref{eq:0.1}) holds for $\vv{\Sigma}_1,\ldots,\vv{\Sigma}_n$, these population covariance matrices are not simultaneously diagonalizable. In this case, if there exist $\vv{\Sigma'}_1,\ldots,\vv{\Sigma'}_n$ such that $\{\vv{\Sigma'}_i\}_{1\leq i\leq n}$ are diagonalizable simultaneously and 
\begin{equation}\label{eq:2.6}
\fr{1}{n^3}\sum_{i=1}^n \left[{\rm tr}(\vv{B}_i(\vv{\Sigma}_i-\vv{\Sigma'}_i))\right]^2=o(1),
\end{equation}
for any sequence of $p\times p$ matrices $\vv{B}_1,\ldots,\vv{B}_n$ with $\sup_i\|\vv{B}_i\|_{op}<\infty$, then we use $\vv{\Sigma'}_i$'s to replace $\vv{\Sigma}_i$'s so that moment condition (4) remains valid for $\vv{\Sigma'}_1,\ldots,\vv{\Sigma'}_n$. One of the sufficient conditions for (\ref{eq:2.6}) is as follows:
\begin{equation}\label{eq:2.7}
\fr{1}{n^2}\sum_{i=1}^n {\rm tr}((\vv{\Sigma}_i-\vv{\Sigma'}_i)(\vv{\Sigma}_i-\vv{\Sigma'}_i)^*)=o(1).
\end{equation}
\end{rem}

With the assumptions above, the main result of this paper 
can be stated as follows.
\begin{thm}\label{thm:0.1}
Suppose that a data matrix $\vv{X}_n=(\vv{x}_1,\ldots,\vv{x}_n)$ of $n$ independent $p$-dimensional samples satisfies Assumption \ref{assm:1}. Then, in the high-dimensional setting $p/n\to c\in(0,\infty)$, with probability $1$, the empirical spectral distribution $\mu_{\vv{S}_n}$ converges weakly to a unique deterministic Borel probability measure $\mu$ with its Stieltjes transform $m(z)$ satisfying
\begin{equation}\label{eq:0.5}
m(z)=-\frac{1}{z}\int_{\mathbb{R}^k}(K(\vv{a},z)+1)^{-1}{\rm d} G(\vv{a})
\end{equation}
in which $K: \mathbb{R}^k\times\mathbb{C}^+\to\mathbb{C}_+:=\{z\in\mathbb{C}:{\rm Im}(z)\geq 0\}$ is the unique solution to the following functional equation
\begin{equation}\label{eq:0.6}
K(\vv{a},z)=\int_{\mathbb{R}^m}\frac{f(\vv{a},\vv{b})}{-z+c\int_{\mathbb{R}^k}\frac{f(\vv{a'},\vv{b})}{K(\vv{a'},z)+1}{\rm d} G(\vv{a'})} {\rm d} H(\vv{b})
\end{equation}
on the subset
$$
\bigg\{(\vv{a},z)\in\mathbb{R}^k\times\mathbb{C}^+:{\rm Im}(K(\vv{a},z))\geq 0,~{\rm Im}(zK(\vv{a},z))\geq0\bigg\}.
$$
in the sense that if $K_1(\vv{a},z)$ and $K_2(\vv{a},z)$ are two solutions to (\ref{eq:0.6}), then for any $z\in\mathbb{C}^+$,
$$
G(\{\vv{a}:K_1(\vv{a},z)\neq K_2(\vv{a},z)\})=0.
$$

\end{thm}

\begin{cor}\label{cor:0.3}
Suppose that a data matrix $\vv{X}=(\vv{x}_1,\ldots,\vv{x}_n)$ of $n$ independent $p$-dimensional samples satisfies Assumption \ref{assm:1}. In addition, assume that in Assumption \ref{assm:1} (3), there exist two continuous function $g:\mathbb{R}^k\to\mathbb{R}_+$ and $h:\mathbb{R}^m\to\mathbb{R}_+$ such that $f(\vv{a},\vv{b})=g(\vv{a}) h(\vv{b})$. Then, in the high-dimensional setting $p/n\to c\in(0,\infty)$, with probability $1$, the empirical spectral distribution $\mu_{\vv{S}_n}$ converges weakly to a unique deterministic Borel probability measure $\mu$  with its Stieltjes transform $m(z)$ satisfying 
\begin{equation}\label{eq:0.7}
m(z)=-\frac{1}{z}\int_{\mathbb{R}^k}(g(\vv{a})K(z)+1)^{-1}{\rm d} G(\vv{a})
\end{equation}
in which $K:\mathbb{C}^+\to\mathbb{C}_+$ is the unique solution to the following functional equation
\begin{equation}\label{eq:0.8}
K(z)=\int_{\mathbb{R}^m}\frac{h(\vv{b})}{-z+ch(\vv{b})\int_{\mathbb{R}^k}\frac{g(\vv{a'})}{g(\vv{a'})K(z)+1}{\rm d} G(\vv{a'})} {\rm d} H(\vv{b})
\end{equation}
on the subset
$$
\bigg\{z\in\mathbb{C}^+:{\rm Im}(K(z))\geq 0,~{\rm Im}(zK(z))\geq0\bigg\}.
$$
\end{cor}

%Generalizations
To allow varying eigenvalues parametrizing function $f$, we make the following additional assumption and obtain similar conclusions.
\begin{assm}\label{assm:2}
\begin{itemize}
\item[$(3')$] For any $n\geq 1$, there exists a uniformly bounded measurable function $f_n:\mathbb{R}^k\times\mathbb{R}^m\to\mathbb{R}_+$ and two families of vectors $\{\vv{a}_{l,n}\in\mathbb{R}^k:l=1,\ldots,p\}$ and $\{\vv{b}_{i,n}\in\mathbb{R}^m:i=1,\ldots,n\}$ such that eigenvalues of $\vv{\Sigma}_1,\ldots,\vv{\Sigma}_n$ satisfy 
\begin{equation}\label{eq:2.10}
\lambda_{i,l}=f_n(\vv{a}_{l,n};\vv{b}_{i,n}),~~~~l=1,\ldots,p; i=1,\ldots,n.
\end{equation}
Moreover, there exists a bounded continuous functions $f:\R^k\times\R^m\to\R_+$ such that
\begin{equation}\label{eq:2.11}
\lim_{n,p\to\infty} \int_{\mathbb{R}^k}\int_{\mathbb{R}^m} |f_n(\vv{a},\vv{b})-f(\vv{a},\vv{b})|H_n({\rm d}\vv{b})G_p({\rm d}\vv{a})=0.
\end{equation}
\end{itemize}
\end{assm}

\begin{thm}\label{thm:0.2}
Suppose that a data matrix $\vv{X}=(\vv{x}_1,\ldots,\vv{x}_n)$ of $n$ independent $p$-dimensional samples satisfies Assumption \ref{assm:1} with (3) replaced by Assumption \ref{assm:2} (3'). Then, in the high-dimensional setting $p/n\to c\in(0,\infty)$, results in Theorem \ref{thm:0.1} is retained.
\end{thm}

\section{Relation to the existing works}
\label{sec:existing}

In this section, we show that models (a)-(e) in Introduction are special cases of our model.

\subsection{Sample covariance matrices for i.i.d. samples}

Let $\vv{x}_1,\ldots,\vv{x}_n$ be an i.i.d. sample from the population $\vv{x}=\vv{\Sgm}^{1/2}\vv{z}$, where
\begin{itemize}
\item[(1)] $\vv{z}$ is a $p$-dimensional random vector and has i.i.d. entries with zero mean, unit variance and finite fourth moment;
\item[(2)] $\vv{\Sgm}$ is a $p\times p$ non-negative definite Hermitian matrix with $\sup_p\|\vv{\Sigma}\|_{\rm op}<\ift$;
\item[(3)] the empirical spectral distribution function $G_p$ of $\vv{\Sgm}$ converges weakly to a Borel probability measure $G$.
\end{itemize} 
To see the above is a special case of our model, let $\vv{\Sgm}_i=\vv{\Sgm}$ for $i=1,\ldots,n$, $\vv{b}=1$ and $\vv{a}_{i,p}=\sgm_{i,p}^2$, where $\{\sgm_{i,p}^2\}$ are eigenvalues of $\vv{\Sgm}_p$. Then, conditions (1), (2), (3) and (5) in Assumption \ref{assm:1} hold with the link function $f(\vv{a},\vv{b})=\vv{a}$. It remains to show the moment condition (4) is valid. To see this, by Lemma B.26 in \red{Bai and Silverstein (2010)}\cite{BS10}, for any $p\times p$ matrix $\vv{B}$,
$$
\mathbb{E}|\vv{x^*}_1\vv{B}\vv{x}_1-{\rm tr}(\vv{B}\vv{\Sigma})|^2\leq C\left(\mathbb{E}|z_{11}|^4{\rm tr}(\vv{B}\vv{\Sigma}^2\vv{B^*})\right)\leq 
C\|\vv{B}\|_{\rm op}^2{\rm tr}(\vv{\Sigma}^2)
$$
so that for $\vv{B}_1,\ldots,\vv{B}_n$ with norm uniformly bounded by $K$,
$$
\sum_{i=1}^n\mathbb{E}|\vv{x^*}_i\vv{B}_i\vv{x}_i-{\rm tr}(\vv{B}_i\vv{\Sigma})|^2\leq nCK{\rm tr}(\vv{\Sigma}^2)=O(n^2)=o(n^3).
$$
Therefore, by Theorem \ref{thm:0.1} and Corollary \ref{cor:0.3}, the LSD of $\vv{S}_n$ exists uniquely. Moreover, its Stieltjes transform satisfies
$$
m(z)=-\frac{1}{z}\int\frac{1}{\lambda K(z)+1}\d G(\lambda),
$$
where
$$
K(z)=\left(-z+c\int \frac{\lambda}{\lambda K(z)+1}{\rm d} G(\lambda)\right)^{-1}.
$$
Note that
$$
1=K(z)\left(-z+c\int \frac{\lambda}{\lambda K(z)+1}{\rm d} G(\lambda)\right)=-zK(z)+c+czm(z)
$$
Thus, $1-c-czm(z)=-zK(z)$ and
$$
m(z)=\int\frac{1}{\lambda (1-c-czm(z))-z}{\rm d} G(x),
$$
which is the celebrated Mar\v{c}enko-Pastur equation.

\subsection{The generalized sample covariance matrices}

Consider the following generalized sample covariance matrix:
$$
\vv{S}_n=\frac{1}{n}\vv{A}_p^{1/2}\vv{Z}_n\vv{B}_n\vv{Z}^*_n\vv{A}_p^{1/2},
$$
with its corresponding data matrix:
$$
\vv{X}=\vv{A}_p^{1/2}\vv{Z}_n\vv{B}_n^{1/2},
$$
where $\vv{A}_p$ and $\vv{B}_n$ are $p\times p$ and $n\times n$ non-negative definite Hermitian matrices, respectively and $\vv{Z}_n=(Z_{ij})_{p\times n}$ is a $p\times n$ random matrix having i.i.d. entries with $\mathbb{E}Z_{ij}=0$, $\mathbb{E}|Z_{ij}|^2<\infty$. The limiting behaviour of 
ESDs of the generalized sample covariance matrix has been studied by \red{Zhang (2006)}. In what follows, we are going to provide a new perspective to this model by using results developed in Section 2.

It should be noticed that when $\vv{B}_n$ is diagonal or entries of $\vv{Z}$ follow standard Gaussian distribution, the data matrix $\vv{X}$ can be viewed as having $n$ independent columns. Indeed, when $\vv{B}_n={\rm diag}\{b_1,\ldots,b_n\}$ with $b_i\geq 0$ for $i=1,\ldots,n$, let $\vv{x}_i$ and $\vv{z}_i$ be the $i$-th column of $\vv{X}$ and $\vv{Z}$, respectively. Then, we have $\vv{x}_i=\sqrt{b_i}\vv{A}_p^{1/2}\vv{z}_i$ for $i=1,\ldots,n$ so that columns of $\vv{X}$ are independent. In this case, $\vv{S}_n=n^{-1}\sum_{i=1}^n c_i^2\vv{A}_p^{1/2}\vv{z}_i\vv{z^*}_i\vv{A}_p^{1/2}$, which is actually the weighted sample covariance matrix. 
For the other case when entries of $\vv{Z}_n$ are Gaussian distributed, let 
$\vv{B}_n=\vv{V}_n\vv{\Lambda}_{\vv{B}}\vv{V^*}_n$, where $\vv{V}_n$ is 
orthogonal and $\vv{\Lambda}_{\vv{B}}$ is diagonal with its diagonal entries eigenvalues of $\vv{B}_n$. Define $\bar{\vv{Z}}_n=\vv{Z}_n\vv{V}_n$. Then, due to the normality of $\vv{Z}_n$, $\bar{\vv{Z}}_n$ still has i.i.d. standard Gaussian entries. Thus, $\vv{S}_n$ has the same limiting behaviour of ESDs as
$$
\bar{\vv{S}_n}=\vv{A}_p^{1/2}\bar{\vv{Z}}_n\vv{\Lambda}_{\vv{B}}\bar{\vv{Z^*}}_n\vv{A}_p^{1/2},
$$
and the latter can be treated as a large sample covariance matrix for $n$ 
independent samples. Although, in general, the data matrix $\vv{X}$ may have both row and column dependence when $\vv{B}_n$ is not diagonal, we show below that the limiting behaviour of $\vv{S}_n$ is the same as that of another generalized sample covariance matrix $\tilde{\vv{S}}_n$ obtained by replacing $\vv{Z}_n$ by $\tilde{Z}_n$ with standard 
Gaussian entries. Therefore, our method still works to study 
the limiting behaviour of $\vv{S}_n$. In other words, the results of LSD of the weighted sample covariance matrices are universal among the generalized sample covariance 
matrices.

Before detailed discussion, we make the following assumptions:
\begin{itemize}
\item[(1)] $\vv{Z}_n=(Z_{ij})_{p\times n}$ is a $p\times n$ random matrix having i.i.d. entries with $\mathbb{E}Z_{ij}=0$, $\mathbb{E}|Z_{ij}|^2=1$ and $\mathbb{E}|Z_{ij}|^4<\infty$;

\item[(2)] $\vv{A}_p$ and $\vv{B}_n$ are uniformly bounded in norm, i.e., $a_0:=\sup_p\|\vv{A}_p\|_{op}<\infty$, $b_0:=\sup_n\|\vv{B}_n\|_{op}<\infty$;

\item[(3)]  ESDs of $\vv{A}_p$ and $\vv{B}_n$ converge weakly two Borel probability measures $G$ and $H$ respectively.
\end{itemize}
\begin{rem}
The only one additional assumption in this paper, compared to those in \cite{ZL06}, is the boundedness condition of $\vv{A}_p$ and $\vv{B}_n$ in (2). By using the truncation technique with the rank inequality, this assumption can be easily removed. However, to remain the consistency 
to our general settings, we are not going remove it in the following discussion.  
\end{rem}

We study the limiting behaviour of the Stieltjes transform $m_n(z)$ of $\vv{S}_n$ by the following routine:
\begin{itemize}
\item[(i)] We show that $m_n(z)-\mathbb{E}m_n(z)\overset{\rm a.s.}{\to}0$ as $n\to\infty$, when entries of $\vv{Z}$ are i.i.d. standardized variables with arbitrary distribution;

\item[(ii)] Let $\tilde{Z}_n=(\tilde{Z}_{ij})_{p\times n}$ be a Gaussian 
matrix with i.i.d. entries having $\mathbb{E}(\tilde{Z}_{ij})=0$, $\mathbb{E}|\tilde{Z}_{ij}|^2=1$. Define
$$
\tilde{S}_n=\frac{1}{n}\vv{A}_p^{1/2}\tilde{Z}_n\vv{B}_n\tilde{Z}_n\vv{A}_p^{1/2}.
$$
The Stieltjes transform of its ESD is denoted by $\tilde{m}_n(z)$. We show that $\mathbb{E}(m_n(z))-\mathbb{E}(\tilde{m}_n(z))\to0$ as $n\to\infty$. 

\item[(iii)] We study the limiting behaviour of $\mathbb{E}(\tilde{m}_n(z))$ by treating $\tilde{S}_n$ as a weighted sample covariance matrix and using the method developed in Section 2.
\end{itemize}
The first two steps are achieved by using the McDiarmid inequality (developed in \cite{McD}) and the Lindeberg principle (developed in \cite{Ch06}). Their proofs are technical and tedious and thus contained in Appendix. In what follows, we only focus on the final step and the alternative sample covariance matrix $\tilde{S}_n$.  

Without loss of generality, we now assume that $\tilde{Z}$ has i.i.d. standard Gaussian entries and $\vv{B}_n={\rm diag}\{b_1,\ldots,b_n\}$ diagonal. Let $\vv{x}_i$ and $\tilde{z}_i$ be the $i$-th column of $\vv{X}$ and $\tilde{Z}$, respectively. Then, $\vv{x}_i=\sqrt{b_i}\vv{A}_p^{1/2}\tilde{z}_i$ for $i=1,\ldots,n$ so that data matrix $\vv{X}$ has $n$ independent columns. 

Let $\vv{\Sgm}_i=b_i\vv{A}_p$, $i=1,\ldots,n$. Under the uniform boundedness assumption on $\vv{A}_p$ and $\vv{B}_n$, $\vv{\Sgm}_i$'s are also uniformly bounded in norm and (2) in Assumption \ref{assm:1} holds. Meanwhile, $\vv{\Sigma}_i$'s are simultaneously diagonalizable with eigenvalues $b_i\sigma_{l,p}^2$, $l=1,\ldots,p$, where $\sigma_{1,p}^2,\ldots,\sigma_{l,p}^2$ are eigenvalues of $\vv{A}_p$. By letting the link function $f(a,b)=ab$, we see (3) in Assumption \ref{assm:1} holds. Under assumption (3) mentioned before, (5) in Assumption \ref{assm:1} is valid. It remains to show (4) in Assumption \ref{assm:1}. To see this, 
notice that $\vv{x}_i=\sqrt{b_i}\vv{A}_p^{1/2}\tilde{z}_i$, by Lemma B.26 in \red{Bai and Silverstein (2010)}\cite{BS10}, for any $p\times p$ matrix $\vv{R}_i$, there exists a positive constant $C$ independent of $i$ such that
$$
\mathbb{E}|\vv{x^*}_i\vv{R}_i\vv{x}_i-c^2{\rm tr}(\vv{R}_i\vv{\Sigma})|^2\leq Cb_i^2\left(\mathbb{E}|z_{11}|^4{\rm tr}(\vv{R}_i\vv{A}_p^2\vv{R}^*_i)\right)\leq C\|\vv{R}_i\|_{\rm op}^2b_i^2{\rm tr}(\vv{A}_p^2).
$$
So, for $\vv{R}_1,\ldots,\vv{R}_n$ with norm uniformly bounded by $K$,
$$
\sum_{i=1}^n\mathbb{E}|\vv{x^*}_i\vv{R}_i\vv{x}_i-{\rm tr}(\vv{R}_i\vv{\Sigma})|^2\leq nCK\|\vv{B}\|_{op}^2{\rm tr}(\vv{A}_p^2)=O(n^2)=o(n^3),
$$
and condition (4) in Assumption \ref{assm:1} is satisfied. 

By Corollary \ref{cor:0.3}, we know that LSD of $\tilde{S}_n$ exists uniquely with its Stieltjes transform $m(z)$ satisfying
$$
m(z)=-\frac{1}{z}\int_{\mathbb{R}_+}\frac{1}{y K(z)+1}{\rm d} G(y),
$$
where
\begin{align*}
K(z)&=\int_{\mathbb{R}^+}\frac{x}{-z+cx\int_{\mathbb{R}_+}\frac{y}{y K(z)+1}{\rm d} G(y)}{\rm d} H(x).
\end{align*}
When either $H$ or $G$ is a one point mass at $0$, $m(z)=-\frac{1}{z}$ so that the LSD becomes also a one point mass $\delta_0$ at $0$. Otherwise, when neither $H$ nor $G$ is not degenerated to $\delta_0$, we have by $zm(z)\in\mathbb{C}^+$ that ${\rm Im}(K(z))>0$ since 
$$
zm(z)=-\int_{\mathbb{R}_+}\frac{1}{1+yK(z)}{\rm d} H(y).
$$
Meanwhile, if ${\rm Im}(zK(z))=0$, then it must have
$$
{\rm Im}\(\int_{\mathbb{R}_+}\frac{y}{-y zK(z)-z}{\rm d} G(y)\)=0,
$$
which contradicts the assumption that neither $H$ nor $G$ is not degenerated to $\delta_0$. Thus, ${\rm Im}(zK(z))>0$.

Let $p(z)=K(z)$ and
$$
q(z)=-\frac{1}{z}\int \frac{y}{y K(z)+1}{\rm d} H(y).
$$
We see immediately that
\begin{itemize}
\item[(a)] when either $H$ or $G$ is a one-point mass at $0$, then $m(z)=-\frac{1}{z}$ so that the LSD becomes also a one point mass $\delta_0$ 
at $0$;

\item[(b)] when neither $H$ nor $G$ is not degenerated to $\delta_0$, $(m(z),p(z),q(z))$ solves the following system of equations uniquely
\begin{align*}
m(z)&=-\frac{1-c^{-1}}{z}-\frac{c^{-1}}{z}\int \frac{1}{1+q(z)x}{\rm d} 
H(x)\\
m(z)&=-\frac{1}{z}\int \frac{1}{1+p(z)y}{\rm d} G(y)\\
m(z)&=-\frac{1}{z}-p(z)q(z).
\end{align*}
on the set $\{z\in\mathbb{C}^+:{\rm Im}(m(z))>0,{\rm Im}(q(z))>0,{\rm Im}(zp(z))>0\}$.
\end{itemize}
which is consistent with Theorem 4.1.1 in \cite{ZL06}.

\subsection{The Centered Gram random matrix model with a given variance profile}

Consider a $p\times n$ random matrix $\vv{X}_n=(X_{n,ij})_{p\times n}$ with 
its entry
$$
X_{n,ij}=\sigma(i/p,j/n) Z_{ij},
$$where
\begin{itemize}
\item[(1)] $Z_{ij}$ being centred i.i.d. random variables with unit variance and finite fourth-moment; 

\item[(2)] the variance profile $\sigma:[0,1]^2\to\mathbb{R}$ is continuous and bounded.

\end{itemize}
The corresponding sample covariance matrix $\vv{S}_n=\frac{1}{n}\vv{X}_n\vv{X}^*_n$. 

Let $\vv{z}_j=(Z_{ij}:i=1,\ldots,p)$ and $\vv{\Sigma}_j={\rm diag}\{\sigma(i/p,j/n)^2:i=1,\ldots,p)\}$ for $j=1,\ldots,n$.  Denote $\vv{x}_{jn}$ the $j$-th column of $\vv{X}_n$, $j=1,\ldots,n$. Then, $\vv{z}_j$'s 
are independent and so are $\vv{x}_{jn}=\vv{\Sigma}^{1/2}_j\vv{z}_j$'s. Meanwhile, it is clear that condition (3) in Assumption \ref{assm:1} holds with $\vv{a}_{l,p}=i/p$, $\vv{b}_{i,n}=i/n$ and the link function $f(a,b)=\sgm^2(a,b)$.  Moreover, the choice of $\{\vv{a}_{l,p}\}$ 
and $\{\vv{b}_{i,n}\}$ ensures Assumption \ref{assm:1} (5) with limiting distributions $H$ and $G$ uniformly distributed on the interval $(0,1)$. Finally, for the moment condition, by Lemma B.26 in \red{Bai and Silverstein (2010)}\cite{BS10}, for any $p\times p$ matrix $\vv{B}_i$, there exists a positive constant $C$ independent of $i$ such that
$$
\mathbb{E}|\vv{x^*}_{in}\vv{B}_i\vv{x}_{in}-c^2{\rm tr}(\vv{B}_i\vv{\Sigma}_i)|^2\leq C\left(\mathbb{E}|z_{11}|^4{\rm tr}(\vv{B}_i\vv{\Sigma}_i^2\vv{B^*}_i)\right)\leq C\|\vv{B}_i\|_{\rm op}^2 \|\sigma^2\|_{\infty}\cdot n=O(n),
$$
where $\|\sigma^2\|_{\infty}=\sup_{(s,t)\in[0,1]^2}|\sigma^2(s,t)|$. The moment condition (4) in Assumption \ref{assm:1} then follows.

Therefore, by Theorem \ref{thm:0.1}, the Stieltjes transform of the LSD satisfies
$$
m(z)=-\frac{1}{z}\int_{0}^1\frac{1}{K(s,z)+1}{\rm d} s,
$$
where
\begin{align*}
K(s,z)&=\int_{0}^1\frac{\sigma^2(s,t)}{-z+c\int_{0}^1\frac{\sigma^2(s,t)}{K(s,z)+1}{\rm d} s}{\rm d} t.
\end{align*}
The final result is consistent with that discussed in Section 3.1 in \cite{HLN}. 

To relax the requirement of entry independence, we propose the following result.
\bg{thm}[Generalization of the centred Gram model]\label{thm:3.1}
Let $\vv{x}_{in}=\vv{\Sigma}_i^{1/2}\vv{z}_i$, $i=1,2,\ldots,n$, where
\begin{itemize}
\item[(1)]  $\vv{Z}=(\vv{z}_1,\ldots,\vv{z}_n)$ has i.i.d. entries with zero mean, unit variance and finite fourth moment;

\item[(2)] $\vv{\Sigma}_i$'s are a sequence of simultaneously diagonalizable non-negative definite Hermitian matrices;

\item[(3)] there exists a bounded continuous function $\sigma:[0,1]^2\to\mathbb{R}$ such that eigenvalues of $\vv{\Sigma}_i$ are $\{\sigma(l/p,i/n)^2:l=1,\ldots,p\}$ for $i=1,\ldots,n$.

\end{itemize}

Let $\vv{X}_n=(\vv{x}_{1n},\ldots,\vv{x}_{nn})$. Then, the LSD of sample covariance matrix $\vv{S}_n=\frac{1}{n}\vv{X}_n\vv{X}^*_n$ exists uniquely with its Stieltjes transform satisfying
$$
m(z)=-\frac{1}{z}\int_0^1\frac{1}{K(s,z)+1}{\rm d} s
$$
where $K:[0,1]\times\mathbb{C}^+\to\mathbb{C}_+$ is the unique solution to the functional equation
$$
K(s,z)=\int_0^1\frac{\sgm(s,t)^2}{-z+c\int_0^1\fr{\sgm(s,t)^2}{K(s,z)+1}{\rm d} s}{\rm d} t.
$$
\end{thm}
The proof of Theorem \ref{thm:3.1} is a direct application of Theorem \ref{thm:0.1} and thus is omitted.

\subsection{Vector-valued time series}
\label{ssec:ts}

Consider the following $p$-dimensional linear process:
\be\lb{eq:3.4.1}
\vv{X}_t=\sum_{j=0}^\infty \vv{A}_j\vv{Z}_{t-j},
\de
where
\begin{itemize}
\item[$(3.4$-$1)$] $(Z_t)_{t\in\mathbb{Z}}$ is a sequence of i.i.d. $p$-dimensional random vectors, with i.i.d. entries having zero mean, unit variance and finite fourth moments; 
\item[$(3.4$-$2)$] The coefficient matrices $\(\vv{A}_j\)_{j=0,1,\cdots}$ are diagonal, non-random with $\sup_p\|\vv{A}_j\|\leq a_j<\ift$ for $j=0,1,\ldots$ and
\be\lb{eq:3.4.2}
\sum_{j=0}^\ift a_j<\ift,
\de
with $\vv{A}_0=\vv{I}_p$.
\end{itemize}
Let $\vv{X}_t=(X_{1,t},\ldots,X_{p,t})'$ and $\vv{Z}_t=(Z_{1,t},\ldots,Z_{p,t})'$, where $'$ represents the matrix transpose. It is easy to see that coordinate processes $(X_{i,t})_{t\geq0}$'s are independent and for $i=1,\ldots,p$, $(X_{i,t})$ follows a scalar linear process:
\be\lb{eq:3.4.3}
X_{i,t}=\sum_{j=0}^\infty a_{i,j}\vv{Z}_{i,t-j},
\de
with $a_{ij}$ the $i$-th diagonal entry of $\vv{A}_j$. Moreover, under (\ref{eq:3.4.2}), all coordinate processes are stationary. It should be mentioned that models considered in \red{Jin et al. (2009), Yao (2012), Jin et al. (2014) and Liu and Paul (2017)}\cite{Jin09,Yao12,Jin14,LP17} all follow (\ref{eq:3.4.1}) and satisfy $(3.4$-$1)$ and $(3.4$-$2)$.

Let $\vv{X}_1,\ldots,\vv{X}_T$ be $T$ consecutive observations of $(\vv{X}_t)$. Denote the data matrix $\vv{X}=(\vv{X}_1,\ldots,\vv{X}_T)$. Consider the corresponding sample covariance matrix
\begin{equation}\label{eq:3.4.4}
\vv{S}_T=\fr{1}{T}\vv{X}\vv{X^*}.
\end{equation}
Let $\mu_T$ be the ESD of $\vv{S}_T$. To study the limiting behaviour of $F_T$ by our method, we need the following two companion matrices
\begin{equation}\lb{eq:3.4.5}
\underline{S}_T=\fr{1}{T}\vv{X^*}\vv{X}=\fr{1}{T}\sum_{i=1}^p \vv{X}_{(i)}\vv{X^*}_{(i)},~~~~\tilde{S}_T=\frac{1}{p}\sum_{i=1}^p\vv{X}_{(i)}\vv{X^*}_{(i)}.
\end{equation}
where $\vv{X}_{(i)}$ is the $i$-th row of $\vv{X}$ for $i=1,\ldots, p$ and $\vv{X}_{(1)},\ldots,\vv{X}_{(p)}$ are mutually independent. Denote $\underline{\mu}_T$ and $\tilde{\mu}_T$ the ESDs of $\underline{S}_T$ and $\tilde{S}_T$, respectively. 

The connections among Stieltjes transforms of $\mu_T$, $\underline{\mu}_T$ and $\tilde{\mu}_T$ are given as follows. Let $m_T(z)$, $\underline{m}_T(z)$ and $\tilde{m}_T(z)$ are Stieltjes transforms of $\mu_T$, $\underline{\mu}_T$ and $\tilde{\mu}_T$, respectively. On the one hand, since $S_T$ and $\underline{S}_T$ share the same positive eigenvalues, it holds 
$$
\underline{\mu}_T=(1-c_T)\dlt_0+c_T\mu_T,
$$
so that
\begin{equation}\lb{eq:3.4.6}
\underline{m}_T(z)=-\fr{1-c_T}{z}+c_Tm_T(z).
\end{equation}
where $c_T=p/T$. On the other hand, since
$$
c_T\tilde{S}_T=\frac{p}{T}\cdot\frac{1}{p}\sum_{i=1}^p\vv{X}_{(i)}\vv{X^*}_{(i)}=\underline{S}_T.
$$
it then holds that
\begin{equation}\lb{eq:3.4.7}
\underline{m}_T(z)=c_T^{-1}\tilde{m}_T\left(\frac{z}{c_T}\right).
\end{equation}
Hence, we have
\begin{equation}\lb{eq:3.4.8}
m_T(z)=\frac{1-c_T}{c_Tz}+\frac{1}{c_T^2}\tilde{m}_T\left(\frac{z}{c_T}\right).
\end{equation}
Hence, it suffices for us to study the limiting behaviour of $\tilde{m}_T(z)$.

Recall that $\tilde{S}_T$ defined above can be viewed as the sample covariance matrix of $p$ independent samples $\vv{X}_{(1)},\ldots,\vv{X}_{(p)}$, where $\vv{X}_{(i)}$ is generated by $T$ consecutive observations of the $i$-th coordinate process satisfying (\ref{eq:3.4.3}). Let $\vv{\Gamma}_{i,T}=\mathbb{E}(\vv{X}_{(i)}\vv{X^*}_{(i)})=(\gamma_{i}(k-j))_{1\leq k,j\leq T}$, $i=1,\ldots,p$, where
\begin{equation}\label{eq:3.4.9}
\gamma_{i}(h)=\mathbb{E}(X_{i,t}X^*_{i,t+h}),~~~~~h\in\mathbb{Z}:=\{0,\pm1,\pm2,\ldots\},
\end{equation}
is the auto-covariance function of the $i$-th coordinate process $(X_{i,t})$ and is independent of $t$. Being $\vv{\Gamma}_{i,T}$ Toeplitz matrices, $\vv{\Gamma}_{i,T}$'s are usually not diagonalizable simultaneously. To build the connection to results in Section 2, we have to find a proxy of $\vv{\Gamma}_{i,T}$ to ensure the simultaneous diagonalizability. The following approximation lemma is crucial to achieve this goal.
\begin{lem}\label{lem:3.4.1}
Under the assumption that $\{\gamma_{i}(h):h\in\mathbb{Z}\}$ is absolutely summable, i.e., $\sum_{h=-\infty}^\infty |\gamma_i(h)|<\infty$, there 
exists a non-negative definite Hermitian circulant matrix $\vv{C}_{i,T}$ such that
\begin{itemize}
\item[(i)] $\vv{\Gamma}_{i,T}$ and $\vv{C}_{i,T}$ are asymptotically iso-spectral in the sense that
\begin{equation}\label{eq:3.4.10}
\begin{split}
\frac{1}{T} {\rm tr}\left((\vv{\Gamma}_{i,T}-\vv{C}_{i,T})(\vv{\Gamma}_{i,T}-\vv{C}_{i,T})^*\right)&\leq  2\left(\sum_{k=T}^\ift |\gamma_{i}(k)|+|\gamma_{i}(-k)|\right)\\
&+2\sum_{k=0}^{T}\frac{k}{T}(|\gamma_{i}(k)|^2+|\gm_{i}(-k)|^2)=o(1)
\end{split}
\end{equation}
as $T\to\infty$ and $p/T=p(T)/T\to c\in(0,\infty)$;
\item[(ii)] $\vv{C}_{i,T}$ has its eigenvalues $\{2\pi f_i(2\pi l/T):l=0,\ldots,T-1\}$, where $f_i$ is the spectral density function of the $i$-th coordinate process defined by
\begin{equation}\label{eq:3.4.11}
f_i(\lambda)=\frac{1}{2\pi}\sum_{h=-\infty}^{\infty} e^{{\bf i}\lambda h}\gamma_i(h),
\end{equation}
where ${\bf i}$ is the imaginary unit with ${\bf i}^2=-1$.
\end{itemize} 
\end{lem}
This lemma can be obtained by combination of Lemma 10 and 11 in Section 4.4 of the summary work of \red{Robert M. Gray(2006)}\cite{G06} on properties of Toeplitz matrices.

The assumption (3.4-2) ensures that equations in right-hand side of (\ref{3.4.10}) for all $i=1,\ldots,p$ can be uniformly bounded by a small o term. Then, combining (\ref{eq:3.4.10}) in the approximation lemma and (\ref{eq:2.7}) in Remark \ref{rem:2.1}, we know that the sequence of non-negative definite Hermitian circulant matrices $\vv{C}_{1,T},\ldots,\vv{C}_{p,T}$ is a good replacement of $\vv{\Gamma}_{1,T},\ldots,\vv{\Gamma}_{p,T}$ in the sense that the moment condition (4) in Assumption \ref{assm:1} is retained. Moreover, it is well-known that circulant matrices are diagonalizable simultaneously. Indeed, for a circulant matrix $C=(C_{k,j}) $ having the form:
$$
C=\left(\begin{array}{cccccc} c_0&c_1&c_2&&\cdots&c_{n-1}\\
c_{n-1}&c_0&c_1&c_2&&\vdots\\
c_{n-2}&c_{n-1}&c_0&c_1&&\vdots\\
\vdots&\vdots&\vdots&\vdots&\ddots&\vdots\\
c_2&c_3&c_4&\cdots&\cdots&c_1\\
c_1&c_{2}&c_3&\cdots&\cdots&c_0
\end{array}\right),
$$
where $\{c_0,\cdots,c_{n-1}\}\in\mathbb{C}$, by \cite{G06} Theorem 7 in Section 3.2, it has eigenvalue
\begin{equation}
\psi_m=\sum_{k=0}^{n-1} c_k e^{-2\pi{\bf i}mk/n}
\end{equation}
and the corresponding eigenvector
\begin{equation}
y^{(m)}=\frac{1}{\sqrt{n}}\left(1,e^{-2\pi{\bf i}k/n},\cdots,e^{-2\pi{\bf i}(n-1)/n}\right)',
\end{equation}
for $m=0,1\ldots,n-1$. Thus, $C$ can be expressed in the form $C=U\Psi U^*$, where $U=(y^{(0)},\ldots,y^{(n-1)})$ and $\Psi={\rm diag}\{\psi_k\}$. Since different circulant matrices share the same set of eigenvectors, they can be simultaneously diagonalizable. 

Therefore, it is natural to choose $\vv{\Sigma}_i=\vv{C}_{i,T}$ for $i=1,2,\ldots,p$. To apply results in Section 2 can be applied to $\tilde{m}_T(z)$ and then $m_T(z)$. 
Hence, in addition to previous assumption $(3.4$-$1)$ and $(3.4$-$2)$, we make the following assumptions.
\bg{assm}\lb{assm:3.4.1}
~~~
\bg{itemize}
\item[$(3.4$-$3)$] Suppose that there exists a sequence of continuous functions $\psi_j:\R^k\to\mathbb{R}$ and a sequence of real numbers $\vv{a}_1,\ldots,\vv{a}_p\in\R^k$ and a unitary matrix $\vv{U}$ such that 
\be\lb{eq:3.4.12}
U^*\vv{A}_jU={\rm diag}(\psi_j(\vv{a}_l):l=1,\ldots,p) 
\de
\item[$(3.4$-$4)$] Denote the ESD of $\vv{a}_1,\ldots,\vv{a}_p$ as $G_p$. Suppose that $G_p$ converges weakly to Borel probability $G$  as $p\to\ift$.

\item[$(3.4$-$5)$] High-dimensional setting:
$$
T\to\infty,~~~~p=p(T)\to\infty,~~~~~\hbox{such that}~\frac{p}{T}\to c\in(0,\infty).
$$

\end{itemize}
\end{assm}

Under assumptions (3.4-1)-(3.4-5), the following result can be obtained as a special case of \cite{LP17}.
\bg{prop}[Special case of \red{Liu and Paul (2017)}\cite{LP17} Theorem 2.1]\lb{thm:3.4.1}
Suppose that the linear process in (\ref{eq:3.4.1}) satisfies assumptions 
$(3.4$-$1)$---$(3.4$-$5)$. Define
\be\lb{eq:3.4.13}
h(\vv{a},\lmd):=\sum_{j=0}^\ift \psi_j(\vv{a})e^{{\bf i}j\lmd},~~~~~\lmd\in[0,2\pi].
\de
Then, with probability one, the empirical spectral distribution $\mu_T$ of $S_T$ converges weakly to a unique non-random probability measure $\mu$ with its Stieltjes transform $m(z)$ satisfying
\begin{equation}\label{eq:3.4.14}
m(z)=\int_{\R^k} \fr{1}{-z+\frac{1}{2\pi}\int_0^{2\pi}\fr{|h(\vv{a},\lmd')|^2}{cK_0(\lambda,z)+1}{\rm d}\lambda'}\d G(\vv{a}).
\end{equation}
in which $K_0:[0,2\pi]\times\mathbb{C}^+\to\mathbb{C}_+=\{z:{\rm Im}(z)\geq 0\}$ is the unique solution to the functional equation
\begin{equation}\label{eq:3.4.15}
K_0(\lambda,z)=\int_{\mathbb{R}^k}\frac{|h(\vv{a},\lambda)|^2}{-z+\frac{1}{2\pi}\int_0^{2\pi}\frac{|h(\vv{a},\lambda')|^2}{cK_0(\lambda',z)+1}{\rm d}\lambda'}{\rm d}G(\vv{a}).
\end{equation}
\end{prop}
\begin{rem}
Note that \red{Liu, Aue and Paul (2017)}\cite{LP17} obtained the existence and uniqueness of LSD for all symmetrized lag-$\tau$ auto-covariance matrices
$$
\vv{C}_\tau=\frac{1}{2T}\sum_{t=1}^T (\vv{X}_t\vv{X}^*_{t+\tau}+\vv{X}_{t+\tau}\vv{X}^*_t)
$$
for $\tau=0,1,\ldots$. Proposition \ref{thm:3.4.1} is only a special case when $\tau=0$.
\end{rem}

\begin{rem}
The main difference between assumptions in Proposition \ref{thm:3.4.1} and those in \red{Liu, Aue and Paul(2017)}\cite{LP17} Theorem 2.1 is Assumption $(3.4$-$2)$. In \red{Liu, Aue and Paul(2017)}\cite{LP17}, the authors assume the dependence of coordinate processes such that
\begin{itemize}
\item[$(3.4$-$2)'$] The coefficient matrices $\(\vv{A}_j\)_{j=0,1,\cdots}$ are diagonalizable simultaneously, non-random with $\sup_p\|\vv{A}_j\|_{op}\leq a_j<\ift$ for $j=0,1,\ldots$ and
\be\lb{eq:3.4.16}
\sum_{j=0}^\ift a_j<\infty,~~~~~\sum_{j=0}^\infty ja_j<\infty,
\de
with $\vv{A}_0=\vv{I}_p$.
\end{itemize}
By using the same universal technique displayed in Section 9 of \red{Liu, 
Aue and Paul(2017)}\cite{LP17}, our method remains valid by replacing $(3.4$-$2)$ by 
$(3.4$-$2)'$. Moreover, for $\tau=0$, the assumption $\sum_{j=0}^\infty a_j<\infty$ is sufficient to obtain (\ref{eq:3.4.14}) and (\ref{eq:3.4.15}), and the assumption $\sum_{j=0}^\infty ja_j<\infty$ thus can be removed.
\end{rem}

We leave the proof of the above proposition in Appendix. We can further relax the linear structure assumption and extend the work to general vector-valued stationary time series.

Consider a $p$-dimensional stationary time series $(\vv{X}_t)$ consisting of $p$ independent stationary time series, i.e., $\vv{X}_t=(X_{1,t},\ldots,X_{p,t})'$, where the coordinate processes $(X_{it})$ are mutually independent for $i=1,\ldots,p$. Let $\vv{X}_1,\ldots,\vv{X}_T$ be $T$ consecutive observations of $(\vv{X}_t)$. Denote the data matrix $\vv{X}=(\vv{X}_1,\ldots,\vv{X}_T)$. Consider the corresponding sample covariance matrix
$$
\vv{S}_T=\fr{1}{T}\vv{X}\vv{X^*}.
$$
Let $\lmd_1,\ldots,\lmd_p$ be eigenvalues of $\vv{S}_T$ and $\mu_T$ the 
ESD of $\vv{S}_T$. Then, we have the following result, with proof presented in Appendix.

\begin{thm}\label{thm:3.4.2}
Suppose that:
\begin{itemize}
\item[(i)] coordinate processes of $(\vv{X}_t)$ have zero mean and finite fourth moments;
\item[(ii)] there exist a sequence of numbers $\{\vv{a}_1,\ldots,\vv{a}_p\}$ in $\mathbb{R}^k$ and a bounded continuous function $f:\mathbb{R}^k\times[0,2\pi]\to[0,\infty)$ such that:
\begin{itemize}
\item[$(1)$] $f$ satisfies the following Lipschitz condition: there exists a positive constant $K>0$ such that for any $\lambda,\lambda'\in[0,2\pi]$, it holds that
\begin{equation}\label{eq:3.4.17}
\sup_{\vv{a}\in\mathbb{R}^k}|f(\vv{a},\lambda)-f(\vv{a},\lambda')|\leq K|\lambda-\lambda'|;
\end{equation}
\item[$(2)$] for each $i=1,2,\ldots,p$, the spectral density function of the $i$-th coordinate process $(X_{i,t})$ is $f(\vv{a}_i,\lambda)$ for 
$\lambda\in[0,2\pi]$;
\item[$(3)$] The empirical distribution of $\{\vv{a}_1,\ldots,\vv{a}_p\}$, denoted as $G_p$, converges weakly to a deterministic Borel probability measure $G$.
\end{itemize}
\end{itemize}
For $i=1,\ldots,p$, let 
\begin{equation}\label{eq:3.4.18}
\gamma_{i}(h):=\int_{0}^{2\pi} e^{{\bf i}h\lambda} f(\vv{a}_i,\lambda){\rm d}\lambda,~~~~~h\in\mathbb{Z},
\end{equation}
be the auto-covariance function of the $i$-th coordinate process $(X_{it})$ and define
\begin{equation}\label{eq:3.4.19}
\sigma^{(i)}_{s,t;s',t'}:=\E\left[(X_sX^*_t-\gamma_i(s-t))(X_{s'}X^*_{t'}-\gamma_i(s'-t'))\right]
\end{equation}
for $s,t,s',t'\in\{1,2,\ldots\}$. 
\begin{itemize}
\item[(iii)] there exists a positive constant $C$ such that for any positive integer $T$ and $T\times T$ matrix $B$ bounded in 
norm, it holds that
\begin{equation}\label{eq:3.4.20}
\sup_{i=1,\ldots,p}\bigg|\sum_{s,t,s',t'=1}^T B_{st}B_{s't'}\sigma^{(i)}_{s,t;s',t'}\bigg|\leq T\cdot C\|B\|_{\rm op}^2.
\end{equation}
\item[(iv)] the high dimensional setting: $p/T\to c\in(0,\infty)$.
\end{itemize}

Then, with probability one, the empirical spectral distribution $\mu_T$ of $\vv{S}_T$ converges weakly to a unique non-random probability measure $\mu$ with its Stieltjes transform $m(z)$ satisfying
\begin{equation}\label{eq:3.4.21}
m(z)=\int_{\R^k} \fr{1}{-z+\int_0^{2\pi}\fr{f(\vv{a},\lmd')}{cK_0(\lambda,z)+1}{\rm d}\lambda'}\d G(\vv{a}).
\end{equation}
in which $K_0:[0,2\pi]\times\mathbb{C}^+\to\mathbb{C}_+=\{z:{\rm Im}(z)\geq 0\}$ is the unique solution to the functional equation
\begin{equation}\label{eq:3.4.22}
K_0(\lambda,z)=2\pi\int_{\mathbb{R}^k}\frac{f(\vv{a},\lambda)}{-z+\int_0^{2\pi}\frac{f(\vv{a},\lambda')}{cK_0(\lambda',z)+1}{\rm d}\lambda'}{\rm 
d}G(\vv{a}).
\end{equation}
\end{thm}

\begin{rem}
According to discussion in Chapter 4 Section 3 \red{E.M. Stein and R. Shakarchi (2007)}\cite{Stein07}, the Lipschitz condition on the link function is only a sufficient condition to ensure the absolute summablities of auto-covariance functions of coordinate processes, which is essential to approximate the corresponding Toeplitz matrices by symmetric circulant matrices like \red{Robert M. Gray (2006)}\cite{G06}. Once other additional conditions on these stationary processes are given to imply the absolute summablities of auto-covariance functions, the Lipschitz condition can then be removed. For example, a scalar linear time series $x_t=\sum_{j=0}^{\infty}a_j z_{t-j}$ has an absolutely summable auto-covariance function if $\sum_{j=0}^\infty |a_j|<\infty$, since its auto-covariance function $\gamma(h)=\mathbb{E}(x_{t+h}x_t^*)=\sum_{j=0}^\infty a_{j+h}a_j^*$ for $h\in\mathbb{Z}=\{0,\pm 1,\pm 2,\ldots\}$, where innovation $(z_t)$ are i.i.d. real- or complex-valued random variables with $\mathbb{E}(z_t)=0$ and $\mathbb{E}|z_t|=1$.
\end{rem}

\section{Applications}
\label{sec:app}

\subsection{The realized sample covariance matrices for diffusion
  processes with time-varying spectrum} 
\label{ssec:rcv}

Suppose that $\vv{X}_t=(\vv{X}_{t,1},\ldots,\vv{X}_{t,p})'$ is a $p$-dimensional diffusion process satisfying
\begin{equation}\label{eq:3.5.1}
{\rm d} \vv{X}_t=\vv{\mu}_t{\rm d} t+\vv{\Gamma}_t{\rm d}\vv{B}_t,~~~~\vv{X}_0=\vv{0},
\end{equation}
where $\vv{\mu}_t=(\mu_{1,t},\ldots,\mu_{p,t})'$ is a $p$-dimensional drift process; $\vv{\Gamma}_t$ is a $p\times p$ matrix-valued instantaneous co-volatility process; and $\vv{B}_t$ is a standard $p$-dimensional Brownian motion. Such multi-dimensional diffusion process is often used to model the log price process in finance. In particular, financial analysts are interested in the integrated covariance (ICV) matrix 
$\Sigma^{ICV}=\int_0^1 \vv{\Gamma}_t\vv{\Gamma'}_t{\rm d}t$, where $'$ stands for the matrix transpose. One estimator of ICV matrix in practice is the realized covariance (RCV) matrix 
\begin{equation}
\vv{\Sigma}_p^{RCV}=\sum_{i=1}^n \Delta\vv{X}_i(\Delta\vv{X}_i)'
\end{equation}
where $\Delta \vv{X}_i=(\vv{X}_{\tau_{i,n},1}-\vv{X}_{\tau_{i-1,n},1},\ldots,\vv{X}_{\tau_{i,n},p}-\vv{X}_{\tau_{i-1,n},p})'$, and $\{\tau_{i,n}\}$ is a sequence of random observation times. In large
sample case, $\vv{\Sigma}^{RCV}$ is consistent to $\vv{\Sigma}^{ICV}$. (See \cite{JP98} for instance.) However, it has been shown in \cite{Zh11} that, in the high-dimensional setting when dimension $p$ grows
proportionally with the observation frequency $n$,
$\vv{\Sigma}^{RCV}$ is not a consistent estimator for
$\vv{\Sigma}^{ICV}$. This motivates us to investigate the relationship
between limiting spectra of $\vv{\Sigma}^{RCV}$ and $\vv{\Sigma}^{ICV}$.

In this section, we focus on a special class of diffusion processes whose co-volatility processes have time-constant eigenvectors and time-varying spectrum. Specifically, for $\vv{X}_t$ satisfying (\ref{eq:3.5.1}), we assume that, almost surely, there exist $p\times p$ orthogonal matrices $U$ and $V$, 
and $p$ continuous eigenvalue processes $\gamma_{1,t},\ldots,\gamma_{p,t}$ such that 
\begin{equation}\label{eq:3.5.12}
\vv{\Gamma}_t=V'{\rm diag}\{\gamma_{l,t}:l=1,\ldots, p\} U.
\end{equation}
In practice, for a multiple log price process, it is reasonable to assume that eigenvectors for its co-volatility process remain unchanged over a 
short period of time. 

It should be noticed that our model covers as a special case the class of processes considered in \red{Zheng et al. (2011)}\cite{Zh11}, in which the co-volatility process has isotropic dynamic spectrum, that is, $\vv{\Gamma}_t=\gamma_t\Lambda$ with $\gamma_t$ a scalar process and $\Lambda$ a $p\times p$ deterministic matrix. Indeed, let the singular value decomposition of $\Lambda$ be $\Lambda=V'{\rm diag}\{\lambda_1,\ldots,\lambda_p\}U$. Define $\gamma_{l,t}=\lambda_l\gamma_t$ for $l=1,\ldots,p$. Then, 
the corresponding $\vv{\Gamma}_t$ can be written in the form of (\ref{eq:3.5.12}). Thus, our model can be viewed as an extension of models in \red{Zheng et al. (2011)}\cite{Zh11} to the anisotropic dynamical spectrum case.

Firstly, consider the case of the drift process $\vv{\mu}_t\equiv\vv{0}$. Note that for $\vv{X}_t$ whose co-volatility process satisfies (\ref{eq:3.5.12}) with the drift part $\vv{\mu}_t\equiv\vv{0}$, after an orthogonal transform $V$, we have
$$
{\rm d}(V\vv{X}_t)=V{\rm d}\vv{X}_t= VU'{\rm diag}\{\gamma_{l,t}:l=1,\ldots,p\}U{\rm d}B_t={\rm diag}\{\gamma_{l,t}:l=1,\ldots,p\}{\rm d} W_t,
$$
where $W_t=UB_t$ is still a standard Brownian motion, and the RCV matrix for $U\vv{X}_t$ satisfies
$$
\sum_{i=1}^n \Delta(V\vv{X})_i(\Delta(V\vv{X}))'_i=\sum_{i=1}^n V\Delta\vv{X}_i(\Delta\vv{X}_i)'V'=V\left[\sum_{i=1}^n \Delta\vv{X}_i(\Delta\vv{X}_i)'\right]V'
$$
and thus has the same spectrum as $\vv{\Sigma}^{RCV}_p$ for $\vv{X}_t$. Therefore, it is equivalent to consider limiting behaviour of ESDs of RCV matrix for $V\vv{X}_t$. For simplicity, assume 
$\vv{\Gamma}_t={\rm diag}\{\gamma_{l,t}:l=1,\ldots,p\}$.

Let $\{\tau_{i,n}\}_{1\leq i\leq n}$ be an increasing sequence of random times. Suppose that $\{\tau_{i,n}\}$ and $\{\gamma_{i,t}\}$ are independent of $\vv{B}_t$ and therefore can be viewed as being non-random.  
\begin{equation}\label{eq:3.5.2}
\Delta \vv{X}_{i}=\int_{\tau_{i-1,n}}^{\tau_{i,n}} \Gamma_t{\rm d}\vv{B}_t\overset{d}{=}\fr{1}{\sqrt{n}}\Delta\vv{\Gamma}_{i}\vv{Z}_i,
\end{equation}
where $\overset{d.}{=}$ stands for "equal in distribution", and
$$
\Delta\vv{\Gamma}_i={\rm diag}\bigg\{\sqrt{n\int_{\tau_{i-1,n}}^{\tau_{i,n}}\gamma_{l,t}^2{\rm d} t}:l=1,\ldots,p\bigg\}
$$ 
and $\vv{Z}_1,\ldots,\vv{Z}_n$ are i.i.d. $p$-dimensional standard normal 
vectors.

Consider the RCV matrix
\begin{equation}\label{eq:3.5.3}
\vv{\Sigma}_p^{RCV}=\sum_{i=1}^n \Delta\vv{X}_i(\Delta\vv{X}_i)'\overset{d}{=}\fr{1}{n}\sum_{i=1}^n\Delta\vv{\Gamma}_{i}\vv{Z}_i\vv{Z'}_i\Delta\vv{\Gamma}_{i}=\frac{1}{n}\sum_{i=1}^n \vv{\Sigma}^{1/2}_i\vv{Z}_i\vv{Z'}_i\vv{\Sigma}^{1/2}_i.
\end{equation}
in which $\vv{\Sigma}_{i}^{1/2}=\Delta\vv{\Gamma}_{i}$, $i=1,\ldots,n$. Obviously, $\vv{\Sigma}_i$'s are simultaneously diagonalizable and have eigenvalues 
$$
w_{i,l}^{n}:=n\int_{\tau_{i-1,n}}^{\tau_{i,n}}\vv{\gamma}_{l,s}^2{\rm d} s
$$
Define $w_n:[0,1]\times[0,1]\to[0,\infty)$ by
\begin{equation}\label{eq:3.5.5}
w_n(s,r)=\sum_{i=1}^n\sum_{l=1}^p w_{i,l}^n I_{\(\fr{l-1}{p},\fr{l}{p}\r]\times\(\fr{i-1}{n},\fr{i}{n}\r]}(s,r).
\end{equation}
Then, $w_{i,l}^n=w_n(i/n,l/p)$ for any $i=1,\ldots,n$ and $l=1,\ldots,p$. By Theorem \ref{thm:0.2}, we have the following limiting theorem about the LSD of RCV matrix.
\begin{thm}[The case when $\vv{\mu}_t\equiv\vv{0}$]\lb{thm:3.5.1}
Suppose that
\begin{itemize}
\item[(i)] $\{w_n(s,r)\}$'s are uniformly bounded by $\kappa\in(0,\ift)$;
\item[(ii)] there exists a continuous bounded function $w:[0,1]\times[0,1]\to[0,\infty)$ such that
\begin{equation}\label{eq:3.5.6}
\lim_{n\to\infty}\int_0^1\int_0^1|w_n(s,r)-w(s,r)|{\rm d} s{\rm d} r=0. 
\end{equation}
\end{itemize}
Then, in the high dimensional setting $p/n\to c\in(0,\ift)$, the LSD of the realized covariance matrix $\vv{\Sgm}^{RCV}_p$ in \rf{eq:3.5.3} exists 
uniquely with its Stieltjes transform $m(z)$ satisfying
\begin{equation}\lb{eq:3.5.7}
m(z)=-\frac{1}{z}\int_0^1\frac{1}{K(s,z)+1}{\rm d} s,
\end{equation}
where $K(s,z):[0,1]\times\mathbb{C}^+\to\mathbb{C}_+$ is a unique solution to the functional equation
\begin{equation}\lb{eq:3.5.8}
K(s,z)=\int_0^1\frac{w(s,r)}{-z+c\int_0^1\frac{w(s,r)}{K(s,z)+1}{\rm d} 
s}{\rm d} r.
\end{equation}

\end{thm}
\begin{rem}\label{rem:3.5.1}
The following conditions on $\{\gm_{l,t}:l=1,\ldots,p\}$ are sufficient to ensure that conditions (i) and (ii) in the previous theorem hold:
\begin{itemize}
\item[(1)] the co-volatility process $\vv{\Gamma}_t^{(p)}$ is diagonal and independent of $\vv{B}_t$; moreover, there exists $\kappa_0<\infty$ such that $|\gamma_{l,t}^{(p)}|\leq \kappa_0$ for any $p\geq 1$, $1\leq i\leq p$ and $t\in[0,1]$; in addition, with probability one, there exists a continuous bounded function $\gamma:[0,1]\times[0,1]\to\mathbb{R}$ such that
\be\lb{eq:3.5.10}
\lim_{n\to\infty}\sum_{l=1}^p\int_{\frac{l-1}{p}}^{\frac{l}{p}}\int_0^1|\gamma_{l,r}^{(p)}-\gamma(s,r)|{\rm d} s{\rm d} r=0. 
\end{equation}

\item[(2)] the observation times $\tau_{i,n}$'s are independent of $\vv{B}_t$; moreover, there exists $\kappa_1<\infty$ such that the observation 
durations $\Delta\tau_{i,n}:=\tau_{i,n}-\tau_{i-1,n}$ satisfy
\begin{equation}\label{eq:3.5.11}
\max_{n}\max_{1\leq i\leq n}n\Delta\tau_{i,n}\leq \kappa_1;
\end{equation}
in addition, with probability one, there exists a process $v_s\in C([0,1];[0,\infty))$ such that
$$
\tau_{[ns],n}\to\Theta_s:=\int_0^s v_u{\rm d}u
$$ 
as $n\to\infty$ for all $s\in[0,1]$, where $[x]$ stands for the integer part of $x$.
\end{itemize}
In this case, the function $w(s,r)=\gamma(s,\Theta_r)^2v_r$.
\end{rem}
Theorem \ref{thm:3.5.1} can be viewed as a direct application of Theorem \ref{thm:0.2} and thus its proof is omitted.

In what follows, let us consider the case when $\vv{\mu}_t\neq\vv{0}$. Lemma 1 in Section 3.1 \red{Zheng(2011)}\cite{Zh11} shows that the drift process has no influence on the limiting behaviour of ESDs of RCV matrices once it is assumed to be uniformly bounded, which allows the process to be stochastic, c\'{a}dl\'{a}g and dependent with each other. This observation and Theorem \ref{thm:3.5.1} imply the following result.
\begin{thm}\label{thm:3.5.2}
Assume that for any $p\geq 1$, $(\vv{X}_t^{(p)})$ is a $p$-dimensional process defined in (\ref{eq:3.5.1}), with the corresponding drift process $\vv{\mu}^{(p)}_t=(\mu_{1,t}^{(p)},\ldots,\mu_{p,t}^{(p)})$, and the co-volatility process $\vv{\Gamma}_t^{(p)}={\rm diag}\{\gamma_{l,t}^{(p)}:l=1,\ldots,p\}$. Suppose that conditions (1) and (2) in Remark \ref{rem:3.5.1} are satisfied. In addition, assume that there exists $C_0<\infty$ such that
\begin{equation}\label{eq:3.5.9}
\sup_{p}\max_{1\leq i\leq p}\max_{t\in[0,1]} |\mu_{i,t}^{(p)}|\leq C_0 
\end{equation}
with probability one.

Then, as $p\to\infty$ and $p/n\to c\in(0,\infty)$, the ESD of $\vv{\Sigma}^{RCV}$ converges almost surely to an LSD with its Stieltjes 
transform satisfying (\ref{eq:3.5.7}) and (\ref{eq:3.5.8}) for $w(s,r)=\gamma(s,\Theta_r)^2v_r$.
\end{thm} 

The proof of this theorem is mainly based on the previous Theorem \ref{thm:3.5.1} and is contained in Appendix.

Recall the case in \red{Zheng et al. (2011)}\cite{Zh11}, where the co-volatility process $\vv{\Gamma}_t^{(p)}=\gamma_t^{(p)}\vv{\Lambda}^{(p)}$ with $\gamma_t^{(p)}$ a bounded scalar process and $\vv{\Lambda}^{(p)}$ a $p\times p$ deterministic matrix.  Let $\vv{\Sigma}^{(p)}=\vv{\Lambda}^{(p)}\vv{\Lambda}^{(p)'}$ and $\vv{\Sigma}^{(p),1/2}$ be the square root of $\vv{\Sigma}^{(p)}$. Without loss of generality, we assume $\vv{\Lambda}^{(p)}=\vv{\Sigma}^{(p),1/2}={\rm diag}\{\lambda_{1,p},\ldots,\lambda_{p,p}\}$ with $0\leq \lambda_{1,p}\leq\cdots\leq\lambda_{p,p}$. Suppose that
\begin{itemize}
\item[(1)] the largest eigenvalues of $\vv{\Sigma}^{(p)}$'s are uniformly bounded in $p$ and the empirical spectral distribution $G_p$ of $\vv{\Sigma}^{(p)}$ converges weakly to a deterministic probability distribution $G$; 

\item[(2)] there exists a bounded continuous function $\gm:[0,1]\to\R$ such that 
$$
\lim_{p\to\infty}\int_0^1 |\gamma_t^{(p)}-\gamma_t|\d t=0.
$$

\item[(3)] condition (2) in Remark \ref{rem:3.5.1} is satisfied. 
\end{itemize}
By \cite{Zh11} Section 2.3 Proposition 5 are satisfied so that the Stieltjes transform of LSD of $\vv{\Sigma}^{RCV}$ satisfies
$$
m(z)=-\frac{1}{z}\int \frac{1}{\lambda M(z)+1}\d G(\lambda)
$$
in which $M(z)$ is the unique solution of the following equation:
$$
M(z)=\int_0^1 \frac{w_s}{-z+c w_s\int \frac{\lambda}{\lambda M(z)+1}\d G(\lambda)}\d s,
$$
where $w_s=\gamma_{\Theta_s}^2v_s$.

To see that our method recovers the same system of equations as in \cite{Zh11}, under assumptions proposed above, the $l$-th eigenvalue $\gamma_{l,t}^{(p)}$ of $\vv{\Gamma}_t^{(p)}$ now becomes $\gamma_{l,t}^{(p)}=\lambda_{l,p}\gamma_{t}^{(p)}$. Let $G_p^{-1}$ and $G^{-1}$ be quantile functions of $G_p$ and $G$ respectively, that is, $G_p^{-1}(s)=\inf\{u|G_p(u)\geq s\}$ and $G^{-1}(s)=\inf\{u|G(u)\geq s\}$ for $s\in(0,1)$. It is well-known that $G_p$ converges weakly to $G$ if and only if $G_p^{-1}$ converges to $G^{-1}$ pointwisely on the set of continuous points of $G^{-1}$. Moreover, we claim that $\lambda_{l,p}^2=G_p^{-1}(l/p)$ for $1\leq l\leq p$. It then follows that
\begin{align*}
\sum_{l=1}^n \int_{\frac{l-1}{n}}^{\frac{l}{n}} |\lambda_l-F_n^{-1}(s)|\d s=0.
\end{align*}

Define $\gamma(s,t)=[G^{-1}(s)]^{1/2}\gamma_t$. Then,
\begin{align*}
\sum_{l=1}^p\int_{\frac{l-1}{p}}^{\frac{l}{p}}\int_0^1|\gamma_{l,r}^{(p)}-\gamma(s,r)|{\rm d} s{\rm d} r&=\sum_{l=1}^p \int_{\frac{l-1}{p}}^{\frac{l}{p}}\int_0^1 |\lambda_{l,p}\gamma_t^{p}-[G^{-1}(s)]^{1/2}\gamma_t|\d s\\
&\leq \left(\sum_{l=1}^p \frac{1}{p}\lambda_{l,p}\right)\cdot \int_0^1 |\gamma_r^{p}-\gamma_r|\d r\\
&+\sum_{l=1}^p \int_{\frac{l-1}{p}}^{\frac{l}{p}} |\lambda_{l,p}-[G^{-1}(s)]^{1/2}|\d s\left(\int_0^1 |\gamma_r|\d r\right)\\
&\leq O(1)\cdot o(1)+ \sum_{l=1}^p \int_{\frac{l-1}{p}}^{\frac{l}{p}} |\lambda_{l,p}-G^{-1}(s)|\d s\cdot O(1).
\end{align*}
Since $\lambda_{l,p}=[G_p^{-1}(l/p)]^{1/2}$ for any $1\leq l\leq p$, $G_p^{-1}$ and $G$ are uniformly bounded and $G_p^{-1}$ converges to $G^{-1}$ almost everywhere on $(0,1]$, we have by the Dominant Convergence Theorem that
\begin{align*}
\sum_{l=1}^p \int_{\frac{l-1}{p}}^{\frac{l}{p}} |\lambda_{l,p}-[G^{-1}(s)]^{1/2}|\d s&=\int_0^1 |[G_p^{-1}(s)]^{1/2}-[G^{-1}(s)]^{1/2}|\d s=o(1).
\end{align*}
Thus, all assumptions in Theorem \ref{thm:3.5.2} are satisfied and we have
$$
m(z)=-\frac{1}{z}\int_0^1\frac{1}{K(s,z)+1}\d s,
$$
in which $K(s,z)$ is the unique solution to the following equation
$$
K(s,z)=\int_0^1 \frac{G^{-1}(s)\gamma^2_{\Theta_t}v_t}{-z+c\gamma^2_{\Theta_t}v_t\int_0^1\frac{G^{-1}(s)}{K(s,z)+1}\d s}\d t.
$$
Observe that $K(s,z)=G^{-1}(s)K(z)$ with $K(z)$ being the unique solution to the equation
\begin{align*}
K(z)&=\int_0^1 \frac{\gamma^2_{\Theta_t}v_t}{-z+c\gamma^2_{\Theta_t}v_t\int_0^1\frac{G^{-1}(s)}{G^{-1}(s)K(z)+1}\d s}\d t\\
&=\int_0^1 \frac{\gamma^2_{\Theta_t}v_t}{-z+c\gamma^2_{\Theta_t}v_t\int \frac{\lambda}{\lambda K(z)+1}\d G(\lambda)}\d t.
\end{align*}
It then follows that
\begin{align*}
m(z)=-\frac{1}{z}\int_0^1\frac{1}{G^{-1}(s)K(z)+1}\d s=-\frac{1}{z}\int \frac{1}{\lambda K(z)+1}\d G(\lambda),
\end{align*}
which coincides with the system of equations obtained in \cite{Zh11}.

\subsection{The matrix-valued auto-regressive model}
\label{ssec:mar}

Consider an $m\times n$ matrix-valued time series $(\vv{X}_t)$ in $\mathbb{R}^{m\times n}$:
\begin{equation}\label{eq:3.6.1}
\vv{X}_t=\vv{A}\vv{X}_{t-1}\vv{B'}+\vv{Z}_t,
\end{equation} 
where 
\begin{itemize}
\item[(1)] $\vv{A}$ and $\vv{B}$ are $m\times m$ and $n\times n$ matrices with $\vv{B}={\rm diag}\{b_1,\ldots,b_n\}$;

\item[(2)] the innovations $\vv{Z}_t=(Z_{ij,t})_{m\times n}$ with $\{Z_{ij,t}:1\leq i\leq m,1\leq n\leq j,t=1,\ldots,T\}$ being i.i.d. random 
variables satisfying $\mathbb{E}(Z_{ij,t})=0$, $\mathbb{E}(|Z_{ij,t}|^2)=0$, $\mathbb{E}(|Z_{ij,t}|^4)<\infty$. 
\end{itemize}
This model is a matrix-valued auto-regressive model with order $1$, first proposed by \red{Chen et al. (2021)}\cite{Chen21} and applied to analysis of macroscopic economic data among countries. The problem of estimation of auto-regressive coefficient matrices $\vv{A}$ and $\vv{B}$ and the model specification testing in the low-dimensional (or, large sample) case has 
been studied comprehensively in \cite{Chen21}. However, inference in the large-dimensional case is not studied yet. In 
this paper, we focus on the relationship among singular value distributions of $\vv{X}_t$ and coefficient matrices $\vv{A}$ and $\vv{B}$ in high dimensional setting.

The following assumptions are made on the model (\ref{eq:3.6.1}). 
\begin{assm}\label{assm:3.6.1}
Suppose that
\begin{itemize}
\item[(1)] $\vv{A}$ is an $m\times m$ symmetric matrix and ${\rm tr}(\vv{A}\vv{A'})=m$;

\item[(2)] $\vv{B}={\rm diag}\{b_1,\ldots,b_n\}$;

\item[(3)] $\sup_{m,n}\|\vv{A}\|_{\rm op}\|\vv{B}\|_{\rm op}<1$, where $\|\cdot\|_{op}$ is the operator norm of matrix.

\item[(4)] Denote empirical spectral distributions of $\vv{A}\vv{A'}$ and $\vv{B}\vv{B'}$ by $G_m$ and $H_n$, respectively. Suppose that $G_m$ and $H_n$ weakly converge to probability measures $G$ and $H$, respectively.

\item[(5)] the innovations $\vv{Z}_t=(Z_{ij,t})_{m\times n}$ with $\{Z_{ij,t}:1\leq i\leq m,1\leq n\leq j,t=1,\ldots,T\}$ being i.i.d. random 
variables satisfying $\mathbb{E}(Z_{ij,t})=0$, $\mathbb{E}(|Z_{ij,t}|^2)=1$, $\mathbb{E}(|Z_{ij},t|^4)<\infty$. 

\item[(6)] High-dimensional setting
$$
n\to\infty,~~~~m=m(n)\to\infty,~~~~\hbox{such that}~\frac{m}{n}\to c\in(0,\infty).
$$
\end{itemize}
\end{assm}
Since $(\vv{A},\vv{B})$ and $(c\vv{A},c^{-1}\vv{B})$ lead to the same model in (\ref{eq:3.6.1}) for $c\neq 0$, restrictions in assumptions (1) and (2) are put to identify the model uniquely up to a sign, and have no influence on spectra of $\vv{A}\vv{A'}$ and $\vv{B}\vv{B'}$. Under these assumptions, we are going to study the LSD of $\vv{S}_t=\frac{1}{n}\vv{X}_t\vv{X'}_t$ in high-dimensional settings for any observation $\vv{X}_t$.

Let $\vv{X}^{(l)}_t$ and $\vv{Z}^{(l)}_t$ be the $l$-th column of $\vv{X}_t$ and $\vv{Z}_t$, respectively for any $l=1,\ldots,n$. Then, since $\vv{B}$ is diagonal, column processes $\vv{X}^{(l)}_t$'s are mutually independent and for any $l=1,\ldots,n$, $(\vv{X}^{(l)}_t)$ follows a vector-valued auto-regressive model with order 1 as below:
\begin{equation}\label{eq:3.6.2}
\vv{X}^{(l)}_t=b_l\vv{A}\vv{X}^{(l)}_{t-1}+\vv{Z}_{t}^{(l)}.
\end{equation}
Since $\sup_{n,m}\|\vv{A}\|\|\vv{B}\|<1$, we have $\sup_m|b_l|\|\vv{A}\|<1$ so that $(\vv{X}^{(l)}_t)$ is stationary and causal. Let $\vv{\Gamma}_{0}^{(l)}=\mathbb{E}(\vv{X}^{(l)}_t\vv{X'}^{(l)}_t)$. Then, by using the moving averaging representation of $(\vv{X}^{(l)}_t)$, we have
\begin{equation}\label{eq:3.6.4}
\vv{\Gamma}_0^{(l)}=\sum_{j=0}^\infty b_i^{2j}(\vv{A}\vv{A'})^j=\sum_{j=0}^\infty b_i^{2j}\vv{A}^{2j}
\end{equation}
for $l=1,\ldots,n$. Since $\vv{A}$ is symmetric, $\vv{\Gamma}_0^{(l)}$'s are simultaneously diagonalizable. Since columns of $\vv{X}_t$ are independent and population covariance matrices of columns are diagonalizable simultaneously, the following result about $\vv{S}_t=\frac{1}{n}\vv{X}_t\vv{X}'_t$ is immediately obtained by using method developed in Section 2. 
\begin{thm}\label{thm:3.6.2}
Suppose that $(\vv{X}_t)$ follows the model (\ref{eq:3.6.1}) and satisfies Assumption \ref{assm:3.6.1}. For any $t\geq0$ fixed, the ESD of $\vv{S}_t=\frac{1}{n}\vv{X}_t\vv{X}'_t$ converges weakly to a deterministic probability measure $F$ independent of $t$, and its Stieltjes transform $m(z)$ satisfying
\begin{equation}\label{eq:3.6.5}
m(z)=-\frac{1}{z}\int \frac{1}{1+K(a,z)}{\rm d}G(a),
\end{equation}
where the kernel function $K:[0,\infty)\times \mathbb{C}^+\to\mathbb{C}_+$ is the unique solution to the following functional equation
\begin{equation}\label{eq:3.6.6}
K(a,z)=\int \frac{1}{(1-ba)\left(-z+c\int \frac{1}{(1-ba)(K(a,z)+1)}{\rm d}G(a)\right)}{\rm d}H(b),
\end{equation}
\end{thm}
\begin{proof}
Recall that columns of $\vv{X}_t$ are independent and population covariance matrices of columns are diagonalizable simultaneously. Let $U{\rm diag}\{a_1,\ldots,a_m\}U'$ be the spectral decomposition of $\vv{A}$. Then, $\vv{\Gamma}_0^{(l)}$, defined as (\ref{eq:3.6.4}), has its eigenvalues $(1-b_l^2a_i^2)^{-1}$ for $i=1,\ldots,m$, $l=1,\ldots,n$. After verifying the moment condition (4) in Assumption \ref{assm:1}, conclusions in the theorem is obtained immediately.
\end{proof}

Theorem \ref{thm:3.6.2} captures the dependence of the singular value distribution of $\vv{X}_t$ on that of $\vv{A}$ and $\vv{B}$. The following numerical experiment illustrates our result. We consider a $400\times 600$ matrix-valued autoregressive model $\vv{X}_t$ where the coefficient matrices $\vv{A}$ and $\vv{B}$ are diagonal, whose singular value distributions are taken to be weighted sum of point masses so that $G(\d a)$ and $H(\d b)$ in Theorem \ref{thm:3.6.2} become:
$$
G=0.2\dlt_{0.25}+0.3\dlt_{0.36}+0.5\dlt_{0.49},~~~~H=0.4\dlt_{0.25}+0.4\dlt_{0.64}+0.2\dlt_{1}.
$$ 
The steps of our simulation are as follows:
\begin{itemize}
\item[(i)] Generate $T=15$ consecutive observations of $\vv{X}_t$ with standard Gaussian innovations and coefficient matrices $\vv{A}$ and $\vv{B}$ satisfying conditions above. We use the function VARMAsim in the MTS package in R for generating observations. 
\item[(ii)] Compute the eigenvalues of $\vv{S}_t=\vv{X}_t\vv{X'}_t/n$ for $t=1,5,10,15$ and plot their histograms. 

\item[(iii)] Solve numerically the system of equations (\ref{eq:3.6.5}) and (\ref{eq:3.6.6}) with distributions $G$ and $H$ given above, of which the detailed steps are explained in the last part of this subsection, and obtain the density function $\rho(x)$ of the LSD of $\vv{S}_t$ by using the fact that
$$
\rho(x)=\lim_{y\to0^+}\frac{1}{\pi}{\rm Im}\(m(x+{\bf i}y)\),
$$ 
where $m(z)$ is the Stieltjes transform of the LSD. Then, plot the density curve and compare it with the histograms of empirical eigenvalues.
\end{itemize}
\begin{figure}[ht]
\centering
\includegraphics[scale=0.25, angle=270]{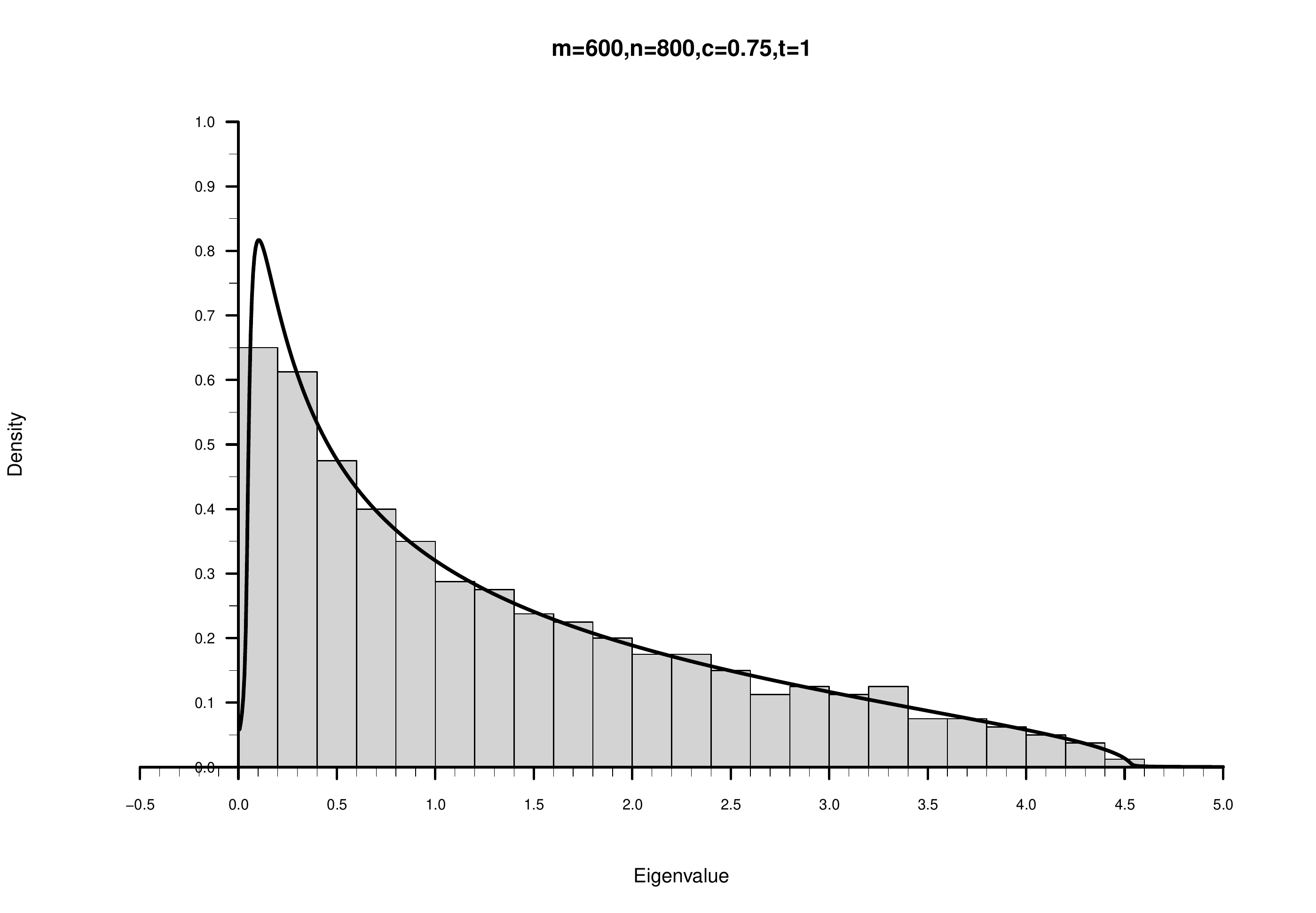}
\includegraphics[scale=0.25, angle=270]{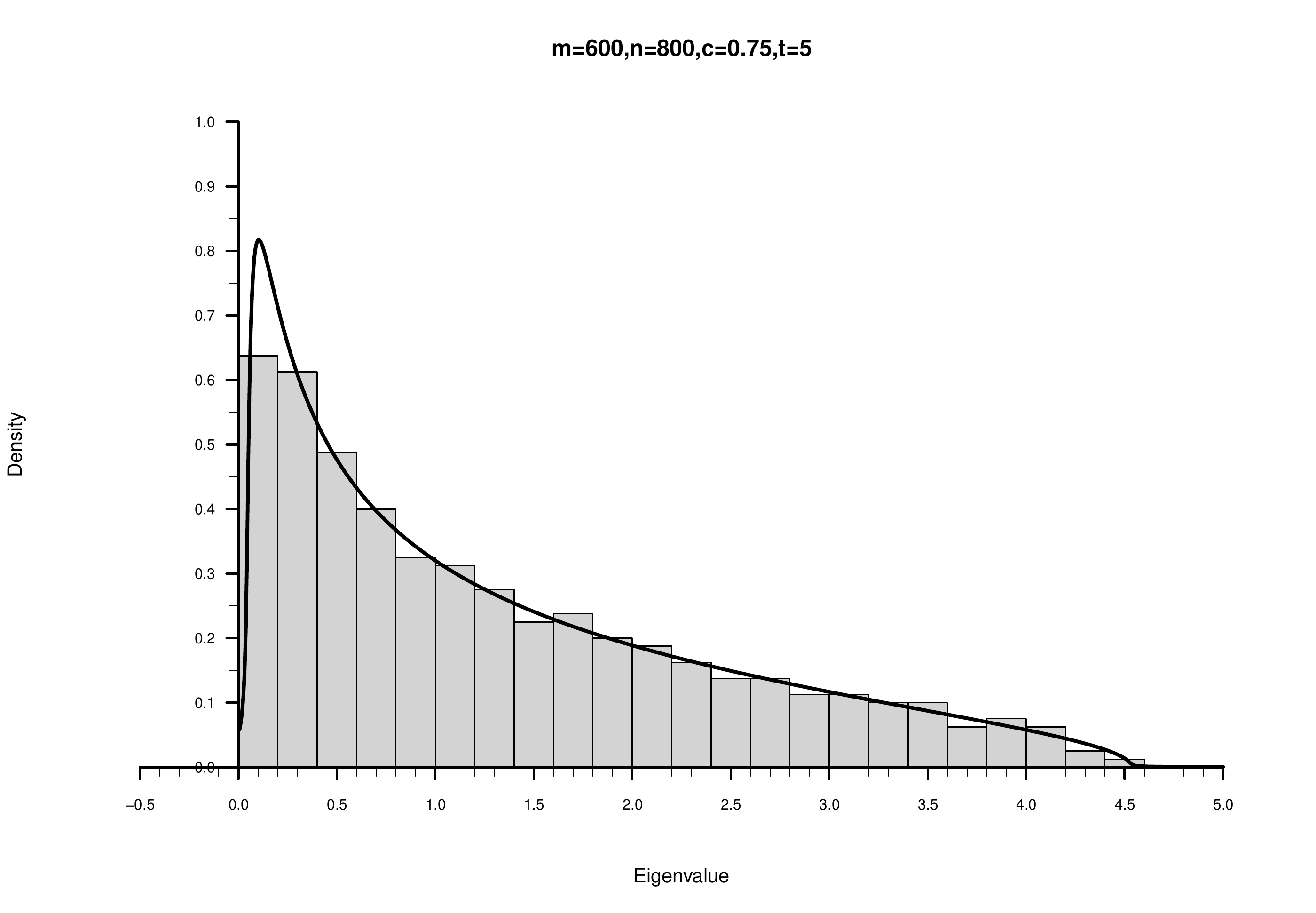}
\includegraphics[scale=0.25, angle=270]{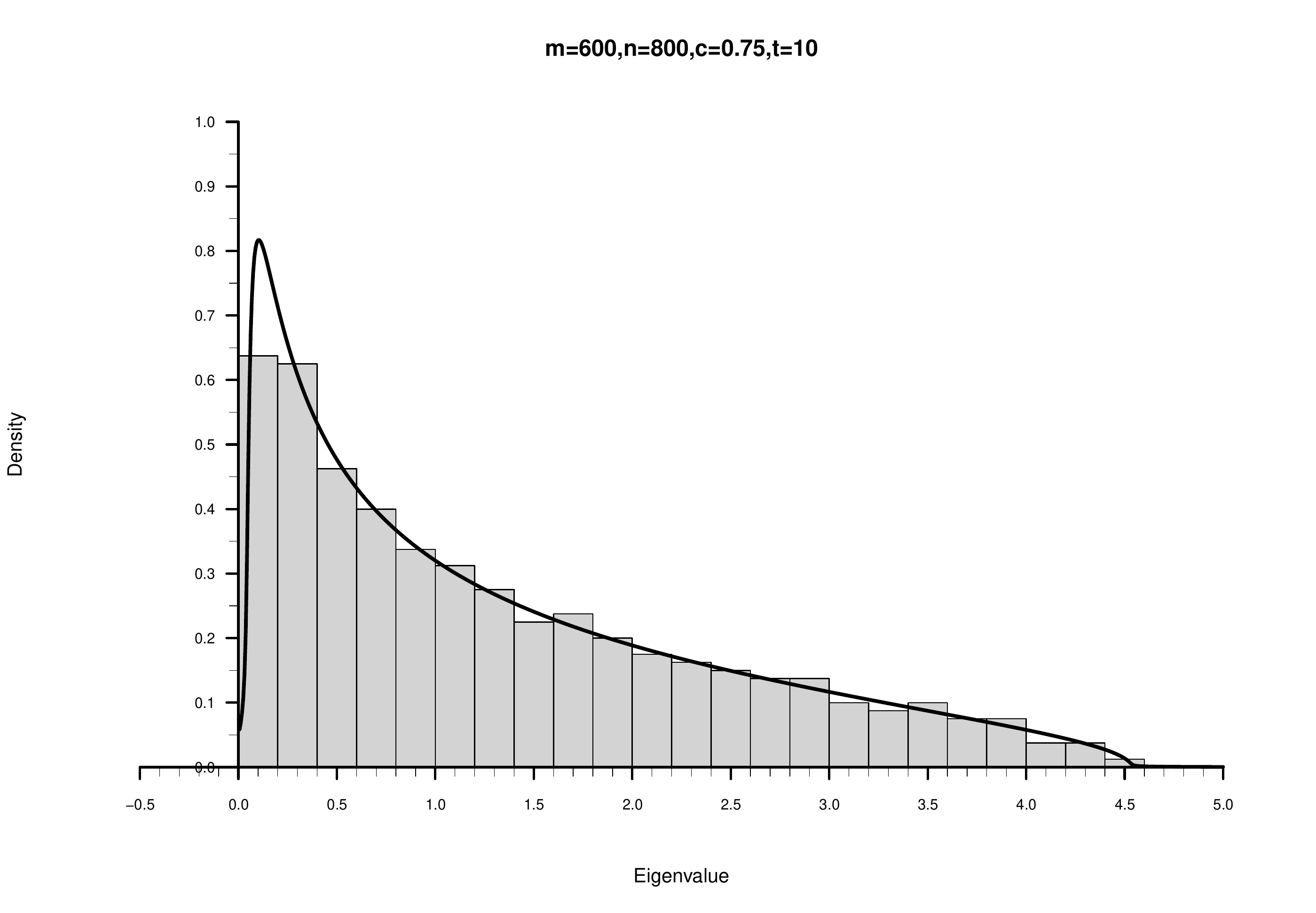}
\includegraphics[scale=0.25, angle=270]{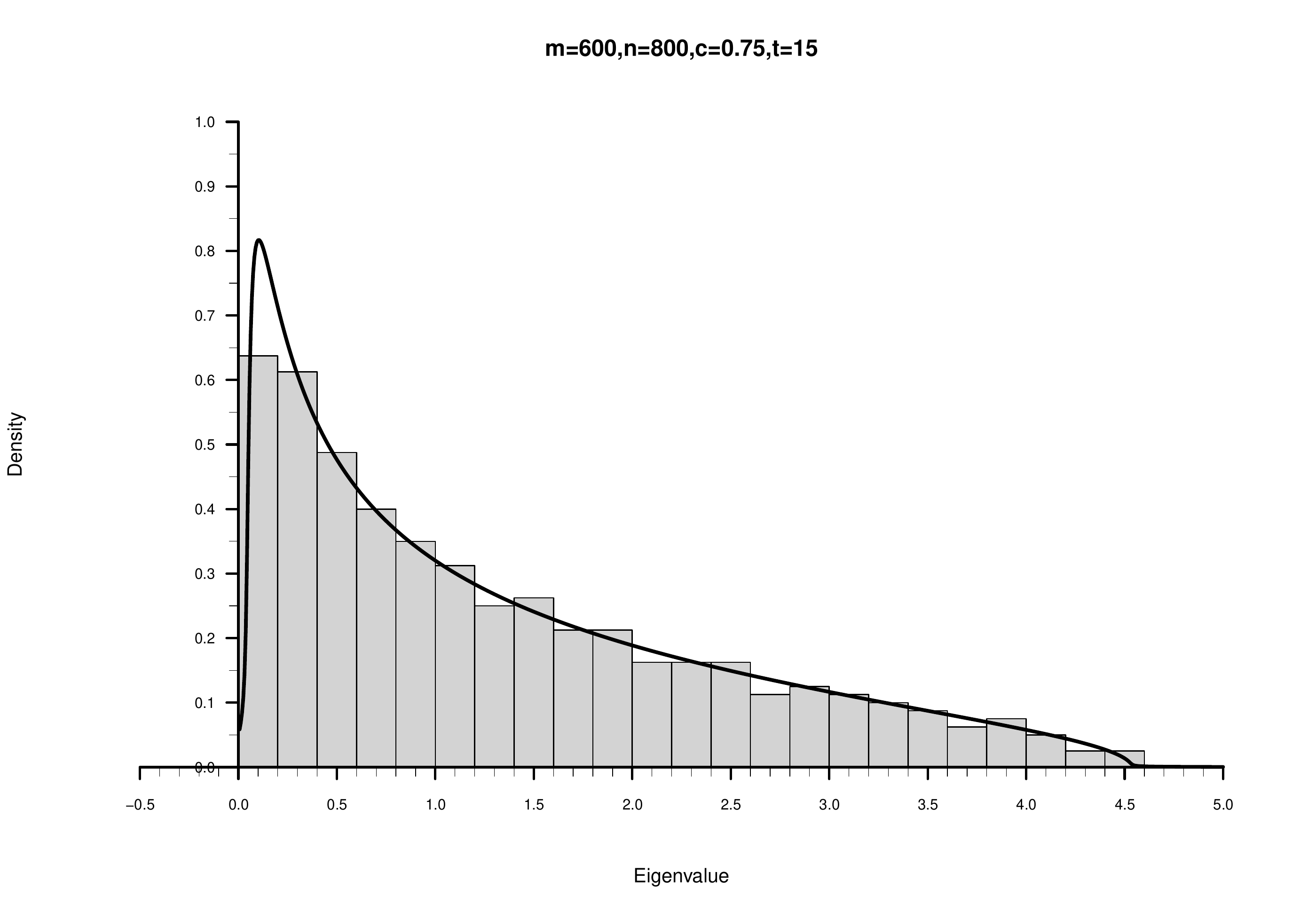}
\caption{Histogram of eigenvalues of the sample covariance matrices $\vv{S}_t$ for $t=1,5,10,15$}
\label{fig:MAR}
\end{figure}

Simulation results are shown in Figure \ref{fig:MAR}, in which the black solid curve in the four pictures represents the theoretical density function obtained by solving (\ref{eq:3.6.6}), and the grey bars are the histograms of eigenvalues of $\vv{S}_t$ for $t=1,5,10,15$, respectively. It can be seen that the empirical eigenvalue distributions match well the theoretical density and their behavior remain stable over time. 

Finally, we show our numerical method used for computing the density function in step (iii). \begin{itemize}
\item[(1)] Choose $800$ equally spaced points over the line $\{z=x+0.01{\bf i}:0\leq x\leq 5\}$, where the upper bound of real part is strictly larger than the maximum of observed empirical eigenvalues;
\item[(2)] Note that $G$ has three point masses. For each $z\in\mathbb{C}^+$ fixed, let $y_i=K(a_i,z)$ for $i=1,2,3$, where $a_i$'s are supports of $G$. Then $\vv{y}=(y_1,\ldots,y_3)$ is the unique solution of the following equation
$$
y_i=\sum_{j=1}^3\frac{q_j}{(1-a_ib_j)\(-z+c\sum_{l=1}^3\frac{p_l}{(1-a_lb_j)(1+y_l)}\)},~~~\hbox{for}~i=1,2,3
$$
over $\mathbb{C}^+$, in which $G=\sum_{i=1}^3 p_i\dlt_{a_i}$ and $H=\sum_{j=1}^3 q_j \dlt_{b_j}$. Thus, for each $z$ in the chosen grid, we are able to compute the values of $(K(a_i,z):i=1,2,3)$ by solving fixed points of the systems of equations above with initial value in $\mathbb{C}^+$. (Here, we use function FixedPoint in R package FixedPoint to find the values.) Then, from (\ref{eq:3.6.5}), we are able to obtain the values of $m(z)$'s for those $z$ in the chosen grid.

\item[(3)] Recall that for any $x_0\in\R$, if $\lim_{z\in\mathbb{C}^+,z\to x_0}{\rm Im}(m_\mu(z))$ exists, where $m_\mu$ is the Stieltjes transform of certain probability measure $\mu$, then the distribution function of $\mu$ is differentiable at $x_0$ with its derivative $\frac{1}{\pi}\lim_{z\in\mathbb{C}^+,z\to x_0}{\rm Im}(m_\mu(z))$. (See Theorem B.10 in \cite{BS10}.) Since the imaginary part of $z=x+0.01{\bf i}$'s are small enough, the value $\frac{1}{\pi}{\rm Im}(m(x+0.01{\bf i}))$ can be approximately viewed as the value of density function of the LSD at $x$.
\end{itemize}

\subsection{Finite mixture models}
\label{ssec:mix}

Statistical models based on finite mixture distributions are widely applied in many areas such as genetics, signal and image processing and machine learning \cite{FS06,Li18,McP00}. Let $\phi(\cdot|\vv{\mu},\vv{\Sigma})$ be the density function of a $p$-dimensional multivariate distribution with mean $\vv{\mu}\in\R^p$ and covariance matrix $\vv{\Sigma}$. A multivariate distribution is said to be a {\it finite mixture distribution} if its density function $f(\cdot)$ can be written in the following form:
\begin{equation}\label{eq:mix1}
f(\vv{x})=\sum_{i=1}^M \eta_i\phi(\vv{x}|\vv{\mu}_i,\vv{\Sigma}_i),
\end{equation}
where $\vv{\mu}_1,\ldots,\vv{\mu}_M\in\mathbb{R}^p$, $\vv{\Sigma}_1,\ldots,\vv{\Sigma}_M$ are $M$ non-negative definite Hermitian matrices, and the positive number $\eta_i$ is called the weight of the $i$-th component, and $\vv{\eta}=(\eta_1,\ldots,\eta_M)$ is the mixing distribution satisfying
$$
\eta_1,\ldots,\eta_M\geq0,~~~~\hbox{and}~~~\eta_1+\cdots+\eta_M=1.
$$
Let $I$ be an index random variable, taking values in $\{1,2,\ldots,M\}$ with $\P(I=i)=\eta_i$, $i=1,\ldots,M$. Consider a special case of finite mixture distribution in (\ref{eq:mix1}), in which $\vv{x}$ has the following representation
\begin{equation}
\vv{x}=\vv{\mu}_I+\vv{\Sigma}_I^{1/2}\vv{z},
\end{equation}
where $\vv{z}$ is $p$-dimensional random vector having i.i.d. entries with mean zero, unit variance and finite fourth moments, and is independent of index variable $I$.

\begin{assm}\label{assm:mix1}
Suppose that
\begin{itemize}
\item[(1)] $\vv{\Sigma}_{1p},\ldots,\vv{\Sigma}_{Mp}$ are uniformly bounded in matrix operator norm and are simultaneously diagonalizable;
\item[(2)] for each $i=1,\ldots,M$ and any sequence of $p\times p$ matrices $\vv{B}_1,\ldots,\vv{B}_n$ with $\sup_j\|\vv{B}_j\|_{\rm op}<\infty$, it holds that
$$
\frac{1}{n^3}\sum_{j=1}^n \E\left|(\vv{y}_{ij}-\vv{\mu}_i)^*\vv{B}_j(\vv{y}_{ij}-\vv{\mu}_i)-{\rm tr}(\vv{B}_j\vv{\Sigma}_{ip})\right|^2=o(1),
$$
where $\vv{y}_{i1},\ldots,\vv{y}_{in}$ are i.i.d. samples from population $\phi(\cdot|\vv{\mu}_i,\vv{\Sigma}_{ip})$;
\item[(3)] for each $i=1,\ldots,M$, the empirical distribution of eigenvalues of $\vv{\Sigma}_{ip}$, denoted as $H_{ip}$, converges weakly to a deterministic probability measure $H_i$ as $p\to\infty$.
\end{itemize}
\end{assm}

Consider a data matrix $\vv{X}_n=(\vv{x}_{1n},\ldots,\vv{x}_{nn})$ consisting of $n$ independent samples from the population with its density satisfying (\ref{eq:mix1}). Then, there exist an i.i.d. sequence $\{I_{1n},\ldots,I_{nn}\}$ of i.i.d. samples from population $I$, and an i.i.d. sequence $\{\vv{z}_1,\ldots,\vv{z}_n\}$ of $p$-dimensional random vectors having i.i.d. entries with mean zero, unit variance and finite fourth moments, such that $\{I_1,\ldots,I_n\}$ is independent of $\{\vv{z}_1,\ldots,\vv{z}_n\}$ and for $j=1,\ldots,n$, $\vv{x}_{jn}=\vv{\mu}_{I_{jn}}+\vv{\Sigma}_{I_{jn}}^{1/2}\vv{z}_{j}$.

\begin{thm}\label{thm:mixture}
Let $\vv{X}_n=(\vv{x}_{1n},\ldots,\vv{x}_{nn})$ be a data matrix consisting of $n$ independent samples from a finite mixture model in (\ref{eq:mix1}) satisfying Assumption \ref{assm:mix1}. In the high-dimensional setting:
\begin{equation}
n\to\infty,~~~~p=p(n)\to\infty,~~~~\hbox{such that}~\frac{p}{n}\to c\in(0,\infty),
\end{equation}
with probability one, the ESD of the sample covariance matrix $\vv{S}_n=\frac{1}{n}\vv{X}_n\vv{X}_n^*$ converges weakly to an LSD with its Stieltjes transform $m(z)$ satisfying 
\begin{equation}
m(z)=-\frac{1}{z}\int_0^1\frac{1}{1+K(s,z)}\d s,
\end{equation}
in which $K:(0,1)\times\mathbb{C}^+\to\mathbb{C}_+$ is the unique solution to the following functional equation
\begin{equation}\label{eq:mix-k}
K(s,z)=\sum_{i=1}^M\frac{H_i^{-1}(s)\eta_i}{-z+c\int_0^1\frac{H_i^{-1}(u)}{1+K(u,z)}\d u}
\end{equation}
and $H_i^{-1}(s)=\inf\{u|H(u)\geq s\}$ is the quantile function of $H_i$ for $i=1,\ldots,M$.
\end{thm}
The proof of Theorem \ref{thm:mixture} is given in Section \ref{ssec:mixproof}.
\begin{rem}
Though our focus is mainly on the finite mixture model, the similar result as in previous theorem can be easily extended to infinite mixture cases with population covariance matrices simultaneously diagonalizable, in which the weight probability can be any probability distribution and the quantile functions of population covariance matrices should be continuously dependent on the population indices and uniformly bounded. 
\end{rem}

\begin{rem}
It should be noticed that the LSD of $\vv{S}_n$ in general does not follow the general Mar\v{c}enko-Pastur law obtained in \cite{BS95}, though $\vv{x}_{1n},\ldots,\vv{x}_{nn}$ are i.i.d. sequence from a common population following the finite mixture distribution defined in (\ref{eq:mix1}). The reason is the same as in Section 1 \cite{Li18}, that is, $\vv{x}_{1n},\ldots,\vv{x}_{nn}$ do not satisfy the {\it weak dependence condition}: for any $j=1,\ldots,n$ and $p\times p$ matrix $\vv{B}$ with bounded operator norm, 
$$
\E\left[(\vv{x}_{jn}-\vv{\mu}_{I_j})^*\vv{B}(\vv{x}_{jn}-\vv{\mu}_{I_j})-{\rm tr}(\vv{B}\vv{\Sigma})\right]^2=o(p^2),
$$ 
where $\vv{\Sigma}={\rm cov}(\vv{x}_{jn})$. Indeed, let us consider a two-component mixture with $\vv{\mu}_1=\vv{\mu}_2=\vv{0}$, $\vv{\Sigma}_1=c_1\vv{\Sigma}$ and $\vv{\Sigma}_2=c_2\vv{\Sigma}$, where $\vv{\Sigma}$ is a $p\times p$ positive definite Hermitian matrix with ${\rm tr}(\vv{\Sigma}_1)=O(p)$ and $c_1,c_2>0$. The finite mixture distribution in (\ref{eq:mix1}) then becomes
$$
\vv{x}=\sqrt{c_1}\vv{\Sigma}^{1/2}\vv{z}_1\mathbb{I}_{\{I=1\}}+\sqrt{c_2}\vv{\Sigma}^{1/2}\vv{z}_2\mathbb{I}_{\{I=2\}},
$$ 
in which $\vv{z}_1$ and $\vv{z}_2$ are now taken to be independent $p$-dimensional real-valued standard normal vectors and are independent of index variable $I$. Then, $\vv{\Sigma}_0={\rm cov}(\vv{x})=(\eta_1c_1+\eta_2c_2)\vv{\Sigma}$ and for $\vv{B}=\vv{I}_p$, it holds that
\begin{align*}
\E(\vv{x}^*\vv{B}\vv{x}-{\rm tr}(\vv{B}\vv{\Sigma}_0))^2&={\rm var}(\vv{x}^*\vv{B}\vv{x})\\
&={\rm var}(\E(\vv{x}^*\vv{B}\vv{x}|I))+\E({\rm var}(\vv{x}^*\vv{B}\vv{x}|I))\\
&=\left(\sum_{k=1}^2\eta_k c_k^2-\left(\sum_{k=1}^2\eta_k c_k\right)^2\right)({\rm tr}(\vv{\Sigma}))^2+o(p^2).
\end{align*}
Once $\left(\sum_{k=1}^2\eta_k c_k^2-\left(\sum_{k=1}^2\eta_k c_k\right)^2\right)>0$, it is easy to see that $\E(\vv{x}^*\vv{B}\vv{x}-{\rm tr}(\vv{B}\vv{\Sigma}_0))^2=O(p^2)$ instead of $o(p^2)$, which implies that the weak dependence condition fails.

\end{rem}
\begin{rem}
We show next our results cover those in \cite{Li18} about the scale mixture model of the form $\vv{x}=w\vv{\Sigma}_p^{1/2}\vv{z}$, where $w$ is a scalar mixing random variable taking $M$ distinct non-negative values $\{u_1,\ldots,u_M\}$ with $\P(w=u_i)=\eta_i$ for each $i=1,\ldots,M$, $\vv{\Sigma}_p$ is a $p\times p$ non-negative definite Hermitian matrix, and $\vv{z}$ is a standard $p$-dimensional normal vector and is independent of $w$. Let $\vv{\Sigma}_{ip}=u_i\vv{\Sigma}_p$. It is easy to see this scale mixture model satisfying assumptions above. Note that $H_i^{-1}(s)=u_iH^{-1}(s)$ for each $i=1,\ldots,M$, where $H^{-1}(s)$ and $H_i^{-1}(s)$ be the quantile functions of LSDs of $\vv{\Sigma}_p$ and $u_i\vv{\Sigma}_p$. The second equation (\ref{eq:mix-k}) in Theorem \ref{thm:mixture} becomes:
$$
K(s,z)=H^{-1}(s)\sum_{i=1}^M\frac{\eta_i u_i}{-z+cu_i\int_0^1 \frac{H^{-1}(s)}{1+K(s,z)}\d s},
$$
which is linear with respect to $H^{-1}(s)$. Define 
$$
\tilde{K}(z)=\sum_{i=1}^M\frac{\eta_i u_i}{-z+cu_i\int_0^1 \frac{H^{-1}(s)}{1+K(s,z)}\d s}.
$$
Thus, $K(s,z)=H^{-1}(s)\tilde{K}(z)$. It then follows that
$$
\int_0^1 \frac{1}{1+K(s,z))}\d s=\int \frac{1}{1+\lambda \tilde{K}(z)}\d H(\lambda),~~~\int_0^1 \frac{H^{-1}(s)}{1+K(s,z)}\d s=\int \frac{\lambda}{1+\lambda \tilde{K}(z)}\d H(\lambda).
$$
Consequently, by using Theorem \ref{thm:mixture}, for a scale mixture model, the ESD of its sample covariance matrix converges weakly to a unique LSD with its Stieltjes transform $m(z)$ satisfying
$$
m(z)=-\frac{1}{z}\int \frac{1}{1+\lambda\tilde{K}(z)}\d H(\lambda),
$$
where $K:\mathbb{C}^+\to\mathbb{C}_+$ is the unique solution to the following functional equation
$$
\tilde{K}(z)=\sum_{i=1}^M\frac{\eta_i u_i}{-z+cu_i\int \frac{\lambda}{1+\lambda\tilde{K}(z)}\d s}.
$$
Moreover, if for some $z\in\mathbb{C}^+$, $\tilde{K}(z)=0$. Then, $m(z)=-\frac{1}{z}$ and
$$
\int \frac{\lambda}{1+\lambda\tilde{K}(z)}\d H(\lambda)=\int \lambda\d H(\lambda)=:\kappa<\infty.
$$
Then, we have
$$
0=\tilde{K}(z)=\sum_{i=1}^M \frac{\eta_i u_i}{-z+cu_i\kappa},
$$
so that $\eta_iu_i=0$ for any $i=1,\ldots,M$, since $u_i$'s are non-negative. This suggests that there must exist one point among $\{u_1,\ldots,u_M\}$ equal to zero and the corresponding mass must equal to $1$. In other words, $\vv{\Sigma}_1=\cdots=\vv{\Sigma}_M=\vv{O}$ so that $m(z)\equiv-1/z$. Consequently, $\tilde{K}(z)=0$ for some $z$ if and only if either $\vv{\Sigma}_p=\vv{O}$ or $w\equiv 0$. 

When $\vv{\Sigma}\neq\vv{O}$ and $w\neq 0$, we have $\tilde{K}:\mathbb{C}^+\to\mathbb{C}^+$. Let $q(z)=\tilde{K}(z)$ and $p(z)=-\frac{1}{z}\int \frac{\lambda}{1+\lambda\tilde{K}(z)}\d H(\lambda)$. It then follows that
$$
zm(z)=\begin{cases}
-1+\int \cfrac{p(z)t}{1+cp(z)t}\d G(t),\\
-\int\cfrac{1}{1+q(z)t}\d H(t),\\
-1-zp(z)q(z),
\end{cases}
$$
where $G$ is the discrete distribution generated by $\vv{\eta}$ and $(z,m(z),p(z),q(z))$ belongs to the region
$$
\{(z,m(z),p(z),q(z)):-(1-c)/z+cm(z)\in\mathbb{C}^+,zp(z)\in\mathbb{C}^+,q(z)\in\mathbb{C}^+,z\in\mathbb{C}^+\}.
$$
This coincides with results obtained in \cite[Sec 2.1]{Li18}.
\end{rem}

Finally, we give a simple illustrative example for Theorem \ref{thm:mixture}.
\begin{example}[Two-population Mixture model.] \label{exm:2mix}
In what follows, a two-component mixture model with one of them possessing a unit population covariance matrix. Specifically, consider the mixture model defined in (\ref{eq:mix1}) with $M=2$, $\vv{\mu}_1=\vv{\mu}_2=\vv{0}$, $\vv{\Sigma}_{1p}=\vv{I}_p$ and $\vv{\Sigma}_{2p}$ a $p\times p$ non-negative definite Hermitian matrix. The mixture weight probability $\vv{\eta}=(\eta_1,\eta_2)$ with $\eta_2=1-\eta_1$. 

Let $\vv{x}_{1n},\ldots,\vv{x}_{nn}$ be a random sample of the centred mixture model with size $n$. Define $\vv{S}_n=\frac{1}{n}\sum_{j=1}^n \vv{x}_{jn}\vv{x}^*_{jn}$. In the high-dimensional setting, i.e., 
$$
n\to\infty~~~,p=p(n)\to\infty,~~~\hbox{such that}~\frac{p}{n}\to c\in(0,\infty),
$$
the LSD of $\vv{S}_n$ exists uniquely with its Stieltjes transform satisfying 
\begin{equation}\label{eq:mix2}
-zm(z)=\int \frac{1}{1-\frac{\eta_1}{z+c zm(z)}+\left(\underline{m}(z)+\frac{\eta_1}{z+c zm(z)}\right)\lambda}\d H_2(\lambda).
\end{equation}
The derivation of equation (\ref{eq:mix2}) is provided in Section \ref{ssec:mixproof}.

When $\eta_1=0$, the model reduces to the classical sample covariance matrix case for i.i.d. samples with population covariance matrix $\vv{\Sigma}_{2p}$, in which $\alp(z)\equiv0$ and (\ref{eq:mix2}) becomes the celebrated Mar\v{c}enko-Pastur equation:
$$
-zm(z)=\int\frac{1}{1+\underline{m}(z) t}\d H_2(t).
$$
\end{example}
\begin{rem}
The way in deriving (\ref{eq:mix2}) only works the two-component mixture case, since the kernel function $K(s,z)$ in (\ref{eq:mix-k}) has no closed-form representation by $z$ and $m(z)$ if the number of mixture components is more than two. 
\end{rem}

\appendix
\section{Proof of Theorems}
\label{sec:proofs}
\subsection{Proof of Theorem \ref{thm:0.1}}

Let $\mathcal{S}_n=\vv{S}_n-z\vv{I}$ and $\vv{r}_i=\vv{x}_i/\sqrt{n}$. Denote $\mathcal{S}_{i,n}=\mathcal{S}_n-\vv{r}_i\vv{r^*}_i$.
 
Define
\begin{equation}\label{eq:a.0.7}
K=\frac{1}{n}\sum_{i=1}^n \frac{\vv{\Sgm}_i}{1+\vv{r}^*_i\mathcal{S}_{i,n}^{-1}\vv{r}_i}.
\end{equation} 
Let 
\begin{equation}\lb{eq:a.0.8}
k_n(\vv{a},z)=\frac{1}{n}\sum_{i=1}^n \frac{f(\vv{a},\vv{b}_i)}{1+\vv{r}^*_i\mathcal{S}_{i,n}^{-1}\vv{r}_i}.
\end{equation}
Then,
\begin{itemize}
\item[(1)] Since $\vv{\Sigma}_i$'s are simultaneously diagonalizable, $(K-z\vv{I})^{-1}$ commutes with $\vv{\Sigma}_i$'s.

\item[(2)] $K$ is bounded in norm, i.e.,
\begin{equation}\label{eq:.a.10}
\|K\|\leq\frac{|z|}{{\rm Im}(z)}\left(\max_{n}\max_{1\leq i\leq n}\|\vv{\Sigma}_i\|_{\rm op}+1\right)=:C_0<\infty.
\end{equation}
In fact, $K$ has its eigenvalue in the following form:
$$
\frac{1}{n}\sum_{i=1}^n \frac{\lambda_{i,l}}{1+\vv{r}^*_i\mathcal{S}_{i,n}^{-1}\vv{r}_i}.
$$ 
Notice that
\begin{equation}\label{ieq:1}
\bigg|\frac{1}{1+\vv{r}^*_i\mathcal{S}_{i,n}^{-1}\vv{r}_i}\bigg|\leq\frac{|z|}{{\rm Im}(z)}
\end{equation}
and $\max_{i,l}|\lambda_{i,l}|\leq \max_{n}\max_{1\leq i\leq n}\|\vv{\Sgm}_i\|+1$. The conclusion then follows.

Consequently, in this case, $|k_n(\vv{a},z)|\leq C_0$ for any $n$ and $z\in\mathbb{C}^+$. Moreover, we truncate $f$ so that $|f(\vv{a},\vv{b})|\leq C_0$.

\item[(3)] The eigenvalues of $K-z\vv{I}$ can be written as
\begin{equation}\label{eq:a.11}
\frac{1}{n}\sum_{i=1}^n \frac{f(\vv{a}_l,\vv{b}_i)}{1+\vv{r}^*_i\mathcal{S}_{i,n}^{-1}\vv{r}_i}-z= k_n(\vv{a}_l,z)-z.
\end{equation}

\item[(4)]$(K-z\vv{I})^{-1}$ is bounded in norm. To see this, we first observe that ${\rm Im}(1+\vv{r}^*_i\mathcal{S}_{i,n}^{-1}\vv{r}_i)^{-1}$ is negative. Indeed, since eigenvalues of $\mathcal{S}_{i,n}$ have a common imaginary part $-{\rm Im}(z)$, elementary calculation suggests that the imaginary part of $\vv{r}^*_i\mathcal{S}_{i,n}^{-1}\vv{r}_i$ is positive. Thus, the imaginary part of $(1+\vv{r}^*_i\mathcal{S}_{i,n}^{-1}\vv{r}_i)^{-1}$ must be negative. It then follows that
$$
{\rm Im}\left(\fr{1}{n}\sum_{i=1}^n \frac{\lambda_{i,l}}{1+\vv{r}^*_i\mathcal{S}_{i,n}^{-1}\vv{r}_i}-z
\right)\leq-{\rm Im}(z)<0.
$$
Therefore, $(K-zI)^{-1}$ is always bounded by ${\rm Im}(z)^{-1}$.
\end{itemize}

Next, let us prove the main theorem step by step.

(a) To prove $m_n(z)-\mathbb{E} m_n(z)\overset{\rm a.s.}{\to}0$. 

Denote $\mathbb{E}_k(\cdot)=\mathbb{E}(\cdot|\vv{r}_{k+1},\ldots,\vv{r}_n)$ with $\mathbb{E}_n(\cdot)=\mathbb{E}(\cdot)$ so that $m_n(z)=\mathbb{E}_0(m_n(z))$ and $\mathbb{E}(m_n(z))=\mathbb{E}_n(m_n(z))$. It then follows that
\begin{align*}
m_n(z)-\mathbb{E}(m_n(z))&=\sum_{k=1}^n \left(\mathbb{E}_{k-1}(m_n(z))-\mathbb{E}_k(m_n(z))\right)\\
&=\frac{1}{n}\sum_{k=1}^n\left[\E_{k-1}-\E_k\r]\left({\rm tr}\(\mathcal{S}_n^{-1}\)-{\rm tr}\left(\mathcal{S}_{k,n}^{-1}\right)\right)\\
&=:\fr{1}{n}\sum_{k=1}^n\left[\E_{k-1}-\E_k\right](\gm_k),
\end{align*}
where by the rank-1 perturbation identity
$$
\gm_k=\fr{\vv{r}^*_k\mathcal{S}_{k,n}^{-2}\vv{r}_k}{1+\vv{r}^*_k\mathcal{S}_{k,n}^{-1}\vv{r}_k}.
$$
Note that
$$
\bigg|\fr{\vv{r}^*_k\mathcal{S}_{k,n}^{-2}\vv{r}_k}{1+\vv{r}^*_k\mathcal{S}_{k,n}^{-1}\vv{r}_k}\bigg|\leq\fr{\vv{r}^*_k\((\mathcal{S}_{k,n}-{\rm Re}(z)\vv{I}_p)^2+{\rm Im}(z)^2\vv{I}_p\)^{-1}\vv{r}_k}{{\rm Im}(1+\vv{r}^*_k\mathcal{S}_{k,n}^{-1}\vv{r}_k)}=\frac{1}{{\rm Im}(z)},
$$
that is, $|\gm_k|\leq{\rm Im}(z)^{-1}$. Thus, $([\E_{k-1}-\E_k](\gm_k))$ is a bounded martingale difference sequence. 
It then follows by the Burkholder inequality (Lemma 2.12 in\cite{BS10}) that
\begin{align*}
\E|m_n(z)-\E(m_n(z))|^q&\leq K_q n^{-q}\E\(\sum_{k=1}^n|[\E_{k-1}-\E_k](\gm_k)|^2\)^{\fr{q}{2}}\\
&\leq K_q\(\fr{2}{{\rm Im}(z)}\)^{q}n^{-\fr{q}{2}}
\end{align*}
holds for any $q>2$, which suggests the almost sure convergence of $m_n(z)-\E(m_n(z))$ by the Borel-Cantelli Lemma.

(b) To prove that $\E m_n(z)-\E\fr{1}{p}{\rm tr}(K-z\vv{I})^{-1}=o(1)$.

Note that for any $n\times n$ matrix $K$ such that $K-z\vv{I}$ is 
invertible, it always holds that
\be\lb{eq:a.1}
\mathcal{S}_n^{-1}-(K-z\vv{I})^{-1}=\sum_{i=1}^n \fr{(K-z\vv{I})^{-1}\vv{r}_i\vv{r}^*_i\mathcal{S}_{i,n}^{-1}}{1+\vv{r}^*_i\mathcal{S}_{i,n}^{-1}\vv{r}_i}-(K-z\vv{I})^{-1}K\mathcal{S}_n^{-1}.
\de

In fact, from \rf{eq:a.1} with the fact that
$$
{\rm Im}(z(1+\vv{r}^*_i\vv{S}_{i,n}^{-1}\vv{r}_i))\geq {\rm Im}(z)>0,
$$
one has
\begin{equation}\label{eq:thm1b.1}
\begin{split}
\E|m_n(z)-\fr{1}{p}{\rm tr}(K-z\vv{I})^{-1}|&\leq\fr{1}{p}\E\(\sum_{i=1}^n\fr{|\vv{r}^*_i\mathcal{S}_{i,n}^{-1}(K-z\vv{I})^{-1}\vv{r}_i-n^{-1}{\rm tr}\(\mathcal{S}_{n}^{-1}(K-z\vv{I})^{-1}\vv{\Sgm}_i\)|}{1+\vv{r}^*_i\mathcal{S}_{i,n}^{-1}\vv{r}_i}\)\\
&\leq\fr{1}{pn}\fr{|z|}{{\rm Im}(z)}\sum_{i=1}^n\E|\vv{x}^*_i\mathcal{S}_{i,n}^{-1}(K-z\vv{I})^{-1}\vv{x}_i-{\rm tr}\(\mathcal{S}_{n}^{-1}(K-z\vv{I})^{-1}\vv{\Sgm}_i\)|\\
&=:\fr{1}{pn}\fr{|z|}{{\rm Im}(z)}\sum_{i=1}^n d_{i}.
\end{split}
\end{equation}
For each $i=1,\ldots,n$, let
\bg{align*}
d_{i,1}&=\E|\vv{x}^*_i\mathcal{S}_{i,n}^{-1}(K-z\vv{I})^{-1}\vv{x}_i-{\rm tr}\(\mathcal{S}_{i,n}^{-1}(K-z\vv{I})^{-1}\vv{\Sgm}_i\)|\\
d_{i,2}&=\E|{\rm tr}\(\mathcal{S}_{i,n}^{-1}(K-z\vv{I})^{-1}\vv{\Sgm}_i\)-{\rm tr}\(\mathcal{S}_{n}^{-1}(K-z\vv{I})^{-1}\vv{\Sgm}_i\)|
\end{align*}
It is easy to see that $d_{i}\leq d_{i,1}+d_{i,2}$. On one hand, since $\|\mathcal{S}_{i,n}^{-1}\|_{\rm op}\leq{\rm Im}(z)^{-1}$ and $\|(K-z\vv{I})^{-1}\|_{\rm op}\leq{\rm Im}(z)^{-1}$, by assumption \rf{eq:0.1} and the independence of $\vv{x}_i$'s, 
\bg{align*}
\fr{1}{n}\sum_{i=1}^n d_{i,1}^2\leq\fr{1}{n}\sum_{i=1}^n\E(\vv{x}^*_i\mathcal{S}_{i,n}^{-1}(K-z\vv{I})^{-1}\vv{x}_i-{\rm tr}\(\mathcal{S}_{i,n}^{-1}(K-z\vv{I})^{-1}\vv{\Sgm}_i\))^2=o(n^2).
\end{align*}
so that
\begin{equation}\label{eq:thm1b.2}
\fr{1}{n}\sum_{i=1}^n d_{i,1}\leq\(\fr{1}{n}\sum_{i=1}^n d_{i,1}^2\)^{1/2}=o(n).
\end{equation}

On the other hand, note that
$$
\mathcal{S}_n^{-1}-\mathcal{S}_{i,n}^{-1}=\fr{\mathcal{S}_{i,n}^{-1}\vv{r}_i\vv{r}^*_i\mathcal{S}_{i,n}^{-1}}{1+\vv{r}^*_i\vv{S}_{i,n}^{-1}\vv{r}_i}.
$$
Let $\vv{B}_i=(K-z\vv{I})^{-1}\vv{\Sgm}_i$. By using Lemma 2.6 in \cite{BS94}, one has
$$
|{\rm tr}\(\mathcal{S}_{i,n}^{-1}(K-z\vv{I})^{-1}\vv{\Sgm}_i\)-{\rm tr}\(\mathcal{S}_{n}^{-1}(K-z\vv{I})^{-1}\vv{\Sgm}_i\)|\leq\|\vv{\Sgm}_i\|_{op}{\rm Im}(z)^{-2}
$$
so that
\begin{equation}\label{eq:thm1b.3}
\fr{1}{pn}\sum_{i=1}^n d_{i,2}\leq{\rm Im}(z)^{-2}\fr{1}{pn}\sum_{i=1}^n \|\vv{\Sgm}_i\|_{op}=o(\fr{1}{n}).
\end{equation}
Combine (\ref{eq:thm1b.1})-(\ref{eq:thm1b.3}), we have $\E m_n(z)-\E\fr{1}{p}{\rm tr}(K-z\vv{I})^{-1}=o(1)$.

It should be noticed that the essence of the above discussion is to prove that for any $B_1,\ldots,B_n$ with norm bounded uniformly, then
\begin{equation}\label{eq:thm1b.4}
\fr{1}{n^2}\sum_{i=1}^n\E|\vv{x}^*_i\mathcal{S}_{i,n}^{-1}B_i\vv{x}_i-{\rm tr}\(\mathcal{S}_{n}^{-1}B_i\vv{\Sgm}_i\)|=o(1).
\end{equation}

(c) To prove that for any $\vv{a}\in\R^k$, $\E|k_n(\vv{a},z)-\E k_n(\vv{a},z)|\to0$ as $n\to\ift$. 

By (\ref{eq:thm1b.4}), one observes first that
$$
\E\bigg|k_n(\vv{a},z)-\fr{1}{n}\sum_{i=1}^n\fr{f(\vv{a},\vv{b}_i)}{1+\fr{1}{n}{\rm tr}(\mathcal{S}_n^{-1}\vv{\Sgm}_i)}\bigg|\leq C_0^2\fr{1}{n^2}\sum_{i=1}^n\E|\vv{x}^*_i\mathcal{S}_{i,n}^{-1}\vv{x}_i-{\rm tr}(\mathcal{S}_n^{-1}\vv{\Sgm}_i)|=o(1).
$$
Let
$$
Q_n(\vv{a},z)=\fr{1}{n}\sum_{i=1}^n\fr{f(\vv{a},\vv{b}_i)}{1+\fr{1}{n}{\rm tr}(\mathcal{S}_n^{-1}\vv{\Sgm}_i)}.
$$
It suffices to show $\E|Q_n-\E Q_n|=o(1)$. \red{We use martingale difference technique similar to (a).} Let
$$
Q_{-k,n}(\vv{a},z)=\fr{1}{n}\sum_{i\neq k}\fr{f(\vv{a},\vv{b}_i)}{1+\fr{1}{n}{\rm tr}(\mathcal{S}_{k,n}^{-1}\vv{\Sgm}_i)}.
$$
Then, since $Q_{-k,n}(\vv{a},z)$ is independent of $\vv{r}_k$, 
\bg{align*}
Q_n(\vv{a},z)-\E Q_n(\vv{a},z)&=\sum_{k=1}^n[\E_{k-1}-\E_k](Q_n(\vv{a},z))\\
&=\sum_{k=1}^n[\E_{k-1}-\E_k](Q_n(\vv{a},z)-Q_{-k,n}(\vv{a},z))=:\fr{1}{n}\sum_{k=1}^n\(\E_{k-1}-\E_k\)(\dlt_k)
\end{align*}
Note that 
$$
|{\rm tr}\((\mathcal{S}_{-k,n}^{-1}-\mathcal{S}_{n}^{-1})\vv{\Sgm}_i\)|\leq\fr{\|\vv{\Sgm}_i\|_{\rm op}}{{\rm Im}(z)}.
$$
Similar to the fact that ${\rm Im}(z(1+\vv{r}^*_i\mathcal{S}_{i,n}^{-1}\vv{r}_i))\geq{\rm Im}(z)$, one has
$$
{\rm Im}\(z(1+\frac{1}{n}{\rm tr}(\mathcal{S}_n^{-1}\vv{\Sigma}_i))\)\geq{\rm Im}(z)>0,~~~~{\rm Im}\(z(1+\frac{1}{n}{\rm tr}(\mathcal{S}_{i,n}^{-1}\vv{\Sigma}_i))\)\geq{\rm Im}(z)>0
$$
so that
$$
\bigg|\frac{1}{1+\frac{1}{n}{\rm tr}(\mathcal{S}_n^{-1}\vv{\Sigma}_i)}\bigg|\leq\frac{|z|}{{\rm Im}(z)},~~~~\bigg|\frac{1}{1+\frac{1}{n}{\rm tr}(\mathcal{S}_{i,n}^{-1}\vv{\Sigma}_i)}\bigg|\leq\frac{|z|}{{\rm Im}(z)}.
$$
It then follows that
\bg{align*}
\fr{1}{n}|\dlt_k|&=|Q_n(\vv{a},z)-Q_{-k,n}(\vv{a},z)|\\
&\leq\fr{1}{n}\sum_{i\neq k}\fr{f(\vv{a},\vv{b}_i)}{|(1+n^{-1}{\rm tr}(\mathcal{S}_{n}^{-1}\vv{\Sgm}_i))(1+n^{-1}{\rm tr}(\mathcal{S}_{-k,n}^{-1}\vv{\Sgm}_i))|}\cdot\fr{1}{n}|{\rm tr}\((\mathcal{S}_{-k,n}^{-1}-\mathcal{S}_{n}^{-1})\vv{\Sgm}_i\)|\\
&+\fr{1}{n}\fr{f(\vv{a},\vv{b}_k)}{|1+n^{-1}{\rm tr}(\mathcal{S}_n^{-1}\vv{\Sgm}_k)|}\\
&\leq\fr{1}{n}\(\frac{|z|^2C_0^2}{{\rm Im}(z)^3}+\frac{|z|C_0}{{\rm Im}(z)}\),
\end{align*}
by using that fact $|f(\vv{a},\vv{b})|\leq\max_{n}\max_{i}\|\vv{\Sigma}_i\|_{\rm op}+1=C_0$.
This suggests that $(\(\E_{k-1}-\E_k\)(\dlt_k)$ is a bounded martingale difference sequence. By the Burkholder inequality, one has
\bg{align*}
\E|Q_n(\vv{a},z)-\E(Q_n(\vv{a},z))|^q&\leq K_q n^{-q}\E\(\sum_{k=1}^n|[\E_{k-1}-\E_k](\dlt_k)|^2\)^{\fr{q}{2}}\\
&\leq K_q \(\frac{|z|^2C_0^2}{{\rm Im}(z)^3}+\frac{|z|C_0}{{\rm Im}(z)}\)^{q}n^{-\fr{q}{2}}
\end{align*}
for any $q>2$, which suggests the almost sure convergence of $Q_n(\vv{a},z)-\E(Q_n(\vv{a},z))$ for any $\vv{a}\in\R^k$. 

(d) To prove that 
$$
\E \fr{1}{p}{\rm tr}(K-z\vv{I})^{-1}=\fr{1}{p}\sum_{l=1}^p\fr{1}{\E k_n(\vv{a}_l,z)-z}+o(1).
$$

Recall that
$$
\fr{1}{p}{\rm tr}(K-z\vv{I})^{-1}=\fr{1}{p}\sum_{l=1}^p\fr{1}{k_n(\vv{a}_l,z)-z}.
$$
It then follows that
\bg{align*}
\E\bigg|\fr{1}{p}{\rm tr}(K-z\vv{I})^{-1}-\fr{1}{p}\sum_{l=1}^p\fr{1}{\E k_n(a_l,z)-z}\bigg|&\leq\fr{1}{p}\sum_{l=1}^p\E\bigg|\fr{1}{k_n(a_l,z)-z}-\fr{1}{\E k_n(a_l,z)-z}\bigg|\\
&\leq\fr{1}{p}\fr{1}{{\rm Im}(z)^2}\sum_{l=1}^p \E|k_n(\vv{a}_l,z)-\E(k_n(\vv{a}_l,z))|\\
&={\rm Im}(z)^{-2}\int \E|k_n(\vv{a},z)-\E(k_n(\vv{a},z))|\d G_p(\vv{a}).
\end{align*}
Note that $|k_n(\vv{a},z)|\leq C_0<\ift$. By (c) and the Dominant Convergence Theorem, one finds the right hand side tends to $0$.

(e) To prove that 
$$
\E \bigg|k_n(\vv{a},z)-\fr{1}{n}\sum_{i=1}^n\fr{f(\vv{a},\vv{b}_i)}{1+\fr{1}{n}{\rm tr}\((K-zI)^{-1}\vv{\Sgm}_i\)}\bigg|=o(1).
$$

Note that
\bg{align*}
&\E \bigg|k_n(\vv{a},z)-\fr{1}{n}\sum_{i=1}^n\fr{f(\vv{a},\vv{b}_i)}{1+\fr{1}{n}{\rm tr}\((K-zI)^{-1}\vv{\Sgm}_i\)}\bigg|\\
=&\E\bigg|\fr{1}{n}\sum_{i=1}^n\(\fr{f(\vv{a},\vv{b}_i)}{1+\vv{r}^*_i\vv{S}_{i,n}^{-1}\vv{r}_i}-\fr{f(\vv{a},\vv{b}_i)}{1+\fr{1}{n}{\rm tr}\((K-zI)^{-1}\vv{\Sgm}_i\)}\)\bigg|\\
\leq& \(\fr{|z|}{{\rm Im}(z)}\)^2C_0\cdot \fr{1}{n}\sum_{i=1}^n\E|\vv{r}^*_i\vv{S}_{i,n}^{-1}\vv{r}_i-\fr{1}{n}{\rm tr}\((K-zI)^{-1}\vv{\Sgm}_i\)|\\
\leq& \(\fr{|z|}{{\rm Im}(z)}\)^2C_0\cdot \fr{1}{n}\sum_{i=1}^n\E|\vv{r}^*_i\mathcal{S}_{i,n}^{-1}\vv{r}_i-\fr{1}{n}{\rm tr}(\mathcal{S}_{n}^{-1}\vv{\Sgm}_i)|\\
+&\(\fr{|z|}{{\rm Im}(z)}\)^2C_0\cdot \fr{1}{n^2}\sum_{i=1}^n\E|{\rm tr}(\mathcal{S}_{n}^{-1}\vv{\Sgm}_i)-{\rm tr}\((K-zI)^{-1}\vv{\Sgm}_i\)|=:I_1+I_2.
\end{align*}
From the assertion (\ref{eq:thm1b.4}) in step (b), one finds $I_1\to 0$ as $n\to\ift$. For $I_2$, by using similar technique in step (b) to analyze the difference $\mathcal{S}_n^{-1}\vv{\Sigma}_i-(K-z\vv{I})^{-1}\vv{\Sigma}_i$, we are able to show that for any $1\leq i\leq n$,
$$
\frac{1}{p}\E|{\rm tr}(\mathcal{S}_n^{-1}\vv{\Sigma}_i-(K-z\vv{I})^{-1}\vv{\Sigma}_i)|=o(1).
$$
Meanwhile, note that
$$
\fr{1}{p}|{\rm tr}(\mathcal{S}_n^{-1}\vv{\Sgm}_i)-{\rm tr}\((K-zI)^{-1}\vv{\Sgm}_i\)|\leq\fr{2}{{\rm Im}(z)}\sup_n\max_{1\leq i\leq n}\|\vv{\Sgm}_i\|_{\rm op}<\ift.
$$
By the Dominant Convergence Theorem, one has $I_2\to0$ as well.

Recall that
\be\lb{eq:a.12}
\fr{1}{p}{\rm tr}(K-z\vv{I})^{-1}\vv{\Sgm}_i=\fr{1}{p}\sum_{l=1}^p \fr{f(\vv{a}_l,\vv{b}_i)}{k_n(\vv{a}_l,z)-z}.
\de

(f) To prove that
$$
\E k_n(\vv{a},z)=\fr{1}{n}\sum_{i=1}^n\fr{f(\vv{a},\vv{b}_i)}{1+c_n\cdot\fr{1}{p}\sum_{l=1}^p\fr{f(\vv{a}_l,\vv{b}_i)}{\E k_n(\vv{a}_l,z)-z}}+o(1),
$$
where $c_n=p/n$.

From \rf{eq:a.12} and step (e), one has
\bg{align*}
&\E\bigg|k_n(\vv{a},z)-\fr{1}{n}\sum_{i=1}^n\fr{f(\vv{a},\vv{b}_i)}{1+c_n\cdot\fr{1}{p}\sum_{l=1}^p\fr{f(\vv{a}_l,\vv{b}_i)}{\E k_n(\vv{a}_l,z)-z}}\bigg|\\
\leq&\E \bigg|k_n(\vv{a},z)-\fr{1}{n}\sum_{i=1}^n\fr{f(\vv{a},\vv{b}_i)}{1+\fr{1}{n}{\rm tr}\((K-zI)^{-1}\vv{\Sgm}_i\)}\bigg|\\
+&\E \bigg|\fr{1}{n}\sum_{i=1}^n\fr{f(\vv{a},\vv{b}_i)}{1+\fr{1}{n}{\rm 
tr}\((K-zI)^{-1}\vv{\Sgm}_i\)}-\fr{f(\vv{a},\vv{b}_i)}{1+c_n\cdot\fr{1}{p}\sum_{l=1}^p\fr{f(\vv{a}_l,\vv{b}_i)}{\E k_n(\vv{a}_l,z)-z}}\bigg|\\
\leq& o(1)+C\fr{1}{n}\sum_{i=1}^n\E\bigg|\fr{1}{p}{\rm tr}((K-zI)^{-1}\vv{\Sgm}_i)-\fr{1}{p}\sum_{l=1}^p\fr{f(\vv{a}_l,\vv{b}_i)}{\E k_n(\vv{a}_l,z)-z}\bigg|.
\end{align*}
in which
\bg{align*}
&\E\bigg|\fr{1}{p}{\rm tr}((K-zI)^{-1}\vv{\Sgm}_i)-\fr{1}{p}\sum_{l=1}^p\fr{f(\vv{a}_l,\vv{b}_i)}{\E k_n(\vv{a}_l,z)-z}\bigg|\\
=&\fr{1}{p}\sum_{l=1}^p\E\bigg|\fr{f(\vv{a}_l,\vv{b}_i)}{k_n(\vv{a}_l,z)-z}-\fr{f(\vv{a}_l,\vv{b}_i)}{\E k_n(\vv{a}_l,z)-z}\bigg|\\
\leq& C\fr{1}{p}\sum_{l=1}^p \E|k_n(\vv{a}_l,z)-\E k_n(\vv{a}_l,z)|=o(1).
\end{align*}
Note that
$$
\bigg|\fr{1}{p}\sum_{l=1}^p \fr{f(\vv{a},\vv{b})}{k_n(\vv{a},z)-z}\bigg|\leq{\rm Im}(z)^{-1}\max_n\max_{i}\|\vv{\Sgm}_i\|_{\rm op}<\ift.
$$
By the Dominant Convergence Theorem, one finds the claim holds.

(g) To prove that for any sequence of $\{\E k_n(\vv{a},z)\}$, there exists a subsequence $\{\E k_{n_k}(\vv{a},z)\}$ such that for any $\vv{a}\in\R^k$, $\E k_{n_k}(\vv{a},z)$ 
converges.

Recall that
\bg{align*}
|k_n(\vv{a}_1,z)-k_n(\vv{a}_2,z)|&=\bigg|\fr{1}{n}\sum_{i=1}^n\fr{f(\vv{a}_1,\vv{b}_i)-f(\vv{a}_2,\vv{b}_i)}{1+\vv{r}^*_i\mathcal{S}_{i,n}^{-1}\vv{r}_i}\bigg|\\
&\leq\fr{|z|}{{\rm Im}(z)}\fr{1}{n}\sum_{i=1}^n|f(\vv{a}_1,\vv{b}_i)-f(\vv{a}_2,\vv{b}_i)|\\
&=\fr{|z|}{{\rm Im}(z)}\int |f(\vv{a}_1,\vv{b})-f(\vv{a}_2,\vv{b})|\d H_n(\vv{b}).
\end{align*}
Therefore, 
\bg{align*}
|\E k_n(\vv{a}_1,z)-\E k_{n+m}(\vv{a}_1,z)|&\leq\fr{|z|}{{\rm Im}(z)}\int 
|f(\vv{a}_1,\vv{b})-f(\vv{a}_2,\vv{b})|(\d H_n(\vv{b})+\d H_{n+m}(\vv{b}))\\
&+|\E k_n(\vv{a}_2,z)-\E k_{n+m}(\vv{a}_2,z)|
\end{align*}
which suggests
\bg{align*}
&\varlimsup_{n\to\ift}\sup_{m\geq 0}|\E k_n(\vv{a}_1,z)-\E k_{n+m}(\vv{a}_1,z)|\\
\leq&\varlimsup_{n\to\ift}\sup_{m\geq 0}\fr{|z|}{{\rm Im}(z)}\int |f(\vv{a}_1,\vv{b})-f(\vv{a}_2,\vv{b})|(\d H_n(\vv{b})+\d H_{n+m}(\vv{b}))\\
+&\varlimsup_{n\to\ift}\sup_{m\geq0}|\E k_n(\vv{a}_2,z)-\E k_{n+m}(\vv{a}_2,z)|.
\end{align*}
Since 
$$
\lim_{n\to\ift}\int |f(\vv{a}_1,\vv{b})-f(\vv{a}_2,\vv{b})|\d H_n(\vv{b})\to\int |f(\vv{a}_1,\vv{b})-f(\vv{a}_2,\vv{b})|\d H(\vv{b}),
$$
one finds
\bg{align*}
\varlimsup_{n\to\ift}\sup_{m\geq 0}|\E k_n(\vv{a}_1,z)-\E k_{n+m}(\vv{a}_1,z)|\leq\fr{2|z|}{{\rm Im}(z)}\int |f(\vv{a}_1,\vv{b})-f(\vv{a}_2,\vv{b})|\d H(\vv{b}).
\end{align*}
once $\E k_n(\vv{a}_2,z)$ converges. Since the left hand side is independent of $\vv{a}_2$, let $\vv{a}_2\to\vv{a}_1$. One finds by the Dominant Convergence Theorem that the right hand side tends to $0$, which implies the convergence of $\E k_n(\vv{a}_1,z)$. 

Therefore, for any countable dense subset $D$ in $\R^k$, one can find a subsequence such that for any $\vv{a}\in D$, $\E k_{n_k}(\vv{a},z)$ converges. Even without the uniform continuity, argument above is enough to ensure the converges of $\E k_{n_k}(\vv{a},z)$ on the whole $\R^k$. Moreover, the limit is still continuous and bounded w.r.t. $\vv{a}$.

(h) To prove that $m(z)$ is a Stieltjes transform.

Recall that for any sequence $\E k_n(\vv{a},z)$ such that it converges 
on $\R^k$, as $n\to\ift$, its limit denoted by $\hat{K}(\vv{a},z)$ must satisfy
\begin{equation}\label{eq:thm1h.1}
\hat{K}(\vv{a},z)=\int \fr{f(\vv{a},\vv{b})}{1+c\int \fr{f(\vv{a},\vv{b})}{\hat{K}(\vv{a},z)-z}\d G(\vv{a})}\d H(\vv{b})
\end{equation}
with ${\rm Im}(\hat{K}(\vv{a},z))\leq 0$ and ${\rm Im}(z^{-1}\hat{K}(\vv{a},z))\leq 0$. It then follows that $m_n(z)$ converges to $m(z)$ such that
$$
m(z)=\int \fr{1}{\hat{K}(\vv{a},z)-z}\d G(\vv{a}).
$$
Let $K(\vv{a},z)=-\frac{1}{z}\hat{K}(\vv{a},z)$, and we obtain (\ref{eq:0.5}) and (\ref{eq:0.6}). 

We prove that $m(z)$ is a Stieltjes transform of certain probability 
measure by showing that $\lim_{v\to\ift}{\bf i}v m({\bf i} v)=-1$, where ${\bf i}$ stands for the imaginary unit. It then follows that the corresponding $\mu_{\vv{S}_n}$ converges weakly to that limiting probability measure.

First, we show that $\lim_{v\to\ift}\fr{1}{{\bf i}v}\hat{K}(\vv{a},{\bf i}v)=0$ for any $\vv{a}\in\R^k$. Note that for any $z\in\mathbb{C}^+$, ${\rm Im}(z^{-1}\hat{K}(\vv{a},z))\leq 0$. It then follows that
$$
{\rm Im}\(z\(1+c\int \fr{f(\vv{a},\b{b})}{\hat{K}(\vv{a},z)-z}\d G(\vv{a})\)\)\geq{\rm Im}(z)>0.
$$
Thus,
$$
\bigg|\fr{1}{z}\hat{K}(\vv{a},z)\bigg|\leq \fr{1}{{\rm Im}(z)}\int f(\vv{a},\vv{b})\d H(\vv{b})\leq \fr{K}{{\rm Im}(z)}
$$
so that
$$
\lim_{v\to\ift}\bigg|\fr{1}{{\bf i}v}\hat{K}(\vv{a},{\bf i}v)\bigg|\leq\lim_{v\to\ift}\fr{K}{v}=0.
$$

Meanwhile, since ${\rm Im}(\hat{K}(\vv{a},z))\leq 0$, one finds ${\rm Im}(\hat{K}(\vv{a},z)-z)\leq -{\rm Im}(z)<0$ so that
$$
\bigg|\fr{{\bf i}v}{\hat{K}(\vv{a},{\bf i}v)-{\bf i}v}\bigg|\leq \fr{v}{v}=1
$$
and
$$
\lim_{v\to\ift}\fr{{\bf i}v}{\hat{K}(\vv{a},{\bf i}v)-{\bf i}v}=\lim_{v\to\ift}\fr{1}{\fr{1}{{\bf i}v}\hat{K}(\vv{a},{\bf i}v)-1}=-1.
$$
Therefore, by the Dominant Convergence Theorem, one has
$$
\lim_{v\to\ift}{\bf i}v m({\bf i}v)=\lim_{v\to\ift}\int\fr{{\bf i}v}{\hat{K}(\vv{a},{\bf i}v)-{\bf i}v}\d G(\vv{a})=-1.
$$

(i) To prove the uniqueness of $\hat{K}(\vv{a},z)$.

Let $\hat{K}_1(z),\hat{K}_2(z)$ be two solutions to (\ref{eq:thm1h.1}) such that
$$
{\rm Im}(\hat{K}_1(\vv{a},z))\leq 0,~{\rm Im}(\hat{K}_2(\vv{a},z))\leq 0,~~~\hbox{and}~~{\rm Im}\(\frac{1}{z}\hat{K}_1(\vv{a},z)\)\leq 0,~{\rm Im}\(\frac{1}{z}\hat{K}_2(\vv{a},z)\)\leq 0.
$$
Then,
\bg{align*}
\hat{K}_1(\vv{a},z)-\hat{K}_2(\vv{a},z)&=\int_{\R^m}\fr{f(\vv{a},\vv{b})c\int_{\R^k}\fr{f(\vv{a'},\vv{b})(\hat{K}_1(\vv{a'},z)-\hat{K}_2(\vv{a'},z))}{(\hat{K}_2(\vv{a'},z)-z)(\hat{K}_1(\vv{a'},z)-z)}\d G(\vv{a'})}{\(1+c\int_{\R^k}\fr{f(\vv{a},\vv{b})}{\hat{K}_1(\vv{a},z)-z}\d G(\vv{a})\)\(1+c\int_{\R^k}\fr{f(\vv{a},\vv{b})}{\hat{K}_2(\vv{a},z)-z}\d G(\vv{a})\)}\d H(\vv{b})
\end{align*}
By ${\rm Im}(\hat{K}(\vv{a},z))\leq0$, one finds ${\rm Im}(\hat{K}(\vv{a},z)-z)\leq-{\rm Im}(z)<0$ so that
$$
\bigg|\(\hat{K}(\vv{a},z)-z\)^{-1}\bigg|\leq\fr{1}{{\rm Im}(z)}.
$$
Meanwhile, by using ${\rm Im}(z^{-1}\hat{K}(\vv{a},z))\leq0$, one observes that
the imaginary parts of the denominator 
$$
\fr{1}{z\(1+c\int_{\R^k}\fr{f(\vv{a},\vv{b})}{\hat{K}(\vv{a},z)-z}\d G(\vv{a})\)}=\fr{1}{z+c\int_{\R^k}\fr{f(\vv{a},\vv{b})}{\fr{1}{z}\hat{K}(\vv{a},z)-1}\d G(\vv{a})}
$$
is positive and strictly greater that ${\rm Im}(z)$. Therefore,
$$
\bigg|\fr{1}{z\(1+c\int_{\R^k}\fr{f(\vv{a},\vv{b})}{\hat{K}(\vv{a},z)-z}\d H(\vv{a})\)}\bigg|\leq\fr{1}{{\rm Im}(z)}.
$$
and
$$
\bigg|\(1+c\int_{\R^k}\fr{f(\vv{a},\vv{b})}{\hat{K}(\vv{a},z)-z}\d G(\vv{a})\)^{-1}\bigg|\leq\fr{|z|}{{\rm Im}(z)}.
$$
It then follows that
\bg{align*}
|\hat{K}_1(\vv{a},z)-\hat{K}_2(\vv{a},z)|&\leq c\fr{|z|^2}{{\rm Im}(z)^4}\int_{\R^k} |\hat{K}_1(\vv{a'},z)-\hat{K}_2(\vv{a'},z)|\(\int_{\R^m} f(\vv{a},b)f(\vv{a'},b)\d H(\vv{b})\)\d G(\vv{a'})\\
&\leq c\fr{|z|^2}{{\rm Im}(z)^4}J(\vv{a})^{1/2}\int_{\R^k}|\hat{K}_1(\vv{a'},z)-\hat{K}_2(\vv{a'},z)| J(\vv{a}')^{1/2}\d G(\vv{a'})\\
&\leq c\fr{|z|^2}{{\rm Im}(z)^4}J(\vv{a})^{1/2}\(\int_{\R^k}|\hat{K}_1(\vv{a'},z)-\hat{K}_2(\vv{a'},z)|^2 \d G(\vv{a'})\)^{1/2}\(\int_{\R^k} J(\vv{a}')\d G(\vv{a'})\)^{1/2}
\end{align*}
where 
$$
J(\vv{a})=\int_{\R^m} f(\vv{a},\vv{b})^2\d H(\vv{b}).
$$
It then follows that
$$
\int_{\R^k}|\hat{K}_1(\vv{a},z)-\hat{K}_2(\vv{a},z)|^2\d G(\vv{a})\leq c^2\fr{|z|^4}{{\rm Im}(z)^8}\(\int_{\R^k} J(\vv{a}')\d G(\vv{a'})\)^2\int_{\R^k}|\hat{K}_1(\vv{a},z)-\hat{K}_2(\vv{a},z)|^2\d G(\vv{a}).
$$
Fix the real part of $z$ and let ${\rm Im}(z)\to\ift$. One has
$$
\int_{\R^k}|\hat{K}_1(\vv{a},z)-\hat{K}_2(\vv{a},z)|^2\d G(\vv{a})=0,
$$
when ${\rm Im}(z)$ is sufficiently large, which then implies the uniqueness.

\subsection{Proof of Theorem \ref{thm:0.2}}

Notations remain the same as those in proof of Theorem \ref{thm:0.1}. Let $\tld{\vv{\Sgm}}_i=\vv{U}^*\tld{\vv{\Lmd}}_i\vv{U}$, where 
\be\lb{eq:proof2.1}
\tld{\vv{\Lmd}}_i={\rm diag}\{f(\vv{a}_l,\vv{b}_i):1\leq l\leq p\},~~~~i=1,\ldots,n.
\de
Define
\be\lb{eq:proof2.2}
\tld{K}=\fr{1}{n}\sum_{i=1}^n\fr{\tld{\vv{\Sgm}}_i}{1+\vv{r^*}_i\mathcal{S}_{i,n}^{-1}\vv{r}_i},
\de
as an analogue of $K$ in \rf{eq:a.0.7}. It suffices to prove that
$$
\E\bigg|\fr{1}{p}{\rm tr}(K-z\vv{I})^{-1}-\fr{1}{p}{\rm tr}(\tld{K}-z\vv{I})^{-1}\bigg|=o(1).
$$
Note that
$$
(\tld{K}-z\vv{I})^{-1}-(K-z\vv{I})^{-1}=\fr{1}{n}\sum_{i=1}^n\fr{(\tld{K}-z\vv{I})^{-1}(\vv{\Sgm}_i-\tld{\vv{\Sgm}}_i)(K-z\vv{I})^{-1}}{1+\vv{r^*}_i\mathcal{S}_{i,n}^{-1}\vv{r}_i}.
$$
Consequently,
\bg{align*}
\fr{1}{p}\bigg|{\rm tr}(\tld{K}-z\vv{I})^{-1}-{\rm tr}(K-z\vv{I})^{-1}\bigg|&\leq\fr{1}{np}\fr{|z|}{{\rm Im}(z)}\sum_{i=1}^n|{\rm tr}((\tld{K}-z\vv{I})^{-1}(\vv{\Sgm}_i-\tld{\vv{\Sgm}}_i)(K-z\vv{I})^{-1})|\\
&\leq\fr{1}{np}\fr{|z|}{{\rm Im}(z)^3}\sum_{i=1}^n\sum_{l=1}^p |f_n(\vv{a}_l,\vv{b}_i)-f(\vv{a}_l,\vv{b}_i)|,
\end{align*}
since $\|(K-z\vv{I})^{-1}\|\leq{\rm Im}(z)^{-1}$ and $\|(\tld{K}-z\vv{I})^{-1}\|\leq{\rm Im}(z)^{-1}$. The conclusion then follows.

\subsection{Proof of The Universality of Generalized Sample Covariance Matrices}
Let
$$
\vv{S}_n=\frac{1}{n}\vv{A}_p^{1/2}\vv{Z}_n\vv{B}_n\vv{Z}^*_n\vv{A}_p^{1/2}.
$$
According to Section 3.2, $a_0:=\sup_{p}\|\vv{A}_p\|_{op}<\infty$ and $b_0:=\sup_p\|\vv{B}_n\|_{op}<\infty$. After truncating, centring and rescaling of entries of $\vv{Z}_n$ as in \cite{ZL06}, we assume that $(Z_{ij})$ of $\vv{Z}_n$ are i.i.d. with $|Z_{ij}|\leq n^{1/4}\varepsilon_p$, $\mathbb{E}(Z_{ij})=0$, $\mathbb{E}|Z_{ij}|^2=1$ and $\mathbb{E}|Z_{ij}|^4<C$ for certain $C<\infty$, where $\varepsilon_p p^{1/4}\to0$. Let $\tilde{Z}=(\tilde{Z}_{ij})_{p\times n}$ be a Gaussian matrix with i.i.d. entries satisfying $\mathbb{E}(\tilde{Z}_{ij})=0$, $\mathbb{E}|\tilde{Z}_{ij}|^2=1$. Define
$$
\tilde{S}_n=\frac{1}{n}\vv{A}_p^{1/2}\tilde{Z}_n\vv{B}_n\tilde{Z}_n\vv{A}_p^{1/2}.
$$
We prove the universality of the generalized sample covariance matrices by showing that
\begin{itemize}
\item[(a)] $m_n(z)-\E m_n(z)\overset{\rm a.s.}{\to}0$ as $n\to\infty$ when entries of $\vv{Z}_n$ are i.i.d. standardized variables with arbitrary distribution;
\item[(b)] the difference
$$
\mathbb{E}\left(\frac{1}{p}{\rm tr}(\vv{S}_n-z\vv{I})^{-1}\right)-\mathbb{E}\left(\frac{1}{p}{\rm tr}(\tilde{S}_n-z\vv{I})^{-1}\right)\to0,
$$
as $n,p\to\infty$.
\end{itemize}

\begin{proof}[\it Proof of The Universality]
(a) The proof is split into two steps: 1) obtain a concentration inequality by using the McDiamid's inequality; 2) complete the proof by the Borel-Cantelli lemma.

\begin{lem}[McDiarmid inequality\cite{McD}]\label{lem:McD}
Let $X_1,\ldots,X_m$ be independent random vectors taking values in $\mathcal{X}$. Suppose that function $f:\mathcal{X}^m\to\R$ satisfies there exist $c_1,\ldots,c_m>0$ such that for any $\vv{x}_1,\ldots,\vv{x}_m$ and $\vv{x'}_i$ in $\mathcal{X}$,
$$
|f(\vv{x}_1,\ldots,\vv{x}_i,\ldots,\vv{x}_m)-f(\vv{x}_1,\ldots,\vv{x'}_i,\ldots,\vv{x}_m)|\leq c_i.
$$
Then, for any $\varepsilon>0$,
$$
\mathbb{P}\left(|f(\vv{X}_1,\ldots,\vv{X}_m)-\mathbb{E}f(\vv{X}_1,\ldots,\vv{X}_m)|>\varepsilon\right)\leq 2\exp\left(-\frac{2\varepsilon^2}{\sum_{i=1}^m c_i^2}\right).
$$
\end{lem}
Since $m_n(z)$ is a complex-valued function, we thus apply the above inequality separately to its real and imaginary parts to derive its concentration bounds. To see this, denote the columns of $\vv{Z}_n$ by $\vv{z}_1,\ldots,\vv{z}_n$ and define $\vv{Z}_{(i)}$ by replacing the $i$-th column of $\vv{Z}_n$ by $\vv{0}$. Let $\vv{e}_i$ be the $n$-dimensional basis vector with its $i$-th entry $1$ and others $0$. Then,
$$
\vv{Z}_{(i)}=\vv{Z}_n-\vv{Z}_n\vv{e}_i\vv{e^*}_i=\vv{Z}_n-\vv{z}_{(i)}\vv{e^*}_i.
$$
Define 
$$
S_{(i)}=\frac{1}{n}\vv{A}_p^{1/2}\vv{Z}_{(i)}\vv{B}_n\vv{Z^*}_{(i)}\vv{A}_p^{1/2}.
$$
Let $\vv{\alp}_i=\vv{A}_p^{1/2}\vv{z}_i$, $\vv{\bt}_i=\frac{1}{n}\vv{A}_p^{1/2}\vv{Z}_{(i)}\vv{B}_n\vv{e}_i$ and $w_i=\frac{1}{n}\vv{e^*}_i\vv{B}_n\vv{e}_i$. It then follows that
$$
\vv{S}_n=S_{(i)}+\vv{\alp}_i\vv{\bt^*}_i+\vv{\bt}_i\vv{\alp^*}_i+w_i\vv{\alp}_i\vv{\alp^*}_i.
$$
Decompose $\vv{\alp}_i\vv{\bt^*}_i+\vv{\bt}_i\vv{\alp^*}_i=\vv{\mu}_i\vv{\mu^*}_i-\vv{\nu}_i\vv{\nu^*}_i$ by using $\vv{\mu}=\frac{1}{2}(\vv{\alp}_i+\vv{\bt}_i)$ and $\vv{\nu}=\frac{1}{2}(\vv{\alp_i}-\vv{\bt}_i)$. 
Let $R_{1,(i)}=S_{(i)}+\vv{\mu}_i\vv{\mu^*}_i$ and $R_{2,(i)}=R_{1,(i)}-\vv{\nu}_i\vv{\nu^*}_i$. Then,
\begin{align*}
{\rm tr}(\vv{S}_n-z\vv{I})^{-1}-{\rm tr}(S_{(i)}-z\vv{I})^{-1}&=\left({\rm tr}(\vv{S}_n-z\vv{I})^{-1}-{\rm tr}(R_{2,(i)}-z\vv{I})^{-1}\right)\\
&+\left({\rm tr}(R_{2,(i)}-z\vv{I})^{-1}-{\rm tr}(R_{1,(i)}-z\vv{I})^{-1}\right)\\
&+\left({\rm tr}(R_{1,(i)}-z\vv{I})^{-1}-{\rm tr}(S_{(i)}-z\vv{I})^{-1}\right),
\end{align*}
in which terms of differences are all bounded by $\frac{1}{{\rm Im}(z)}$ by Lemma 2.6 in \cite{BS95}. Thus, denote $S'_n$ by replacing $\vv{z}_i$ by $\vv{z'}_i$ in $\vv{Z}_n$ to derive $\vv{Z'}_n$, one has 
$$
\bigg|\frac{1}{p}{\rm tr}(\vv{S}_n-z\vv{I})^{-1}-\frac{1}{p}{\rm tr}(S'_{n}-z\vv{I})^{-1}\bigg|\leq\frac{6}{p{\rm Im}(z)}.
$$
It then follows by the McDiarmid inequality that for any $\varepsilon>0$,
$$
\mathbb{P}\left(|m_n(z)-\mathbb{E}m_n(z)|>\varepsilon\right)\leq 4\exp\left(-\frac{p^2{\rm Im}(z)^2\varepsilon^2}{18 n}\right).
$$
Therefore, by Borel-Cantelli Lemma, we see $m_n(z)-\mathbb{E}m_n(z)\overset{\rm a.s.}{\to}0$.

(b) We next prove that the difference
$$
\mathbb{E}\left(\frac{1}{p}{\rm tr}(\vv{S}_n-z\vv{I})^{-1}\right)-\mathbb{E}\left(\frac{1}{p}{\rm tr}(\tilde{S}_n-z\vv{I})^{-1}\right)\to0,
$$
as $n,p\to\infty$. To see this, we apply the Lindeberg Principle developed in \cite{Ch06}. Denote
$$
Z_{11},\ldots,Z_{1n},Z_{2n},\ldots,Z_{pn}~~~~\hbox{by}~Y_1,\ldots,Y_{pn},
$$
and
$$
\tilde{Z}_{11},\ldots,\tilde{Z}_{1n},\tilde{Z}_{2n},\ldots,\tilde{Z}_{pn}~~~~\hbox{by}~\tilde{Y}_1,\ldots,\tilde{Y}_{pn}.
$$
Let $m=pn$. For each $i=0,1,\ldots,m$, define
$$
X_i=(Y_1,\ldots,Y_{i-1},Y_i,\tilde{Y}_{i+1},\ldots,\tilde{Y}_m)
$$
and
$$
X_i^0=(Y_1,\ldots,Y_{i-1},0,\tilde{Y}_{i+1},\ldots,\tilde{Y}_m).
$$
Let $f(\vv{x})=p^{-1}{\rm tr}(n^{-1}\vv{A}_p^{1/2}\vv{X}\vv{B}_n\vv{X^*}\vv{A}_p^{1/2}-z\vv{I})^{-1}$, where $\vv{X}$ is a $p\times n$ matrix obtained by converting the $m\times 1$ vector $\vv{x}$. Then, $f(X_m)=p^{-1}{\rm tr}(\vv{S}_n-z\vv{I})^{-1}$, $f(X_0)=p^{-1}{\rm tr}(\tilde{S}_n-z\vv{I})^{-1}$ and
$$
\mathbb{E}\left(\frac{1}{p}{\rm tr}(S_n-z\vv{I})^{-1}\right)-\mathbb{E}\left(\frac{1}{p}{\rm tr}(\tilde{S}_n-z\vv{I})^{-1}\right)=\sum_{i=1}^m 
\mathbb{E}\left[f(X_i)-f(X_{i-1})\right].
$$
Since $f$ is analytic, its third order Taylor expansion with integral reminder yield:
$$
f(X_i)=f(X_i^0)+Y_i\partial_if(X_i^0)+\frac{1}{2}Y_i^2\partial_i^2 f(X_i^0)+\frac{1}{2}Y_i^3\int_0^1(1-t)^2\partial_i^3 f(X_i(t)){\rm d}t
$$
and
$$
f(X_{i-1})=f(X_i^0)+\tilde{Y}_i\partial_if(X_i^0)+\frac{1}{2}\tilde{Y}_i^2\partial_i^2 f(X_i^0)+\frac{1}{2}\tilde{Y}_i^3\int_0^1(1-t)^2\partial_i^3 f(\tilde{X}_i(t)){\rm d}t,
$$
where $\partial_i^r$ is the $r$-fold partial derivative with respect to the $i$-th coordinate $(r=1,2,3)$ and
$$
X_i(t)=(Y_1,\ldots,Y_{i-1},tY_i,\tilde{Y}_{i+1},\ldots,\tilde{Y}_m)
$$
and
$$
\tilde{X}_i(t)=(Y_1,\ldots,Y_{i-1},t\tilde{Y}_i,\tilde{Y}_{i+1},\ldots,\tilde{Y}_m).
$$
It then follows by the Lindeberg principle that
$$
\sum_{i=1}^m \mathbb{E}\left[f(X_i)-f(X_{i-1})\right]=\frac{1}{2}\sum_{i=1}^m\mathbb{E}\left[Y_i^3\int_0^1(1-t)^2\partial_i^3 f(X_i(t)){\rm d}t-\tilde{Y}_i^3\int_0^1(1-t)^2\partial_i^3 f(\tilde{X}_i(t)){\rm d}t\right].
$$
Thus, it suffices to find reasonable bound of the $3$-fold derivatives $\partial_i^3 f(\tilde{X}_i(t))$'s. Define $P_n=(\vv{S}_n-z\vv{I})^{-1}$. From $\frac{\partial P_n}{\partial Z_{ij}}=-P_n\frac{\partial \vv{S}_n}{\partial Z_{ij}} P_n$, we have
$$
\frac{1}{p}{\rm tr}\left(\frac{\partial^3 P_n}{\partial Z_{ij}^3}\right)=\frac{6}{p}{\rm tr}\left(\frac{\partial\vv{S}_n}{\partial Z_{ij}}P_n\frac{\partial^2 \vv{S}_n}{\partial Z_{ij}^2}P_n^2\right)-\frac{6}{p}{\rm tr}\left(\frac{\partial S_n}{\partial Z_{ij}}P_n\frac{\partial \vv{S}_n}{\partial Z_{ij}}P_n\frac{\partial\vv{S}_n}{\partial Z_{ij}}P_n^2\right)-\frac{1}{p}{\rm tr}\left(\frac{\partial^3\vv{S}_n}{\partial Z_{ij}^3}P_n^2\right),
$$
where by letting $\vv{r}_j=\vv{A}_p^{1/2}\vv{Z}_n\vv{B}_n\tilde{e}_j$ and $\vv{\xi}_i=\vv{A}_p^{1/2}e_i$ for $e_i$ as well as $\tilde{e}_j$ a $p\times 1$ and $n\times 1$ unit vector, respectively, 
\begin{align*}
\frac{\partial \vv{S}_n}{\partial Z_{ij}}&=\frac{1}{n}\left(\vv{A}_p^{1/2}\vv{Z}\vv{B}_n\tilde{e}_je^*_i\vv{A}_p^{1/2}+\vv{A}_p^{1/2}e_i\tilde{e}^*_jB_nZ^*_n\vv{A}_p^{1/2}\right)=\frac{1}{n}(\vv{r}_j\vv{\xi}^*_i+\vv{\xi}_i\vv{r}^*_j);\\
\frac{\partial^2\vv{S}_n}{\partial Z_{ij}^2}&=\frac{2}{n}\vv{A}_p^{1/2}e_i\tilde{e}^*_j\vv{B}_n\tilde{e}_je^*_i\vv{A}_p^{1/2}=\frac{2}{n}b_{jj}\vv{\xi}_i\vv{\xi}^*_i;~~~~~\frac{\partial^3\vv{S}_n}{\partial Z_{ij}^3}=\vv{O},
\end{align*}
in which $\vv{O}$ is a $p\times p$ matrix with all entries equal to $0$.
It then follows that 
\begin{align*}
\frac{6}{p}{\rm tr}\left(\frac{\partial\vv{S}_n}{\partial Z_{ij}}P_n\frac{\partial^2\vv{S}_n}{\partial Z_{ij}^2}P_n^2\right)&=\frac{12b_{jj}}{pn^2}{\rm tr}(\vv{\xi^*}_i P_n^2\vv{r}_j\vv{\xi^*}_iP_n\vv{\xi}_i)+\frac{12b_{jj}}{pn^2}{\rm tr}(\vv{\xi^*}_i P_n^2\vv{\xi}_i\vv{r^*}_jP_n\vv{\xi}_i)=:\eta_{1,n}+\eta_{2,n}
\end{align*}
and
\begin{align*}
\frac{6}{p}{\rm tr}\left(\frac{\partial\vv{S}_n}{\partial Z_{ij}}P_n\frac{\partial\vv{S}_n}{\partial Z_{ij}}P_n\frac{\partial\vv{S}_n}{\partial Z_{ij}}P_n^2\right)&=\frac{1}{pn^3}{\rm tr}((\vv{r}_j\vv{\xi^*}_i+\vv{\xi}_i\vv{r^*}_j)P_n(\vv{r}_j\vv{\xi^*}_i+\vv{\xi}_i\vv{r^*}_j)P_n(\vv{r}_j\vv{\xi^*}_i+\vv{\xi}_i\vv{r^*}_j)P_n^2)\\
&=\frac{2}{pn^3}{\rm tr}((\vv{r^*}_jP_n\vv{\xi}_i)^2\vv{r^*}_jP_n^2\vv{\xi}_i)+\frac{2}{pn^3}{\rm tr}(\vv{r^*}_jP_n\vv{\xi}_i\vv{r^*}_jP_n\vv{r}_j\vv{\xi^*}_jP_n^2\vv{\xi}_i)\\
&+\frac{2}{pn^3}{\rm tr}(\vv{r^*}_jP_n\vv{r}_j\vv{\xi^*}_iP_n\vv{\xi}_i\vv{r^*}_jP_n^2\vv{\xi}_i)+\frac{2}{pn^3}{\rm tr}((\vv{\xi^*}_jP_n\vv{\xi}_i)\vv{r^*}_jP_n\vv{\xi}_i\vv{r^*}_jP_n^2\vv{r}_i)\\
&=:2(\eta_{3,n}+\eta_{4,n}+\eta_{5,n}+\eta_{6,n}).
\end{align*}
Using the similar argument developed in \cite[\S~3.3]{WP14}, it holds that
\begin{lem}
For $k=1,2$, $\mathbb{E}\|\vv{r}_j\|^{2k}\leq C_kp^k$ for some constant 
$C_k$.
\end{lem}
Note that for $k=1,2$, $|\eta_{k,n}|\leq\frac{12b_0}{pn^2}\left(\frac{a_0}{{\rm Im}(z)}\right)^3\|\vv{r}_j\|$. By H\"{o}lder's inequality, it holds that
$$
\mathbb{E}|Z_{ij}^3\eta_{k,n}|\leq\frac{M}{pn^2}(\mathbb{E}|Z_{ij}|^4)^{\frac{3}{4}}\(\mathbb{E}\|\vv{r}_j\|^4\)^{1/4}\leq\frac{M}{n^2p^{1/2}}.
$$
Meanwhile, for $k=3,4,5,6$, $|\eta_{k,n}|\leq\frac{1}{pn^3}\frac{a_0^3}{{\rm Im}(z)^4}\|\vv{r}_j\|^3$. Since $|Z_{ij}|\leq n^{1/4}\varepsilon_p$, it holds that
\begin{align*}
\mathbb{E}|Z_{ij}^3\eta_{k,n}|&\leq\frac{M}{n^3p}\mathbb{E}\left[|Z_{ij}|^3\|\vv{r}_j\|^3\right]\\
&\leq\frac{M}{n^3p}(\mathbb{E}|Z_{ij}|^{12})^{\frac{1}{4}}(\mathbb{E}\|\vv{r}_j\|^4)^{\frac{3}{4}}\leq\frac{Mp^{\frac{1}{2}}}{n^{\frac{9}{4}}}\varepsilon_p^3.
\end{align*}
Similar conclusion holds for $\tilde{Z}_{ij}$'s due to Gaussianity. Combining discussions above, we conclude that
$$
\sum_{i=1}^m \int_0^1 (1-t)^2\left[\mathbb{E}\bigg|Y_i^3\partial^3_i f(X_i(t))\bigg|-\mathbb{E}\bigg|\tilde{Y}_i^3\partial^3_i f(\tilde{X}_i(t))\bigg|\right]{\rm d} t\leq M\max\{p^{1/4}\varepsilon_p^3,n^{-1/2}\}\to0.
$$
Therefore, we complete the proof of step (b).
\end{proof}

\subsection{Proof of Proposition \ref{thm:3.4.1}}

We are going to complete the proof in two steps:
\begin{itemize}
\item[(a)] By the method developed in Section 2, we obtain the limiting equations for $\tilde{m}_T(z)$; 

\item[(b)] Using (\ref{eq:3.4.8}), we obtain the limiting equations for $m_T(z)$.
\end{itemize}

(a) Recall that $\tilde{m}_T(z)$ is the Stieltjes transform of ESDs for 
$$
\tilde{S}_T=\frac{1}{p}\sum_{i=1}^p \vv{X}_{(i)}\vv{X^*}_{(i)},
$$
where $\vv{X}_{(i)}$'s are independent and made by $T$ consecutive observations of coordinate processes. Moreover, for any $i=1,\ldots,p$, $(X_{i,t})$ is a linear time series satisfying
$$
X_{i,t}=\sum_{j=0}^\infty a_{i,j}Z_{i,t-j},
$$
where $a_{i,j}=\psi_j(\vv{a}_i)$ for $j=0,1,\ldots$. Then, $\vv{X}_{(i)}=(X_{i,1},\ldots,X_{i,T})'$ and $\vv{\Gamma}_{i,T}=\mathbb{E}(\vv{X}_{(i)}\vv{X^*}_{(i)})=(\gamma_{i}(k-j))_{1\leq k,j\leq T}$ where $\gamma_{i}(\cdot)$ is the auto-covariance function of $(X_{i,t})$ with $\gamma_{i}(h)=\mathbb{E}(X_{i,t}^*X_{i,t+h})$, $h\in\mathbb{Z}$. By Lemma \ref{lem:3.4.1}, 
there exists a non-negative definite Hermitian circulant matrix $\vv{C}_{i,T}$ such that
\begin{align*}
\frac{1}{T} {\rm tr}\left((\vv{\Gamma}_{i,T}-\vv{C}_{i,T})(\vv{\Gamma}_{i,T}-\vv{C}_{i,T})^*\right)&\leq  2\left(\sum_{k=T}^\ift |\gamma_{i}(k)|+|\gamma_{i}(-k)|\right)\\
&+2\sum_{k=0}^{T}\frac{k}{T}(|\gamma_{i}(k)|^2+|\gm_{i}(-k)|^2)=o(1)
\end{align*}
and  $\vv{C}_{i,T}$ has eigenvalues $\{2\pi f_i(2\pi l/T):l=0,\ldots,T-1\}$, where $f_i$ is the spectral density function of the $i$-th coordinate process defined by
$$
f_i(\lambda)=\frac{1}{2\pi}\sum_{h=-\infty}^{\infty} e^{{\bf i}\lambda h}\gamma_i(h).
$$

On one hand, from (\ref{eq:3.4.13}), we know that the spectral density function of $(X_{i,t})$ now becomes
$\frac{1}{2\pi}|h(\vv{a}_i,\lambda)|$ so that eigenvalues of $\vv{C}_{i,T}$ becomes $\{|h(\vv{a}_i,2\pi l/T)|:l=0,\ldots,T-1\}$. On the other hand, by (\ref{eq:3.4.2}), it is easy to see from (\ref{eq:3.4.10}) that
\begin{equation}\label{eq:4.3.4.2}
\max_{1\leq i\leq p}\frac{1}{T} {\rm tr}\left((\vv{\Gamma}_{i,T}-\vv{C}_{i,T})(\vv{\Gamma}_{i,T}-\vv{C}_{i,T})^*\right)=o(1).
\end{equation}

Next, we verify the moment condition (4) in Assumption \ref{assm:1}. First, we show that there exists a positive constant $C$ such that  
\begin{equation}\label{eq:4.3.4.1}
\mathbb{E}\left[\vv{X^*}_{(i)}B\vv{X}_{(i)}-{\rm tr}(B\vv{\Gamma}_{i,T})\right]^2\leq C\|B\|_{op}^2\cdot T.
\end{equation}
To see this, let $B=(B_{st})$ be a $T\times T$ matrix bounded in norm. Consider the variance of quadratic form
\bg{align*}
&\E\(\sum_{s,t=1}^T B_{st}\(X_{i,s}X_{i,t}^*-\gm_i(s-t)\)\)^2\\
=&\sum_{s,t=1}^T\sum_{s',t'=1}^T B_{st}B_{s't'}\E\l[\(X_{i,s}X_{i,t}^*-\gm_i(s-t)\)\(X_{i,s'}X_{i,t'}^*-\gm_i(s'-t')\)\r]
\end{align*}
By the MA($\ift$) representation of stationary linear sequence, for $t\geq s$, it holds that
$$
X_{i,t}^*X_{i,s}-\gm_{i}(s-t)=\sum_{j=0}^\ift a_{i,j}a_{i,j+t-s}^*(|Z_{i,s-j}|^2-1)+\sum_{k\neq j+t-s} a_{i,k}a_{i,j} Z_{i,t-k}^*Z_{i,s-j}.
$$
Then, for $t\geq s$, $t'\geq s'$ and $s\geq s'$, one has
\bg{align*}
\E\((X_{i,t}^*X_{i,s}-\gm_{i}(s-t))(X_{i,t'}^*X_{i,s'}-\gm_{i}(s'-t'))\)&=(\E|Z_{11}|^4-1)\sum_{j=0}^\ift a_{i,j}a_{i,j+s-s'}^*a_{i,j+t'-s'}a_{i,j+t-s'}^*\\
&+\sum_{j=0}^\ift a_{i,j}a_{i,j+s-s'}^*a_{i,j+t'-s'}a_{i,j+t-s'}^*\\
&+\sum_{j=0}^\ift a_{i,j}a_{i,j+s-s'}a_{i,j+t'-s'}^*a_{i,j+t-s'}^*\\
&+\gm_{i}(t-t')\gm_{i}(s-s')+\gm_{i}(t-s')\gm_{i}(t'-s).
\end{align*}
Similar argument implies that
\bg{align*}
\E\((X_{i,t}^*X_{i,s}-\gm_i(s-t))(X_{i,t'}^*X_{i,s'}-\gm_i(s'-t'))\)&=(\E|Z_{11}|^4-1)\sum_{j=0}^\ift a_{i,j}a_{i,j+\mu_2-\mu_1}^*a_{i,j+\mu_3-\mu_1}a_{i,j+\mu_4-\mu_1}^*\\
&+\sum_{j=0}^\ift a_{i,j}a_{i,j+s-s'}^*a_{i,j+t'-s'}a_{i,j+t-s'}^*\\
&+\sum_{j=0}^\ift a_{i,j}a_{i,j+s-s'}a_{i,j+t'-s'}^*a_{i,j+t-s'}^*\\
&+\gm_{i}(t-t')\gm_{i}(s-s')+\gm_{i}(t-s')\gm_{i}(t'-s).
\end{align*}
where $\mu_t$ is $t$-th smallest value among $t,s,t',s'$. Therefore, for any matrix $B$ bounded in norm, one one hand,
\begin{align*}
&\sum_{s,t=1}^T\sum_{s',t'=1}^T B_{st}B_{s't'}(\gm_{i}(t-t')\gm_{i}(s-s')+\gm_{i}(t-s')\gm_{i}(t'-s))\\
=&{\rm tr}\(\vv{\Gm}_{i,T} B\vv{\Gm}_{i,T} (B+B')\)\leq 2\|B\|_{op}^2\|\vv{\Gm}_{i,T}\|^2_{op}\cdot T.
\end{align*}
According to Lemma 6 in Section 4.2 in \cite{G06}, if $\vv{\Gm}_{i,T}$ is a Hermitian Toeplitz matrix, then
$$
\|\vv{\Gm}_{i,T}\|_{op}\leq 2\sum_{k=-\ift}^\ift|\gm_{i}(k)|.
$$
Under our assumption $\sum_{k=0}^\ift |a_k|<\ift$, the right hand side must be convergent uniformly so that we are able to bound $\|\vv{\Gm}_{i,T}\|$ uniformly in norm. On the other hand,
\begin{align*}
&(\E|Z_{11}|^4-1)\sum_{s,t=1}^T\sum_{s',t'=1}^T B_{st}B_{s't'}\sum_{j=0}^\ift a_ja_{j+\mu_2-\mu_1}a_{j+\mu_3-\mu_1}a_{j+\mu_4-\mu_1}\\
\leq&|\E|Z_{11}|^4|\|B\|_{op}^2\(\sum_{j=0}^\ift |a_j|\)^4\cdot T
\end{align*}
Similarly, we have
\begin{align*}
\bigg|\sum_{j=0}^\ift a_{i,j}a_{i,j+s-s'}^*a_{i,j+t'-s'}a_{i,j+t-s'}^*\bigg|\leq&|\E|Z_{11}|^4|\|B\|_{op}^2\(\sum_{j=0}^\ift |a_j|\)^4\cdot T\\
\bigg|\sum_{j=0}^\ift a_{i,j}a_{i,j+s-s'}a_{i,j+t'-s'}^*a_{i,j+t-s'}^*\bigg|\leq&|\E|Z_{11}|^4|\|B\|_{op}^2\(\sum_{j=0}^\ift |a_j|\)^4\cdot T.
\end{align*}
Thus, (\ref{eq:4.3.4.1}) holds. Combining (\ref{eq:4.3.4.1}) and (\ref{eq:4.3.4.2}), with the Remark 2.2., we finish verifying the moment condition. 

Now, let $\vv{\Sigma}_{i}=\vv{C}_{i,T}$ and $f(\vv{b},\vv{a})=|h(\vv{b},\vv{a})|^2$ with $\vv{b}_{l,T}=2\pi (l-1)/T$ and $\vv{a}_{i,p}=\vv{a}_{i}$. By Theorem \ref{thm:0.1}, one has
\begin{equation}\label{eq:4.3.4.3}
\begin{split}
\tilde{m}(z)&=\int_0^1 \frac{1}{K(2\pi s,z)-z}{\rm d}s;\\
K(2\pi s,z)&=\int_0^1 \frac{|h(\vv{a},2\pi s)|^2}{1+c^{-1}\int_0^1 \frac{|h(\vv{a},2\pi s)|^2}{K(2\pi s,z)-z}{\rm d}s}{\rm d}G(\vv{a}).
\end{split}
\end{equation}

(b) From (\ref{eq:4.3.4.3}), it holds that
$$
K(2\pi s,z)-z=\int \fr{|h(\vv{a},2\pi s)|^2}{1+c^{-1}\int_0^{1}\fr{|h(\vv{a},2\pi s)|^2}{K(2\pi s,z)-z}{\rm d}s}{\rm d} G(\vv{a})-z
$$
so that
$$
1=c-c\int \fr{1}{1+c^{-1}\int_0^{1}\fr{|h(\vv{a},2\pi s)|^2}{K(2\pi s,z)-z}{\rm d}s}{\rm d} G(\vv{a})-z\int_0^{1}\frac{1}{K(2\pi s,z)-z}{\rm d}s.
$$
Thus,
$$
1-c+z\tilde{m}(z)=-c\int \fr{1}{1+c^{-1}\int_0^{1}\fr{|h(\vv{a},2\pi s)|^2}{K(2\pi s,z)-z}{\rm d}s}{\rm d} G(\vv{a}).
$$
Note that
$$
\int\fr{1}{1+c^{-1}\int_0^{1}\fr{|h(\vv{a},2\pi s)|^2}{K(2\pi s,z/c)-z/c}{\rm d}s}{\rm d} G(\vv{a})=\int\fr{1}{1+\int_0^{1}\fr{|h(\vv{a},2\pi s)|^2}{cK(2\pi s,z/c)-z}{\rm d}s}{\rm d} G(\vv{a})
$$

From (\ref{eq:3.4.8}),
\begin{align*}
m(z)&=\frac{1-c+c^{-1}z\tilde{m}(z/c)}{cz}\\
&=-\frac{1}{z}\int\fr{1}{1+\int_0^{1}\frac{|h(\vv{a},2\pi s)|^2}{cK(2\pi s,z/c)-z}{\rm d}s}{\rm d} G(\vv{a})\\
&=\int \frac{1}{-z+\int_0^1\frac{|h(\vv{a},2\pi s)|^2}{-cK(2\pi s,z/c)/z+1}{\rm d}s}{\rm d} G(\vv{a})
\end{align*}
Let $K_0(\lambda,z)=-\frac{1}{z}K(\lambda,z/c)$. Then, 
\begin{align*}
m(z)=\int \frac{1}{-z+\int_0^1\frac{|h(\vv{a},2\pi s)|^2}{cK_0(2\pi s,z)+1}{\rm d}s}{\rm d} G(\vv{a})
\end{align*}
Moreover, by (\ref{eq:4.3.4.3}),
\begin{align*}
K_0(2\pi s,z)=-\frac{1}{z}K(2\pi s,z/c)&=-\frac{1}{z}\int\frac{|h(\vv{a},2\pi s)|^2}{1+c^{-1}\int_0^{1}\frac{|h(\vv{a},2\pi s)|^2}{K(2\pi s,z/c)-z/c}{\rm d}s}{\rm d}G(\vv{a})\\
&=\int\frac{|h(\vv{a},2\pi s)|^2}{-z+\int_0^{1}\frac{|h(\vv{a},2\pi s)|^2}{cK_0(2\pi s,z)+1}{\rm d}s}{\rm d}G(\vv{a}).
\end{align*}
Finally, by changing of variable $\lambda=2\pi s$, we obtain (\ref{eq:3.4.14}) and (\ref{eq:3.4.15}). The proof is completed.

\subsection{Proof of Theorem \ref{thm:3.4.2}}

The proof is in general the same as Theorem \ref{thm:3.4.1}. We only need 
to verify that the auto-covariance function $\gamma_i(h)$'s are uniformly 
absolutely summable. By using the argument in Chapter 3 Section 3 {\bf 17} in \red{Stein and Rami Shakarchi (2007)}\cite{Stein07}, if the spectral 
density satisfies the Lipschitz condition, then for each $i=1,\ldots, p$the corresponding auto-covariance function $\gamma_i(\cdot)$ satisfies $\sum_{|h|>2^{m-1}}|\gamma_i(h)|^2\leq\frac{K^2\pi^2}{2^{2m}}$ and $\sum_{|h|>2^{m-1}}|\gamma_i(h)|\leq\frac{\sqrt{2} K\pi}{\sqrt{2}-1}2^{-\frac{m}{2}}$ so that $\gamma_i(\cdot)$ is 
absolutely summable. Note that equation in the right hand side is independent of $i$. The assumption (ii) ensures that
$$
\sum_{h=-\infty}^\infty |\gamma_{i}(h)|<\infty.
$$
and for $r=1,2$,
$$
\max_{1\leq i\leq p} \sum_{|h|>2^{m-1}}|\gamma_i(h)|^{r}=o(1).
$$
Thus, with the help of the Lipschitz condition, the approximation Lemma \ref{lem:3.4.1} still holds. Moreover, we still have that there exists a sequence of non-negative definite Hermitian circulant matrices $\vv{C}_{1,T},\ldots,\vv{C}_{p,T}$ such that
\begin{itemize}
\item[(i)] for any $i=1,\ldots,p$, $\vv{\Gamma}_{i,T}$ and $\vv{C}_{i,T}$ are asymptotically iso-spectral in the sense that
$$
\sup_{1\leq i\leq p} \frac{1}{p} {\rm tr}\left((\vv{\Gamma}_{i,T}-\vv{C}_{i,T})(\vv{\Gamma}_{i,T}-\vv{C}_{i,T})^*\right)=o(1)
$$
as $T\to\infty$ and $p/T=p(T)/T\to c\in(0,\infty)$;
\item[(ii)] for any $i=1,\ldots,p$, $\vv{C}_{i,T}$ has its eigenvalues $\{2\pi f_i(2\pi l/T):l=0,\ldots,T-1\}$, where $f_i$ is the spectral density function of the $i$-th coordinate process.
\end{itemize} 

Since the moment condition in Assumption \ref{assm:1} can be ensured by assumption (iii) in Theorem \ref{thm:3.4.2}, the rest of proof is totally the same as that for Theorem \ref{thm:3.4.1} and thus omitted.

\subsection{Proof of Theorem \ref{thm:3.5.2}}

First, we remove terms of drift process by using assumption (1) and the following lemma.
\begin{lem}[Zheng et al. (2011)\cite{Zh11} Lemma 1]
Suppose that for each $p$, $\vv{v}_l^{(p)}=(v_{l}^{(p,1)},\ldots,v_l^{(p,p)})'$ and $\vv{w}_l^{(p)}=(w_{l}^{(p,1)},\ldots,w_l^{(p,p)})'$, $l=1,\ldots,n$, are all $p$-dimensional vectors. Define
$$
\tilde{S}_n=\sum_{l=1}^n (\vv{v}_l^{(p)}+\vv{w}_l^{(p)})(\vv{v}_l^{(p)}+\vv{w}_l^{(p)})',~~~~~S_n=\sum_{l=1}^n \vv{w}_l^{(p)}\vv{w}_l^{(p)'}.
$$
If the following conditions are satisfied:
\begin{itemize}
\item[(i)] $n=n(p)$ with $\lim_{p\to\infty}p/n=y>0$;

\item[(ii)] there exists a sequence $\varepsilon=o(1/\sqrt{p})$ such that for all $p$ and all $l$, all the entries of $\vv{l}^{(p)}$ are bounded 
by $\varepsilon$ in absolute value;

\item[(iii)] $\limsup_{p\to\infty}{\rm tr}(S_n)/p<\infty$ almost surely.
\end{itemize}
Then, $L(F^{\tilde{S}_n},F^{S_n})\to0$ almost surely, where for any two probability distribution functions $F$ and $G$, $L(F,G)$ denotes the Levy distance between them.
\end{lem}
Thus, without loss of generality, we remove the drift process by assuming 
that $\vv{\mu}_t\equiv \vv{0}$. It remains to verify conditions in \ref{thm:3.5.1}. To see this, let
$$
G^{(p)}_l(t)=\int_0^t \gamma_{l,u}^{(p)2}{\rm d} u
$$
for $l=1,\ldots,p$. Then, $w_{i,l}^{n}=n[G_l^{(p)}(\tau_{i,n})-G_l^{(p)}(\tau_{i-1,n})]$ for $i=1,\ldots,n$ and $l=1,\ldots,p$. Moreover, $w_{i,l}^{(n)}$'s are uniformly bounded by our assumption and thus assumption (1) in Theorem \ref{thm:3.5.1} holds. 

Let $w(s,r)=\gamma(s,\Theta_r)^2v_r$. Define
$$
\tilde{w}_{i,l}^n= n\int_{\tau_{i-1,n}}^{\tau_{i,n}}\gamma(l/p,u)^2{\rm 
d}u.
$$
for $i=1,\ldots,n$ and $l=1,\ldots,p$. Then, we have
\begin{align*}
|w_{i,l}^n-\tilde{w}_{i,l}^n|&\leq n\left|\int_{\tau_{i-1,n}}^{\tau_{i,n}} \gamma^{(p)2}_{l,u}-\gamma(l/p,u)^2{\rm d} u\right|\\
&\leq 2\kappa_0 n\int_{\tau_{i-1,n}}^{\tau_{i,n}}|\gamma_{l,u}^{(p)}-\gamma(l/p,u)|{\rm d}u.
\end{align*}
It follows that
\begin{align*}
\frac{1}{np}\sum_{l=1}^p\sum_{i=1}^n |w_{i,l}^n-\tilde{w}_{i,l}^n|&\leq \frac{2\kappa_0}{p}\sum_{l=1}^p \int_0^1|\gamma_{l,u}^{(p)}-\gamma(l/p,u)|{\rm d}u\\
&\leq2\kappa_0\sum_{l=1}^p \int_{\frac{l-1}{p}}^{\frac{l}{p}}\int_0^1|\gamma_{l,u}^{(p)}-\gamma(s,u)|{\rm d}u{\rm d}s\\
&+2\kappa_0\sum_{l=1}^p \int_{\frac{l-1}{p}}^{\frac{l}{p}}\int_0^1|\gamma(s,u)-\gamma(l/p,u)|{\rm d}u{\rm d}s
\end{align*}
Thus, due to the continuity of $\gamma(s,r)$ and (\ref{eq:3.5.10}),
$$
\lim_{n,p\to\infty}\frac{1}{np}\sum_{l=1}^p\sum_{i=1}^n |w_{i,l}^n-\tilde{w}_{i,l}^n|=0.
$$
For $\tilde{w}_{i,l}^n$ defined above, informally, under assumption (3), we have by changing of variable that
$$
\tilde{w}_{i,l}^n=\gamma(l/p,\Theta_{\frac{i}{n}})^2v_\frac{i}{n}+o_{\rm a.s.}(1)
$$
since $\tau_{[ns],n}\to\Theta_s=\int_0^s v_r{\rm d}r$ almost surely as $n\to\infty$ so that $\tau_{i,n}-\tau_{i-1,n}=n^{-1}v_{i/n}+o_{\rm a.s.}(1)$. Mathematically, we have
\begin{align*}
|\tilde{w}_{i,l}^{n}-\gamma(l/p,\tau_{i/n})^2v_{\frac{i}{n}}|&\leq \left|n\int_{\tau_{i-1,n}}^{\tau_{i,n}}\gamma(l/p,u)^2-\gamma(l/p,\tau_{i/n})^2{\rm d}u+\gamma(l/p,\tau_{i/n})^2(n\Delta\tau_{i,n}-v_{i/n})\right|\\
&\leq 2\kappa_0\int_{\tau_{i-1,n}}^{\tau_{i,n}}|\gamma(l/p,u)-\gamma(l/p,\tau_{i,n})|{\rm d}u+2\kappa_0^2|n\Delta\tau_{i,n}-v_{i/n}|
\end{align*}
Similarly, by assumption (2) and (3), we find
$$
\lim_{n,p\to\infty}\frac{1}{np}\sum_{l=1}^p\sum_{i=1}^n |\tilde{w}_{i,l}^n-\gamma(l/p,\tau_{i,n})^2v_{\frac{i}{n}}|=0.
$$
Note that, with probability one,
$$
\lim_{n,p\to\infty}\sum_{l=1}^p\sum_{i=1}^n\gamma(l/p,\tau_{i,n})^2v_{\frac{i}{n}}I_{\left(\frac{l-1}{p},\frac{l}{p}\right]\times \left(\frac{i-1}{n},\frac{i}{n}\right]}(s,r)=\gamma(s,\Theta_r)^2v_{r},
$$
our conclusion then follows, which completes the proof.

\subsection{Proofs in Section \ref{ssec:mix}}
\label{ssec:mixproof}
\begin{proof}[\it Proof of Theorem \ref{thm:mixture}]
Without loss of generality, we assume that component means are centred, i.e., $\vv{\mu}_1=\cdots=\vv{\mu}_M=\vv{0}$. In fact, let $\vv{y}_{jn}=\vv{x}_{jn}-\vv{\mu}_{I_j}$, $j=1,\ldots,n$. Define
$$
\tilde{\vv{S}}_{n}:=\frac{1}{n}\sum_{j=1}^n \vv{y}_{jn}\vv{y}^*_{jn}.
$$
Note that ${\rm rank}(\vv{S}_n-\tilde{\vv{S}}_n)\leq 3M$. By the rank inequality (Theorem A.43 in \cite{BS10}), $\vv{S}_n$ and $\tilde{\vv{S}}_n$ have the same LSD.

Let $m_n(z)=\frac{1}{p}{\rm tr}(\vv{S}_n-z\vv{I})^{-1}$. It suffices to prove that conditioned on $\{I_{1n},\ldots,I_{nn}:n\geq1\}$, with probability one, $m_n(z)$ converges weakly to $m(z)$ obtained in the theorem. Indeed, let $A=\{m_n(z)\not\to m(z)~\hbox{for certain $z\in\mathbb{C}^+$}\}$. Then $\P(A|\{I_{1n},\ldots,I_{nn}:n\geq1\})=0$ implies $\P(A)=0$.

Conditioned on $\{I_{1n},\ldots,I_{nn}:n\geq1\}$, we have
$$
\vv{x}_{jn}=\vv{\Sigma}_{I_j,p}^{1/2}\vv{z}_{j},~~~~j=1,\ldots,n,
$$
can be viewed as $n$ independent samples from populations with population covariance matrices $\vv{\Sigma}_{I_j}$'s diagonalizable simultaneously. Let $\lambda_{i,l}^{(p)}$ be the $l$-th eigenvalue of $\vv{\Sigma}_{ip}$. For $i=1,\ldots,M$, we have $\lambda_{i,l}^{(p)}=H_{ip}^{-1}(l/p)$ for $1\leq l\leq p$. Meanwhile, it is well-known that $H_{ip}$ converges weakly to $H_i$ if and only if $H_{ip}^{-1}$ converges to $H_i^{-1}$ pointwisely on the set of continuous points of $H_i^{-1}$. By the uniformly boundedness of norms of $\vv{\Sigma}_{ip}$'s, $H_{ip}^{-1}$'s and $H^{-1}_i$'s are uniformly bounded as well. Consequently, for each $i=1,\ldots,M$,
\begin{equation}\label{eq:mix11}
\lim_{p\to\infty}\sum_{l=1}^p |H_{ip}^{-1}(l/p)-H_i^{-1}(l/p)|=0.
\end{equation}

Now, for $\vv{\Sigma}_{I_j,p}$, its eigenvalues are $\lambda_{I_j,l}^{(p)}=H_{I_j,p}^{-1}(l/p)$, which is indexed by $I_j$ and $l/p$. It is suffices for us to verify conditions in Theorem \ref{thm:0.2}. Note that by the law of large numbers, the empirical distribution $G_n$ of $\{I_{1n},\ldots,I_{nn}\}$ converges weakly to the population distribution $\vv{\eta}=(\eta_1,\ldots,\eta_M)$, that is,
$$
G_n(x)=\frac{1}{n}\sum_{i=1}^M I\{I_{in}\leq x\}\overset{\rm a.s.}{\to} \vv{\eta}.
$$
Meanwhile, by the uniform boundedness of $H_{ip}^{-1}$'s and $H_i^{-1}$'s, (\ref{eq:mix11}) and the Dominant Convergence Theorem,
$$
\lim_{n\to\infty}\frac{1}{np}\sum_{j=1}^n\sum_{l=1}^p|H_{I_j,p}^{-1}(l/p)-H_{I_j}^{-1}(l/p)|=0
$$ 
so (\ref{eq:2.10}) is ensured. The conclusion then follows by Theorem \ref{thm:0.2}.
\end{proof}

\begin{proof}[\it Proof of Example \ref{exm:2mix}]
Since $\vv{\Sigma}_1=\vv{I}$, we have $H_1^{-1}(s)=1$. Then, since $-zm(z)=\int_0^1\frac{1}{1+K(s,\lambda)}\d s$, we have
\begin{align*}
1+K(s,z)&=1+\frac{\eta_1}{-z-czm(z)}+\frac{\eta_2 H_2^{-1}(s)}{-z+c\int_0^1\frac{H_2^{-1}(s)}{1+K(s,z)}\d s}\\
&=\(1-\frac{\eta_1}{z+czm(z)}\)+\frac{\eta_2 H_2^{-1}(s)}{-z+c\int_0^1\frac{H_2^{-1}(s)}{1+K(s,z)}\d s}.
\end{align*}
It then follows that
\begin{align*}
1&=\(1-\frac{\eta_1}{z+czm(z)}\)(-zm(z))+\frac{\eta_2}{-z+c\int_0^1\frac{H_2^{-1}(s)}{1+K(s,z)}\d s}\(\int_0^1\frac{H_2^{-1}(s)}{1+K(s,z)}\d s\). 
\end{align*}
Consequently, we have
$$
1+zm(z)\(1-\frac{\eta_1}{z+c zm(z)}\)=\frac{\eta_2}{c}\(1+\frac{z}{-z+c\int_0^1\frac{H_2^{-1}(s)}{1+K(s,z)}\d s}\)
$$
So,
\begin{align*}
c+czm(z)\(1-\frac{\eta_1}{z+czm(z)}\)&=c+czm(z)-\eta_1\(1-\frac{z}{z+czm(z)}\)\\
&=\eta_2+\frac{\eta_2 z}{-z+c\int_0^1\frac{H_2^{-1}(s)}{1+K(s,z)}\d s}
\end{align*}
and
$$
\frac{\eta_2}{-z+c\int_0^1\frac{H_2^{-1}(s)}{1+K(s,z)}\d s}=\frac{1}{z}\l[c+czm(z)-\eta_1\(1-\frac{z}{z+czm(z)}\)-\eta_2\r]=:F(\vv{\eta},z,m(z)).
$$
For simplification, we have
\begin{align*}
F(\vv{\eta},z,m(z))&=\frac{1}{z}\(-1+c+czm(z)+\frac{\eta_1z}{z+czm(z)}\)\\
&=\underline{m}(z)+\frac{\eta_1}{z+czm(z)}
\end{align*}
Then, the kernel function becomes
$$
K(s,z)=-\frac{\eta_1}{z+c zm(z)}+F(\vv{\eta},z,m(z))H_2^{-1}(s)
$$
It then follows that
\begin{align*}
-zm(z)&=\int_0^1\frac{1}{1+K(s,z)}\d s\\
&=\int_0^1\frac{1}{1-\frac{\eta_1}{z+c zm(z)}+F(\vv{\eta},z,m(z))H_2^{-1}(s)}\d s\\
&=\int \frac{1}{1-\frac{\eta_1}{z+c zm(z)}+\left(\underline{m}(z)+\frac{\eta_1}{z+c zm(z)}\right)\lambda}\d H_2(\lambda),
\end{align*}
which completes the proof. 
\end{proof}

%

%%%%%%%%%%%%%%%%%%%%%%%%%%%%%%%%%%%%%%%%%%%%%%%%%%%%%%%%%%%%%%
%\begin{acks}[Acknowledgments]
%The authors would like to thank the anonymous referees, an Associate
%Editor and the Editor for their constructive comments that improved the
%quality of this paper.
%\end{acks}
%
%%%%%%%%%%%%%%%%%%%%%%%%%%%%%%%%%%%%%%%%%%%%%%%
%%% Funding information, if any,             %%
%%% should be provided in the                %%
%%% funding section.                         %%
%%%%%%%%%%%%%%%%%%%%%%%%%%%%%%%%%%%%%%%%%%%%%%%
%\begin{funding}
%The first author was supported by NSF Grant DMS-??-??????.
%
%The second author was supported in part by NIH Grant ???????????.
%\end{funding}
%
%%%%%%%%%%%%%%%%%%%%%%%%%%%%%%%%%%%%%%%%%%%%%%%
%%% Supplementary Material, if any, should   %%
%%% be provided in {supplement} environment  %%
%%% with title and short description.        %%
%%%%%%%%%%%%%%%%%%%%%%%%%%%%%%%%%%%%%%%%%%%%%%%
%\begin{supplement}
%\stitle{Title of Supplement A}
%\sdescription{Short description of Supplement A.}
%\end{supplement}
%\begin{supplement}
%\stitle{Title of Supplement B}
%\sdescription{Short description of Supplement B.}
%\end{supplement}
%
\bibliographystyle{imsart-number}
\bibliography{test}

\begin{thebibliography}{28}
% BibTex style file: imsart-number.bst, 2017-11-03
% Default style options (sort=1,type=number).
% Used options (sort=1,type=number).

\bibitem{And01}
\begin{barticle}[author]
\bauthor{\bsnm{Andersen},~\bfnm{T.~G.}\binits{T.~G.}},
  \bauthor{\bsnm{Bollerslev},~\bfnm{T.}\binits{T.}},
  \bauthor{\bsnm{Diebold},~\bfnm{F.~X.}\binits{F.~X.}} \AND
  \bauthor{\bsnm{Labys},~\bfnm{P.}\binits{P.}}
(\byear{2001}).
\btitle{The distribution of realized exchange rate volatility}.
\bjournal{Journal of the American Statistical Association}
\bvolume{96}
\bpages{42-55}.
\bdoi{10.1198/016214501750332965}
\end{barticle}
\endbibitem

\bibitem{BS10}
\begin{bbook}[author]
\bauthor{\bsnm{Bai},~\bfnm{Zhidong}\binits{Z.}} \AND
  \bauthor{\bsnm{Silverstein},~\bfnm{Jack~W}\binits{J.~W.}}
(\byear{2010}).
\btitle{Spectral analysis of large dimensional random matrices}
\bvolume{20},
\bedition{2} ed.
\bpublisher{Springer}.
\end{bbook}
\endbibitem

\bibitem{BZ}
\begin{barticle}[author]
\bauthor{\bsnm{Bai},~\bfnm{Z.}\binits{Z.}} \AND
  \bauthor{\bsnm{Zhou},~\bfnm{W.}\binits{W.}}
(\byear{2008}).
\btitle{Large sample covariance matrices without independence structures in
  columns}.
\bjournal{Statistica Sinica}
\bvolume{18}
\bpages{425-442}.
\end{barticle}
\endbibitem

\bibitem{Ch06}
\begin{barticle}[author]
\bauthor{\bsnm{Chatterjee},~\bfnm{S.}\binits{S.}}
(\byear{2006}).
\btitle{A generalization of the lindeberg principle}.
\bjournal{Annals of Probability}
\bvolume{34}
\bpages{2061-2076}.
\bdoi{10.1214/009117906000000575}
\end{barticle}
\endbibitem

\bibitem{Chen20}
\begin{barticle}[author]
\bauthor{\bsnm{Chen},~\bfnm{E.~Y.}\binits{E.~Y.}},
  \bauthor{\bsnm{Tsay},~\bfnm{R.~S.}\binits{R.~S.}} \AND
  \bauthor{\bsnm{Chen},~\bfnm{R.}\binits{R.}}
(\byear{2020}).
\btitle{Constrained Factor Models for High-Dimensional Matrix-Variate Time
  Series}.
\bjournal{Journal of the American Statistical Association}
\bvolume{115}
\bpages{775-793}.
\bdoi{10.1080/01621459.2019.1584899}
\end{barticle}
\endbibitem

\bibitem{Chen21}
\begin{barticle}[author]
\bauthor{\bsnm{Chen},~\bfnm{R.}\binits{R.}},
  \bauthor{\bsnm{Xiao},~\bfnm{H.}\binits{H.}} \AND
  \bauthor{\bsnm{Yang},~\bfnm{D.}\binits{D.}}
(\byear{2021}).
\btitle{Autoregressive models for matrix-valued time series}.
\bjournal{Journal of Econometrics}
\bvolume{222}
\bpages{539-560}.
\bdoi{10.1016/j.jeconom.2020.07.015}
\end{barticle}
\endbibitem

\bibitem{FS06}
\begin{bbook}[author]
\bauthor{\bsnm{Fruhwirth-Schnatter},~\bfnm{S.}\binits{S.}}
(\byear{2006}).
\btitle{Finite Mixture and Markov Switching Models}.
\bpublisher{New York: Springer}.
\end{bbook}
\endbibitem

\bibitem{G06}
\begin{barticle}[author]
\bauthor{\bsnm{Gray},~\bfnm{R.~M.}\binits{R.~M.}}
(\byear{2006}).
\btitle{Toeplitz and circulant matrices: A review}.
\bjournal{Foundations and Trends in Communications and Information Theory}
\bvolume{2}
\bpages{155-239}.
\bdoi{10.1561/0100000006}
\end{barticle}
\endbibitem

\bibitem{HLN}
\begin{barticle}[author]
\bauthor{\bsnm{Hachem},~\bfnm{W.}\binits{W.}},
  \bauthor{\bsnm{Loubaton},~\bfnm{P.}\binits{P.}} \AND
  \bauthor{\bsnm{Najim},~\bfnm{J.}\binits{J.}}
(\byear{2006}).
\btitle{The empirical distribution of the eigenvalues of a Gram matrix with a
  given variance profile}.
\bjournal{Annales de l'institut Henri Poincare (B) Probability and Statistics}
\bvolume{42}
\bpages{649-670}.
\bdoi{10.1016/j.anihpb.2005.10.001}
\end{barticle}
\endbibitem

\bibitem{JP98}
\begin{barticle}[author]
\bauthor{\bsnm{Jacod},~\bfnm{Jean}\binits{J.}} \AND
  \bauthor{\bsnm{Protter},~\bfnm{Philip}\binits{P.}}
(\byear{1998}).
\btitle{Asymptotic error distributions for the Euler method for stochastic
  differential equations}.
\bjournal{Annals of Probability}
\bvolume{26}
\bpages{267--307}.
\end{barticle}
\endbibitem

\bibitem{Jin14}
\begin{barticle}[author]
\bauthor{\bsnm{Jin},~\bfnm{B.}\binits{B.}},
  \bauthor{\bsnm{Wang},~\bfnm{C.}\binits{C.}},
  \bauthor{\bsnm{Bai},~\bfnm{Z.~D.}\binits{Z.~D.}},
  \bauthor{\bsnm{Nair},~\bfnm{K.}\binits{K.}} \AND
  \bauthor{\bsnm{Harding},~\bfnm{M.}\binits{M.}}
(\byear{2014}).
\btitle{Limiting spectral distribution of a symmetrized auto-cross covariance
  matrix}.
\bjournal{Annals of Applied Probability}
\bvolume{24}
\bpages{1199-1225}.
\bdoi{10.1214/13-AAP945}
\end{barticle}
\endbibitem

\bibitem{Jin09}
\begin{barticle}[author]
\bauthor{\bsnm{Jin},~\bfnm{B.}\binits{B.}},
  \bauthor{\bsnm{Wang},~\bfnm{C.}\binits{C.}},
  \bauthor{\bsnm{Miao},~\bfnm{B.}\binits{B.}} \AND
  \bauthor{\bsnm{Lo~Huang},~\bfnm{M.~N.}\binits{M.~N.}}
(\byear{2009}).
\btitle{Limiting spectral distribution of large-dimensional sample covariance
  matrices generated by VARMA}.
\bjournal{Journal of Multivariate Analysis}
\bvolume{100}
\bpages{2112-2125}.
\bdoi{10.1016/j.jmva.2009.06.011}
\end{barticle}
\endbibitem

\bibitem{Li18}
\begin{barticle}[author]
\bauthor{\bsnm{Li},~\bfnm{W.}\binits{W.}} \AND
  \bauthor{\bsnm{Yao},~\bfnm{J.}\binits{J.}}
(\byear{2018}).
\btitle{On structure testing for component covariance matrices of a high
  dimensional mixture}.
\bjournal{Journal of the Royal Statistical Society. Series B: Statistical
  Methodology}
\bvolume{80}
\bpages{293-318}.
\bdoi{10.1111/rssb.12248}
\end{barticle}
\endbibitem

\bibitem{LP17}
\begin{barticle}[author]
\bauthor{\bsnm{Liu},~\bfnm{H.}\binits{H.}},
  \bauthor{\bsnm{Aue},~\bfnm{A.}\binits{A.}} \AND
  \bauthor{\bsnm{Paul},~\bfnm{D.}\binits{D.}}
(\byear{2015}).
\btitle{On the Mar{\v c}enko-Pastur law for linear time series}.
\bjournal{Annals of Statistics}
\bvolume{43}
\bpages{675-712}.
\bdoi{10.1214/14-AOS1294}
\end{barticle}
\endbibitem

\bibitem{MP}
\begin{barticle}[author]
\bauthor{\bsnm{Mar{\v{c}}enko},~\bfnm{Vladimir~A}\binits{V.~A.}} \AND
  \bauthor{\bsnm{Pastur},~\bfnm{Leonid~Andreevich}\binits{L.~A.}}
(\byear{1967}).
\btitle{Distribution of eigenvalues for some sets of random matrices}.
\bjournal{Mathematics of the USSR-Sbornik}
\bvolume{1}
\bpages{457--483}.
\end{barticle}
\endbibitem

\bibitem{McD}
\begin{barticle}[author]
\bauthor{\bsnm{McDiarmid},~\bfnm{Colin}\binits{C.}}
(\byear{1989}).
\btitle{On the method of bounded differences}.
\bjournal{Surveys in combinatorics}
\bvolume{141}
\bpages{148--188}.
\end{barticle}
\endbibitem

\bibitem{McP00}
\begin{bbook}[author]
\bauthor{\bsnm{McLachlan},~\bfnm{G.~J.}\binits{G.~J.}} \AND
  \bauthor{\bsnm{Peel},~\bfnm{D.}\binits{D.}}
(\byear{2000}).
\btitle{Finite Mixture Models}.
\bpublisher{New York: Wiley.}
\end{bbook}
\endbibitem

\bibitem{PajorPastur}
\begin{barticle}[author]
\bauthor{\bsnm{Pajor},~\bfnm{A.}\binits{A.}} \AND
  \bauthor{\bsnm{Pastur},~\bfnm{L.}\binits{L.}}
(\byear{2009}).
\btitle{On the limiting empirical measure of eigenvalues of the sum of rank one
  matrices with log-concave distribution}.
\bjournal{Studia Mathematica}
\bvolume{195}
\bpages{11-29}.
\bdoi{10.4064/sm195-1-2}
\end{barticle}
\endbibitem

\bibitem{PA14}
\begin{barticle}[author]
\bauthor{\bsnm{Paul},~\bfnm{D.}\binits{D.}} \AND
  \bauthor{\bsnm{Aue},~\bfnm{A.}\binits{A.}}
(\byear{2014}).
\btitle{Random matrix theory in statistics: A review}.
\bjournal{Journal of Statistical Planning and Inference}
\bvolume{150}
\bpages{1-29}.
\bdoi{10.1016/j.jspi.2013.09.005}
\end{barticle}
\endbibitem

\bibitem{BS94}
\begin{barticle}[author]
\bauthor{\bsnm{Silverstein},~\bfnm{J.~W.}\binits{J.~W.}}
(\byear{1995}).
\btitle{Strong convergence of the empirical distribution of eigenvalues of
  large dimensional random matrices}.
\bjournal{Journal of Multivariate Analysis}
\bvolume{55}
\bpages{331-339}.
\bdoi{10.1006/jmva.1995.1083}
\end{barticle}
\endbibitem

\bibitem{BS95}
\begin{barticle}[author]
\bauthor{\bsnm{Silverstein},~\bfnm{J.~W.}\binits{J.~W.}} \AND
  \bauthor{\bsnm{Bai},~\bfnm{Z.~D.}\binits{Z.~D.}}
(\byear{1995}).
\btitle{On the empirical distribution of eigenvalues of a class of large
  dimensional random matrices}.
\bjournal{Journal of Multivariate Analysis}
\bvolume{54}
\bpages{175-192}.
\bdoi{10.1006/jmva.1995.1051}
\end{barticle}
\endbibitem

\bibitem{Stein07}
\begin{bbook}[author]
\bauthor{\bsnm{Stein},~\bfnm{Elias~M}\binits{E.~M.}} \AND
  \bauthor{\bsnm{Shakarchi},~\bfnm{Rami}\binits{R.}}
(\byear{2011}).
\btitle{Fourier analysis: an introduction}
\bvolume{1}.
\bpublisher{Princeton University Press}.
\end{bbook}
\endbibitem

\bibitem{Chen19}
\begin{barticle}[author]
\bauthor{\bsnm{Wang},~\bfnm{D.}\binits{D.}},
  \bauthor{\bsnm{Liu},~\bfnm{X.}\binits{X.}} \AND
  \bauthor{\bsnm{Chen},~\bfnm{R.}\binits{R.}}
(\byear{2019}).
\btitle{Factor models for matrix-valued high-dimensional time series}.
\bjournal{Journal of Econometrics}
\bvolume{208}
\bpages{231-248}.
\bdoi{10.1016/j.jeconom.2018.09.013}
\end{barticle}
\endbibitem

\bibitem{WP14}
\begin{barticle}[author]
\bauthor{\bsnm{Wang},~\bfnm{L.}\binits{L.}} \AND
  \bauthor{\bsnm{Paul},~\bfnm{D.}\binits{D.}}
(\byear{2014}).
\btitle{Limiting spectral distribution of renormalized separable sample
  covariance matrices when $p/n\to0$}.
\bjournal{Journal of Multivariate Analysis}
\bvolume{126}
\bpages{25--52}.
\bdoi{10.1016/j.jmva.2013.12.015}
\end{barticle}
\endbibitem

\bibitem{Yao12}
\begin{barticle}[author]
\bauthor{\bsnm{Yao},~\bfnm{J.}\binits{J.}}
(\byear{2012}).
\btitle{A note on a Mar{\v c}enko-Pastur type theorem for time series}.
\bjournal{Statistics and Probability Letters}
\bvolume{82}
\bpages{22-28}.
\bdoi{10.1016/j.spl.2011.08.011}
\end{barticle}
\endbibitem

\bibitem{Yao15}
\begin{bbook}[author]
\bauthor{\bsnm{Yao},~\bfnm{J.}\binits{J.}},
  \bauthor{\bsnm{Zheng},~\bfnm{S.}\binits{S.}} \AND
  \bauthor{\bsnm{Bai},~\bfnm{Z.}\binits{Z.}}
(\byear{2015}).
\btitle{Large sample covariance matrices and high-dimensional data analysis}.
\bpublisher{Cambridge University Press}.
\bdoi{10.1017/CBO9781107588080}
\end{bbook}
\endbibitem

\bibitem{ZL06}
\begin{bphdthesis}[author]
\bauthor{\bsnm{Zhang},~\bfnm{Lixin}\binits{L.}}
(\byear{2006}).
\btitle{Spectral analysis of large dimensional random matrices},
\btype{PhD thesis},
\bpublisher{National University of Singapore}.
\end{bphdthesis}
\endbibitem

\bibitem{Zh11}
\begin{barticle}[author]
\bauthor{\bsnm{Zheng},~\bfnm{X.}\binits{X.}} \AND
  \bauthor{\bsnm{Li},~\bfnm{Y.}\binits{Y.}}
(\byear{2011}).
\btitle{On the estimation of integrated covariance matrices of high dimensional
  diffusion processes}.
\bjournal{Annals of Statistics}
\bvolume{39}
\bpages{3121-3151}.
\bdoi{10.1214/11-AOS939}
\end{barticle}
\endbibitem

\end{thebibliography}

%%%%%%%%%%%%%%%%%%%%%%%%%%%%%%%%%%%%%%%%%%%%%%%%%%%%%%%%%%%%%%
%

%
\end{document}